\documentclass[11pt]{amsart}
\usepackage{calc,amssymb,amsthm,amsmath}
\usepackage{mathtools}
\RequirePackage[dvipsnames,usenames]{xcolor}
\usepackage{hyperref}
\hypersetup{
bookmarks,
bookmarksdepth=3,
bookmarksopen,
bookmarksnumbered,
pdfstartview=FitH,
colorlinks,backref,hyperindex,
linkcolor=Sepia,
anchorcolor=BurntOrange,
citecolor=MidnightBlue,
citecolor=OliveGreen,
filecolor=BlueViolet,
menucolor=Yellow,
urlcolor=OliveGreen
}
\usepackage{alltt}
\usepackage{multicol}
\usepackage{etex}
\usepackage{xspace}
\newcommand{\ie}{{\itshape ie} }
\usepackage{mabliautoref}
\usepackage{colonequals}
\input{kmacros3.sty}
\usepackage{stmaryrd}
\usepackage{tikz}
\usetikzlibrary{matrix,arrows}

\usepackage[left=1.11in,top=1.0in,right=1.11in,bottom=1.0in]{geometry}

\usepackage{verbatim}
\usepackage{enumerate}

\theoremstyle{theorem}

\theoremstyle{remark}
\newtheorem*{*setting}{Setting}

\theoremstyle{remark}

\usepackage{graphicx}

\usepackage[all,cmtip]{xy}

\renewcommand{\O}{\mathcal{O}}
\renewcommand{\ie}{{\itshape i.e.}}
\renewcommand{\cf}{{\itshape cf.\xspace\xspace}}

\DeclareMathOperator{\Sol}{{Sol}}
\newcommand{\RHom}{\myR{\hskip -1.5pt}\sHom}

\renewcommand{\phi}{\varphi}

\renewcommand{\theta}{\vartheta}

\renewcommand{\epsilon}{\varepsilon}

\renewcommand{\to}[1][]{\xrightarrow{\ #1\ }}
\newcommand{\onto}[1][]{\protect{\xrightarrow{\ #1\ }\hspace{-0.8em}\rightarrow}}
\newcommand{\into}[1][]{\lhook \joinrel \xrightarrow{\ #1\ }}

\newcommand{\nil}{\mathrm{nil}}

\renewcommand{\theenumi}{\alph{enumi}}
\begin{document}

\title [$p^{-1}$-linear maps in algebra and geometry]{$p^{-1}$-linear maps in algebra and geometry}
\author{Manuel Blickle}

\address{ Institut f\"ur Mathematik\\ Johannes Gutenberg-Universit\"at Mainz\\55099 Mainz, Germany}
\email{blicklem@uni-mainz.de}

\author{Karl Schwede}
\address{Department of Mathematics\\ The Pennsylvania State University\\ University Park, PA, 16802, USA}
\email{schwede@math.psu.edu}

\thanks{The first author was partially supported by a Heisenberg Fellowship and the SFB/TRR45}
\thanks{The second author was partially supported by the NSF grant DMS \#1064485}

\maketitle

\begin{abstract}
In this article we survey the basic properties of $p^{-e}$-linear endomorphisms of coherent $\O_X$-modules, \ie~of $\O_X$-linear maps $F_* \sF \to \sG$ where $\sF,\sG$ are $\O_X$-modules and $F$ is the Frobenius of a variety of finite type over a perfect field of characteristic $p > 0$. We emphasize their relevance to commutative algebra, local cohomology and the theory of test ideals on the one hand, and global geometric applications to vanishing theorems and lifting of sections on the other.
\end{abstract}

\setcounter{tocdepth}{1}
\tableofcontents

\section{Introduction}
%\numberwithin{subsection}{theorem}

%In the study of rings and schemes over a field of positive characteristic $p$ the Frobenius morphism (\ie~the $p$th power map) is a rich source of additional structure often providing a substitute for powerful techniques commonly used in complex geometry, such as Hodge theory or Kodaira vanishing, for example.

In this survey we study the basic properties of $p^{-1}$-linear morphisms between coherent sheaves on a scheme $X$ over a perfect field of positive characteristic $p$. If $F \colon X \to X$ is the Frobenius morphism (\ie~the $p$th power map on the structure sheaf) we denote by $F_*$ the restriction functor along $F$ (\cf~\autoref{sec.frobpush}). A $p^{-1}$-linear map is then an $\O_X$-linear map $\phi : F_* \sF \to \sG$ for two $\O_X$-modules $\sF$ and $\sG$. The name stems from the fact that if we view $\phi$ as a map on the underlying sheaves of Abelian groups, $\phi$ satisfies the condition $\phi(r^p f)=r \phi(f)$ for local sections $r \in \O_X$ and $f \in \sF$. In particular, if $r$ has a $p^\mathrm{th}$ root, then we may write this relation as $\phi(rf)=r^{p^{-1}}\phi(f)$.

As an example for a $p^{-1}$-linear map, we start with a splitting of the Frobenius map, that is an $\O_X$-linear map $\phi \colon F_*\O_X \to \O_X$ such that the composition
\[
    \O_X \to[F] F_*\O_X \to[\phi] \O_X
\]
is equal to the identity. The mere existence of such a $\phi$ has strong implications for the local geometry of $X$ (it is reduced, for example). Furthermore, it immediately implies a highly effective version of Serre vanishing: the higher cohomology of \emph{any} ample line bundle vanishes. In the light of such strong implications, it is somewhat surprising that there are varieties of interest that are Frobenius split. For example, regular affine varieties, projective spaces, normal toric varieties, and most prominently flag- and Schubert varieties are Frobenius split. And it was precisely for the latter varieties where the above vanishing yields a simple proof of Kempf's vanishing theorem \cite{MehtaRamanathanFrobeniusSplittingAndCohomologyVanishing}, see also \cite{HaboushAShortProofOfKempf}.  Frobenius split varieties have been extensively studied \cite{BrionKumarFrobeniusSplitting} and in \autoref{sec.FSplittings} we give a detailed account of their theory, explaining some of the more delicate vanishing and extension results, and discussing criteria to decide if a given variety is Frobenius split.

In \autoref{sec.NonSplittings} we show how some of the results and techniques for Frobenius splittings can be extended to more general contexts (where the variety is not $F$-split) to derive similar conclusions (vanishing and extension results).
For example, a systematic use of certain $p^{-1}$-linear maps can replace Kodaira and Kawamata-Viehweg vanishing theorems \cite{KawamataVanishing,ViehwegVanishingTheorems} in some applications, see \autoref{subsec.GlobalConsiderationsForNonSplittings}. These techniques rely on an explicit connection between $p^{-e}$-linear maps $\phi \in \Hom_{\O_X}(F^e_*\sL,\O_X)$ and $\bQ$-divisors $\Delta$ such that $\O_X((p^e-1)(K_X + \Delta)) \cong \sL^{-1}$ which is explained in detail in \autoref{sec.ConnectionsWithDivisors}. Indeed, this correspondence between $p^{-e}$-maps and $\mathbb{Q}$-divisors pervades much of the paper.  This correspondence also provides us with valuable geometric intuition in working with $p^{-e}$-linear maps.

In \autoref{sec.ChangeOfVariety} we state a number of general results on the behavior of $p^{-e}$-linear maps under certain functorial operations, such as pullback along closed immersions, localization, pushforward along a birational map, and finally pullback along a finite map. In all these cases, viewing $p^{-e}$-linear maps as $\bQ$-divisors and performing operations on divisors is the guiding principle.

A second key example of a $p^{-1}$-linear map is the classical Cartier operator $C \colon F_*\omega_X \to \omega_X$ introduced in \cite{CartierUneNouvelle}. There are various guises in which this operator on the dualizing sheaf appears, but most generally one may view it as the \emph{trace of Frobenius} under the duality for finite morphisms, see \autoref{sec.CartierOperator}. The Cartier operator has been extensively studied in connection to residues of differentials in positive characteristic, and plays a crucial role in Deligne and Illusie's \cite{DeligneIllusie_1987} algebraic proof of Kodaira vanishing.

In the final two Sections \ref{sec.CartierModules} and \ref{sec.CartierModLocalCohom} we describe the category of Cartier modules introduced in \cite{BlickleBoeckleCartierModulesFiniteness}. This category consists of coherent $\O_X$-modules $\sF$ equipped with a $p^{-e}$-linear endomorphism, \ie~a $\O_X$-linear map $F^e_*\sF \to \sF$. We show that the Abelian category of Cartier modules satisfies some remarkable properties. Most importantly, Cartier modules have finite length up to nilpotence\footnote{A coherent Cartier module $\sF$ is nilpotent is some power of the structural map is zero.}. Furthermore, Cartier modules are related to a number of other categories which have been extensively used in the study of local cohomology in positive characteristic. Hence the finiteness results about Cartier modules imply and generalize previous finiteness results about local cohomology, see \autoref{sec.LocalProperties} where we indicate how results of Hartshorne-Speiser \cite{HartshorneSpeiserLocalCohomologyInCharacteristicP}, Lyubeznik \cite{LyubeznikFModulesApplicationsToLocalCohomology} and Enescu and Hochster \cite{EnescuHochsterTheFrobeniusStructureOfLocalCohomology} can be derived easily.

In the final section we explain a certain degree-reducing property of $p^{e}$-linear maps and show how this property yields a completely elementary approach to the above mentioned finiteness result. In the last subsection we finally close the gap to the theory of tight closure \cite{HochsterHunekeTC1,HunekeTightClosureBook,HochsterFoundations}, which im- and explicitly heavily relied on $p^{-e}$-linear maps since its beginnings, in showing how the test ideals of Hara and Yoshida \cite{HaraYoshidaGeneralizationOfTightClosure} are obtained from certain generalizations of Cartier modules. We include as another demonstration of the utility of this viewpoint a quick proof of the discreteness of the jumping numbers of the test ideal.

The target audience for this article is a researcher or student who is familiar with commutative algebra and algebraic geometry and who wishes to learn how to use $p^{-1}$-linear maps in a wide variety of contexts.  We do not assume the reader has one particular background (\ie{} representation theory/Frobenius splitting, tight closure theory, $\mathcal{D}$-modules, or higher dimensional complex algebraic geometry).  Because we view this article as a place where material can be learned, at the end of each section there are many many exercises.  The more difficult exercises are decorated with a *.  The exercises are a fundamental part of this document.

\subsection*{Acknowledgements}
The authors are deeply indebted to Alberto Fernandez Boix, Lance Miller, Claudiu Raicu, Kevin Tucker, Wenliang Zhang and the referee for innumerable valuable comments on previous drafts of this paper.

\section{Preliminaries on Frobenius}\label{sec.prelims}

In this section we introduce our conventions on notation -- in particular with regards to the Frobenius morphism.

\subsection{Prerequisites and Notation}
We assume that the reader is familiar with the basics of commutative algebra and algebraic geometry, all of which is covered in the standard reference works \cite{Hartshorne} and \cite{MatsumuraCommutativeAlgebra}.  Beyond this, a familiarity with Grothendieck duality, \cite{HartshorneResidues,ConradGDualityAndBaseChange}, will be particularly helpful. Explicitly, Serre vanishing, canonical modules, dualizing and Serre duality, and the connection between divisors and line bundles, will appear frequently, see also \cite{BrunsHerzog}. The notion of $\bQ$-divisors will be used extensively (see \cite{KollarMori} or \cite{LazarsfeldPositivity1,LazarsfeldPositivity2}). The process of reflexification of sheaves on normal varieties and its relation to Weil divisors will be recalled in \autoref{sec.reflex} for the convenience of the reader, also see \cite{HartshorneGeneralizedDivisorsOnGorensteinSchemes} where the same theory is worked out in substantially greater generality.

Throughout this paper all rings and schemes are assumed to be of finite type over a perfect field $k$ of characteristic $p > 0$, or they are a localization or completion of such at a prime.  This implies that our schemes are excellent and possess canonical modules and dualizing complexes \cite{MatsumuraCommutativeAlgebra,HartshorneResidues}.  We further assume that all schemes are separated.

\subsection{Frobenius and push-forward}
\label{sec.frobpush}

We begin by reviewing the most basic notation (since it varies wildly in the literature).

The key structure in algebra and geometry over a field of positive characteristic $p > 0$ is the \emph{(absolute) Frobenius} endomorphism. For a ring $R$ this is just the $p$th power ring endomorphism
\[
    F=F_R \colon R \to R
\]
given by sending $r \in R$ to $r^p$. %Its importance stems from the simple fact that in characteristic $p$ it is a ring homomorphism which holds since the binomial coefficients $\binom{p}{a}$ are divisible by $p$ for $1 \leq a \leq p-1$.

%\subsubsection{Frobenius on schemes}
Since the Frobenius is canonical it induces a morphism for any scheme $X$ over a field $k$ of characteristic $p > 0$, also called the Frobenius endomorphism and also denoted by
\[
    F=F_X \colon X \to X.
\]
Supposing that $k$ is perfect and $X$ is a $k$-variety (or a scheme according to our convention) then $F_X$ is a finite map\footnote{An abstract scheme with a finite Frobenius is called \emph{$F$-finite}.} by \autoref{ex.VarietiesAreFFinite}. Note that $F_X$ is in general not a morphism of $k$-schemes -- however this point can be rectified by changing the $k$-structure on the first copy of $X$, if desired. We denote by $F^e$ the $e$-fold self composition of Frobenius.

Even in the affine situation $X=\Spec R$ we use geometric notation and denote the Frobenius on $R$ by $F \colon R \to F_*R$ to remind us that it is not $R$-linear. This has the added benefit that we now can distinguish the source and target of $F = F_R$.

Given an ideal $I = \langle f_1, \dots, f_m \rangle \subseteq R$, we define its \emph{$p^e$th Frobenius power} to be $I^{[p^e]} = \langle f_1^{p^e}, \dots, f_m^{p^e} \rangle$.  This is independent of the choice of generators $f_j$, see \autoref{ex.FrobeniusPowerChoiceOfGenerators}.  The formation of $I^{[p^e]}$ commutes with localization and so for any ideal sheaf $\sI \subseteq \O_X$, we can define $\sI^{[p^e]}$ in the obvious way.

%\subsubsection{Frobenius pushforward for modules}

Note $F^e_* \O_X$ is isomorphic to $\O_X$ as a sheaf of rings -- but as $\O_X$-modules they are distinct: namely, $\O_X$ acts on $F^e_* \O_X$ via $p^e$th powers. More generally, for any $\O_X$-module $\sM$, one observes that $F^e_* \sM$ is isomorphic to $\sM$ as a sheaf of Abelian groups but the $\O_X$-module structure is given by $r.m = r^{p^e}m$ for a local section $r \in \O_X$ and $m \in F^e_* \sM$.  Of course, $F^e_* \sM$ also has an $F^e_* \O_X$-module structure, which coincides with $\sM$'s original $\O_X$-module structure. We also use the notation $F^e_* M$ in the affine case $X = \Spec R$ to denote an $R$-module with the twisted (restriction of scalars) Frobenius structure.

One immediately verifies that $F_* \tld{M}$ coincides with $\tld{F_* M}$ as $\O_X$-modules, where $\tld{M}$ denotes the $\O_X$-module associated to the $R$-module $M$.  However, we caution the reader that the same identification does not hold in the graded case with respect to $\Proj$, see for example \cite[Lemma 5.6]{SchwedeSmithLogFanoVsGloballyFRegular} and \autoref{ex.FPushforwardAndProj}.

\begin{notation}
Given an element $m \in M$, we will sometimes use $F^e_* m$ to denote the corresponding element of $F^e_* M$.  Likewise for sheaves of $\O_X$-modules $\sM$ on $X$.
\end{notation}

\subsection{Frobenius pull-back and the projection formula}
\label{sec.ProjectionFormula}
Let $X$ be a scheme over a perfect field $k$ of characteristic $p>0$ and let $\sF$ be a coherent sheaf and $\sL$ a line bundle on $X$. Since the Frobenius is an isomorphism on the underlying topological space, the pullback $F^{e*} \sF$ (as an $\O_X$-module) can be identified with $\sF \tensor_{\O_X} F^e_* \O_X$ as an $F^e_* \O_X$-module, again using that $F^e_* \O_X$ is isomorphic with $\O_X$ as sheaves of rings. If the line bundle $\sL$ is given by the datum of a local trivialization and transition functions, then the line bundle $F^{e*}\sL$ is given by the $p^e$th powers of the transition functions in that datum for $\sL$.  This shows that
\begin{equation}
\label{eq.FPullBackPowers}
(F^{e})^* \sL \cong \sL^{p^e}\, ,
\end{equation}
\ie{} the pullback along the Frobenius of a line bundle just raises that line bundle to the $p^e$th tensor power. Combining this observation with the projection formula \cite[Chapter II, Exercise 5.1(d)]{Hartshorne} we obtain
\begin{equation}
\label{eq.FrobeniusProjectionFormula}
(F^e_* \sF) \tensor_{\O_X} \sL \cong F^e_* (\sF \tensor_{\O_X} F^{e*} \sL) \cong F^e_* (\sF \tensor_{\O_X} \sL^{p^e})\, .
\end{equation}
This basic equality is used frequently throughout the theory and will be referred to as the \emph{projection formula}.

\subsection{Exercises}

\begin{exercise}
\label{ex.FeLowerStarIsFreeForPolynomial}
Set $X = \Spec k[x_1, \dots, x_n]$ for some perfect field $k$.  Show that $F^e_* \O_X$ is a free $\O_X$-module with basis $F^e_* x_1^{\lambda_1} \cdots x_n^{\lambda_n}$ where $0 \leq \lambda_i \leq p^e - 1$.  Show that the same result also holds for power series $\Spec k\llbracket x_1, \dots, x_n \rrbracket$.
\end{exercise}

\begin{exercise}
\label{ex.VarietiesAreFFinite}
Suppose that $k$ is a perfect field and that $X$ is scheme of (essentially) finite type over $k$.  Prove that the Frobenius map on $X$ is a finite map.
\end{exercise}

\begin{exercise}
\label{ex.FrobeniusPowerChoiceOfGenerators}
Suppose that $I \subseteq R$ is an ideal in a ring $R$ of characteristic $p > 0$.  Show that $I^{[p^e]}$ can be identified with $\Image(F^{e*}I \to F^{e*} R) \subseteq F^{e*}R \cong R$ where the last isomorphism is the canonical one identifying $R$ with $F^e_*R$ sending $r$ to $F^e_* r$.  Conclude that $I^{[p^e]}$, the Frobenius power of $I$, is independent of the choice of generators of $I$.
\end{exercise}

\begin{exercise}
Suppose that $X$ is a smooth $d$-dimensional variety and $\sL$ is a vector bundle of rank $m$ on $X$.  Prove that $F_* \sL$ is also a vector bundle and find its rank.
\vskip 3pt
\noindent
\emph{Hint: } Complete, use Cohen structure theorem \cite[Theorem 28.3]{MatsumuraCommutativeRingTheory}, and use \autoref{ex.FeLowerStarIsFreeForPolynomial}.
\end{exercise}

\begin{exercise}
Suppose that $\sE$ is a locally free sheaf of finite rank on $X$.  Is $\sE^{\tensor p^e}$ isomorphic to  $(F^e)^* \sE$?
\end{exercise}

\begin{exercise}
\label{ex.LocalizationCompletionAndFeLowerStar}
Suppose that $R$ is (essentially) of finite type over a perfect field.
\begin{itemize}
\item[(a)]  If $W \subseteq R$ is any multiplicatively closed set, then show that $W^{-1} (F^e_* R) \cong F^e_* (W^{-1} R)$.  Here the first $F^e_* R$ means as an $R$-module, and the second is as an $W^{-1}R$-module.
\item[(b)]  If $\bm \subseteq R$ is a maximal ideal, prove that $F^e_* \widehat{R} \cong \widehat{F^e_* R}$ where $\hat{\blank}$ denotes completion along $\bm$.  Again, the first $F^e_*$ is the Frobenius for $\hat{R}$, and the second is that of $R$-modules.
\end{itemize}
\end{exercise}

%Given a complex $C^{\mydot}$, we will use $\myH^i(C^{\mydot})$ to denote the $i$th cohomology of $C^{\mydot}$.

\begin{exercise}
\label{ex.FPushforwardAndProj}  Suppose that $X$ is a projective variety with ample line bundle $\sL$ and suppose that $\sF$ is a coherent sheaf on $X$.  Set $S = \bigoplus_{i \in \bZ} H^0(X, \sL^i)$ to be the section ring with respect to $\sL$ and set $M = \bigoplus_{i \in \bZ} H^0(X, \sF \otimes \sL^i)$ to be the saturated graded $S$-module corresponding to $\sF$.  Verify that $F^e_* S$ is a $({1 \over p^e} \cdot \bZ)$-graded ring\footnote{Here ${1 \over p^e} \cdot \bZ$ is the subgroup of $\bQ$ generated by ${1 \over p^e}$.}, the natural map $S \to F^e_* S$ is graded, and $F^e_* M$ is a graded $F^e_* S$-module.  Of course, $F^e_* M$ is also a graded $S$-module.

Show that $F^e_* M$ is not in general isomorphic to $\bigoplus_{i \in \bZ} H^0(X, (F^e_* \sF) \otimes \sL^i)$.  Instead, prove that $\bigoplus_{i \in \bZ} H^0(X, (F^e_* \sF) \otimes \sL^i)$ is isomorphic to a (graded) direct summand of $F^e_* M$.  The summand whose terms have integral gradings.
\end{exercise}

\begin{exercise}
A ring $R$ (or scheme $X$) such that the Frobenius map $F : R \to F_* R$ is a finite map is called \emph{$F$-finite}.  Essentially all rings considered in this paper are $F$-finite but not all rings are.  Find an example of a field which is not $F$-finite.
\end{exercise}

If $X$ is a smooth variety, then we have already seen that $F_* \O_X$ is a locally free (in other words flat) $\O_X$-module.  In this exercise, you will prove the converse.  First we introduce a definition.

\begin{definition}
Suppose that $(R, \bm)$ is a local ring.
A sequence of elements $f_1, \dots, f_n \in \bm \subseteq R$ is called \emph{Lech-independent} if for any $a_1, \dots, a_n \in R$ such that $a_1 f_1 + \dots + a_n f_n = 0$, then each $a_i \in \langle f_1, \dots, f_n \rangle$.
\end{definition}

Now we come to the exercise.

\begin{starexercise} [Kunz's regularity criterion \cite{KunzCharacterizationsOfRegularLocalRings}]
\label{ex.KunzRegularityCriterion}
Suppose that $(R, \bm)$ is a local ring.  We will show that if $F_* R$ is flat, then $R$ is regular.  We need some Lemmas due to Lech \cite{LechInequalitiesRelatedToCertainCouples}.

\begin{itemize}
\item[(i)]  \cite[Lemma 3]{LechInequalitiesRelatedToCertainCouples}.  If $f_1, \dots, f_n$ are Lech-independent elements and $f_1 \in gR$ for some $g \in R$, then $g, f_2, \dots, f_n$ is also Lech-independent. Furthermore, $\langle f_2, \dots, f_n \rangle : g \subseteq \langle f_1, \dots, f_n \rangle$.
\item[(ii)] \cite[Lemma 4]{LechInequalitiesRelatedToCertainCouples}.  If $f_1, \dots, f_n$ are Lech-independent and $f_1 = gh$.  Then
\[
l_R\left( R /\langle f_1, \dots, f_n \rangle \right) = l_R\left(R/\langle g, f_2, \dots, f_n \rangle \right)+ l_R\left(R/\langle h, f_2, \dots, f_n\rangle \right).
\]

\item[(iii)]  Now we return to the proof of the theorem of Kunz.  Show that $\bm^{[p^e]} / (\bm^{[p^e]})^2$ is a free $R/\bm^{[p^e]}$-module.  Conclude that if $\bm = \langle x_1, \dots, x_n \rangle$ is generated by a minimal set of generators, then $x_1^{p^e}, \dots, x_n^{p^e}$ is Lech-independent.
\item[(iv)]  Use the previous parts of the exercise to conclude that $l_R(R/\bm^{[p^e]}) = p^{ne}$.
\item[(v)]  Reduce to the case that $R$ is complete and write $R = S/\ba = k[[x_1, \dots, x_n]]/\ba$ using the Cohen structure theorem \cite[Theorem 28.3]{MatsumuraCommutativeRingTheory} where $k = R/\bm$.  Then notice that $l_S(S/\bm_S^{[p^e]}) = p^{ne}$ for all $e \geq 0$.  Complete the proof of Kunz' regularity criterion.
\end{itemize}
\end{starexercise}

\begin{remark}
A simpler proof of Kunz' result using the Buchsbaum-Eisenbud acyclicity criterion can be found on page 12 of \cite{HunekeTightClosureBook}.  Alberto Fernandez Boix pointed out to us that another short proof can be found in \cite[Theorem 4.4.2]{MajadasRodicioSmoothnessRegularityAndCI}.
\end{remark}

\section{$p^{-e}$-linear maps: definition and examples}
\label{sec.Defn}

In this section we introduce $p^{-e}$-linear maps and give a number of examples which will be discussed in more detail throughout the rest of the paper.

\begin{definition}[$p^{-e}$-linear map]
Suppose that $X$ is a scheme and $\sM$ and $\sN$ are $\O_X$-modules.  A \emph{$p^{-e}$-linear map} is an additive map $\varphi : \sM \to \sN$ such that
\begin{equation}
\label{eq.pMinusERelation}
\varphi(r^{p^e} m) = r\varphi(m)
\end{equation}
for all local sections $r \in \O_X$ and $m  \in \sM$.

Equivalently, we may specify a \emph{$p^{-e}$-linear map} by the data of an $\O_X$-linear map
\[
    \varphi \colon F^e_*\sM \to \sN\, .
\]
We will \emph{frequently and freely} switch between these two points of view, depending on the context.
\end{definition}

If $R$ is a ring, then a $p^{-e}$-linear map between $R$-modules $M$ and $N$ is simply an additive map between them satisfying the rule from \autoref{eq.pMinusERelation}. If $k$ is a perfect field, then $p^{-e}$-linearity for an additive map $\varphi : k \to k$ just means $\varphi(\lambda x) = \lambda^{1/p^e} \varphi(x)$ for all $x,\lambda \in k$.  In particular, such a map is completely determined by where it sends any nonzero element.

\begin{example}
\label{ex.EasyA1Example}
Consider $R = k[x]$.  Then $F_* R$ is a free module with basis
\[
\{F_* 1, F_*x, F_*x^2, \dots, F_*x^{p-1} \},
 \]
see \autoref{ex.FeLowerStarIsFreeForPolynomial} above.  Therefore any $p^{-1}$-linear map from $k[x]$ to any other $R$-module $N$ is simply a choice of where to send these basis elements.
\end{example}

\begin{example}
\label{ex.BasicAnExample}
Consider $R = k[x_1, \ldots, x_n]$ then as we saw in \autoref{ex.FeLowerStarIsFreeForPolynomial}, $F_* R$ is a free $R$-module with basis $F_* x_1^{\lambda_1} \cdots x_n^{\lambda_n}$ for $0 \leq \lambda_i \leq p-1$.  A map $F_* R \to R$ is uniquely determined by where it sends the elements of this basis.

Consider the $R$-linear map $\Phi : F_* R \to R$ which sends $F_* x_1^{p-1} \cdots x_n^{p-1}$ to $1$ and all other basis elements to zero.  In other words:
\[
\Phi \left(F_* (x_1^{\lambda_1} \cdots x_n^{\lambda_n})\right) = \left\{ \begin{array}{cl} x_1^{{\lambda_1 - (p - 1) \over p }} \cdots x_n^{{\lambda_n - (p - 1) \over p }}, & \text{if all ${\lambda_i - (p - 1) \over p} \in \bZ$} \\ & \\ 0, & \text{otherwise}\end{array} \right.
\]
For each tuple ${\bf \lambda} = (\lambda_1, \dots, \lambda_n) \in \{0, 1, \dots, p-1\}^n$, consider the map
$\phi_{\bf \lambda} : F_* R \to R$ defined by the rule $\phi_{\bf \lambda}(F_* \blank) = \Phi(F_* (x_1^{p-1-\lambda_1} \cdots x_n^{p-1 -\lambda_n} \cdot \blank))$.  It is easy to see that $\phi_{\bf \lambda}$ sends ${\bf x}^{\lambda}$ to $1$ and all other basis monomials to zero.

Because we can thus obtain all of the projections this way, it follows that the map $F_* R \to \Hom_R(F_* R, R)$ which sends $F_* z$ to the map $\phi_z(F_* \blank) = \Phi(F_* (z \cdot \blank))$ is surjective as a map of $F_* R$-modules.  On the other hand, it is clearly injective as well and so it is an isomorphism. In other words, we just showed that $\Hom_R(F_* R, R)$ is a free $F_*R$--module generated by $\Phi$.  In other words, $\Phi$ generates $\Hom_R(F_* R, R)$ as an $F_* R$-module.
\end{example}

The most pervasive type of $p^{-1}$-linear maps are maps $\phi : R \to R$.  Of course, for fixed $e$, the set of $p^{-e}$-linear maps $\{ \phi : R \to R\,|\, \phi \text{ is $p^{-e}$-linear} \}$ form a group under addition.  However, as we vary $e$, we have a multiplication of these maps as well.  Indeed, suppose that $\phi : R \to R$ is $p^{-e}$-linear and $\psi : R \to R$ is $p^{-d}$-linear.  Then both $\phi \circ \psi$ and $\psi \circ \phi$ are $p^{-e-d}$-linear.  However, they need not be equal as the following example shows.
It follows that
\[
\bigoplus_{e \geq 0} \left\{ \phi : R \to R\,|\, \phi \text{ is $p^{-e}$-linear} \right\}
\]
forms a noncommutative graded ring.  This graded ring will be studied more in \autoref{subsec.AlgebrasOfMaps}.

\begin{example}
Suppose that $R = \bF_p[x]$.  We will describe two $p^{-1}$-linear maps, $\phi, \psi$ presented as in \autoref{ex.EasyA1Example}.
\begin{itemize}
\item{} $\phi : R \to R$ satisfies $\phi(x^{p-1}) = 1$ and $\phi(x^i) = 0$ for $0 \leq i < p-1$.
\item{} $\psi : R \to R$ satisfies $\psi(1) = 1$ and $\psi(x^i) = 0$ for $0 < i \leq p-1$.
\end{itemize}
Then $\psi \circ \phi$ and $\phi \circ \psi$ are $p^{-2}$-linear maps.  However, notice that
\[
\phi(\psi(x^{p-1})) = \phi(0) = 0
\]
but that
\[
\psi(\phi(x^{p-1})) = \psi(1) = 1.
\]
In particular, $\psi \circ \phi \neq \phi \circ \psi$.
\end{example}

An important class of examples of $p^{-1}$-linear maps are the splittings of Frobenius.
\begin{example} \cite{MehtaRamanathanFrobeniusSplittingAndCohomologyVanishing,RamananRamanathanProjectiveNormality}
Let $X$ be a scheme.  A \emph{Frobenius splitting} is any $p^{-1}$-linear map $\phi : \O_X \to \O_X$ that sends $1$ to $1$.  Equivalently, it is an $\O_X$-linear map $\phi : F_* \O_X\to \O_X$ that sends $F_* 1$ to $1$.  This, in particular, implies that the composition
\begin{equation}
\label{eq.ExFrobSplitIsomorphism}
\O_X \xrightarrow{F} F_* \O_X \xrightarrow{\phi} \O_X
\end{equation}
is an isomorphism, hence $\phi$ ``splits'' the Frobenius.

If $X$ has a Frobenius splitting, then it satisfies many remarkable properties as we shall discuss in detail in \autoref{sec.FSplittings}.  Let us just mention two of them to taste.

If $X$ is a scheme that has some Frobenius splitting $\phi \colon F_*\O_X \to \O_X$ (we call such $X$ Frobenius split), then $X$ is reduced: Indeed, if $x \in \Gamma(U, \O_X)$ is such that $0=x^{p^e}=F^e(x)$ for some $e \geq 0$, then $0=\phi^e(F^e(x))=x$, simply by that fact that $\phi \circ F=\id$. This is a simple but important local property of Frobenius split varieties.

A similarly fundamental global result is the following \emph{vanishing theorem}:  Suppose that $\sL$ is a line bundle and that $H^i(X, \sL^{p}) = 0$ for some $i > 0$ (for example, $H^i(X, \sL^{p^e}) = 0$ holds for $e \gg 0$ for ample $\sL$ by Serre vanishing), then $H^i(X, \sL) = 0$ as well since we have the following isomorphism obtained by tensoring \autoref{eq.ExFrobSplitIsomorphism} by $\sL$, using the projection formula and taking cohomology:
    \[
        H^i(X, \sL) \xrightarrow{F} H^i(X, F_* \sL^{p}) = 0 \xrightarrow{\phi} H^i(X, \sL).
    \]
If $e > 1$, rinse and repeat.
We will study vanishing theorems for Frobenius split varieties in much greater detail in \autoref{thm.VanishingForFrobeniusSplit}.
\end{example}

\subsection{The Cartier isomorphism}
\label{sec.CartierOperator}

%\todo{{\bf Karl: } I was thinking of something along the lines of the bit on the Cartier isomorphism in Esnault-Viehweg, but with more examples (definitely only over perfect fields)?
%Also some connection with the Grothendieck trace map and how things work on say normal varieties.}

We now come to the most important example of a $p^{-1}$-linear map, coming from the Cartier operator.
Suppose that $X$ is a smooth variety over a perfect field $k$ of characteristic $p > 0$.  Consider the de-Rham complex, $\Omega_X^{\mydot}$.  This is not a complex of $\O_X$-modules (the differentials are not $\O_X$-linear).  However, the complex
\[
F_* \Omega_X^{\mydot}
\]
is a complex of $\O_X$-modules (notice that $d(x^p) = 0$).  We now state the Cartier isomorphism.  We take this presentation from \cite{CartierUneNouvelle}, \cite{KatzNilpotentConnections}, \cite{EsnaultViehwegLecturesOnVanishing}, and \cite{BrionKumarFrobeniusSplitting}.

\begin{definition-proposition}
\label{defprop.CartierIsomorphism}
There is a natural isomorphism (of $\O_X$-modules):
\[
C^{-1} : \Omega_X^i\rightarrow \myH^i(F_* \Omega_X^{\mydot} )
\]
\end{definition-proposition}

\begin{remark}
It might strike the reader as odd that we put an inverse on $C$.  This is because the isomorphism in the other direction is called the \emph{Cartier operator} and represented by $C$.  It is just more convenient for us to define $C^{-1}$ than it is to define $C$.
\end{remark}

We will not use the details of this isomorphism later in the paper.  However, the map $T$ we obtain from it below in \autoref{subsec.TraceOfFrob} will be indispensable.

Let us explain how to construct this isomorphism $C^{-1}$.  We follow \cite[9.13]{EsnaultViehwegLecturesOnVanishing} and \cite{KatzNilpotentConnections}.
We begin with $C^{-1}$ in the case that $i = 1$.  We work locally on $X$ (which we assume is affine) and we define $C^{-1}$ by its action on $dx \in \Omega_X^i$, $x \in \O_X$;
\[
C^{-1}(dx) := F_* x^{p-1}dx,
\]
or rather, the image of $F_* x^{p-1}dx$ in cohomology.  In order for this to make sense, we observe that $d(x^{p-1} dx) = 0$, in other words, that $C^{-1}(dx)$ is in the kernel of $d$.  We now must show that $C^{-1}$ is additive.

Now $C^{-1}(d(x)+d(y)) = C^{-1}(d(x+y)) = F_* (x+y)^{p-1} d(x+y)$, we need to compare this to $F_* x^{p-1}dx + F_* y^{p-1}dy = C^{-1}(dx) + C^{-1}(dy)$.  Write $f = {1 \over p} \left( (x+y)^p -x^p - y^p \right)$.  Here the $1 \over p \cdot $ is a formal operation that simply cancels $p$s from the binomial coefficients.  Then
\[
\begin{array}{rl}
df = &  d \left(\displaystyle\sum_{i = 1}^{p-1} \gamma_i x^i y^{p-i} \right)\\
= & \left( \displaystyle\sum_{i = 1}^{p-1} \gamma_{i} i x^{i-1} y^{p-i} \right)dx + \left( \displaystyle\sum_{i = 1}^{p-1} \gamma_i (p-i) x^i y^{p-i-1}\right) dy\\
= & \left( \displaystyle\sum_{i = 1}^{p-1} \gamma_{i} i x^{(p-1)-(p-i)} y^{p-i} \right)dx + \left( \displaystyle\sum_{i = 1}^{p-1} \gamma_i (p-i) x^i y^{(p-1)-i}\right) dy
\end{array}
\]
where $\gamma_i = {1 \over p}{p \choose i} = { (p-1)(p-2) \cdots 1 \over i! (p-i)! } = {1 \over i} {p-1 \choose p-i} = {1 \over p-i} { p-1 \choose i }$.  Thus
\[
df = (x+y)^{p-1} (dx + dy) - x^{p-1}dx - y^{p-1}dy.
\]
Therefore, $x^{p-1}dx + y^{p-1}dy$ and $(x+y)^{p-1} d(x+y)$ are the same in $\Omega_X^{1} \big/ d\left(\Omega_X^0 \right)$.  This proves that $C^{-1}$ is additive.  Finally, we extend by $p$-linearity to obtain that
\[
C^{-1}(f dx) =F_* f^p x^{p-1} dx.
\]

We should also show that $C^{-1}$ is an isomorphism.  We only show that this initial $C^{-1}$ is injective -- in a special case.  Set $X = \Spec \bF_p[x,y]$ (see for example \cite[Theorem 9.14]{EsnaultViehwegLecturesOnVanishing} for how to reduce to the polynomial ring case in general).

Suppose that $C^{-1}(f dx + g dy) = 0$.  Let $h \in \O_X$ be such that we have $f^p x^{p-1} dx + g^p y^{p-1} dy = dh = {\partial h \over \partial x} dx + {\partial h \over \partial y} dy$.  Therefore if $f = \sum \lambda_{i,j} y^i x^j$ we see that
\[
\sum \lambda_{i,j} y^{i p} x^{j p + p - 1} = f^p x^{p-1} = {\partial h \over \partial x}.
\]
However, this is ridiculous unless $fdx + gdy = 0$.  If you take a derivative of some non-zero polynomial in $x$ with respect to $x$, no output can ever have $x^{jp + p - 1}$ in it.  This completes the proof of injectivity of $C^{-1} : \Omega_X^1\rightarrow \myH^1(F_* \Omega_X^{\mydot} )$ in the case that $X = \Spec \bF_p[x,y]$.  The general case is similar.

The surjectivity of $C^{-1}$ is more involved.  See for example, \cite[Theorem 9.14(d)]{EsnaultViehwegLecturesOnVanishing} or \cite[Theorem 1.3.4]{BrionKumarFrobeniusSplitting} or do \autoref{ex.CInverseSurjective}.

At this point, we have only defined
\[
\Omega_X^1\rightarrow \myH^1(F_* \Omega_X^{\mydot} ).
\]
We define $C^{-1} : \Omega_X^i\rightarrow \myH^i(F_* \Omega_X^{\mydot} )$ for $i > 1$ using wedge powers of $C^{-1}$ for $i = 1$.  We make this definition for any $X$.

\begin{example}[Cartier isomorphism $\bA^2$]
\label{example.CartierIsomorphismOnA2}
Returning again to $X = \bA^2 = \bF_p[x,y]$, we explicitly compute $C^{-1} : \Omega_X^2\rightarrow \myH^2(F_* \Omega_X^{\mydot} )$ at the top cohomology.

By definition
\[
C^{-1}(f dx dy) = C^{-1}(f dx \wedge dy) := F_* \big(f^p (x^{p-1} dx) \wedge (y^{p-1} dy)\big) = F_* f^p x^{p-1} y^{p-1} dx dy
\]
or rather its image in cohomology.  Again, this map is an isomorphism, \autoref{ex.TopCartierIsomorphismForA2}.
\end{example}

\subsection{Grothendieck-trace of Frobenius}
\label{subsec.TraceOfFrob}

Suppose that $X$ is a smooth $n$-dimensional variety over a perfect field $k$ of characteristic $p > 0$.  Then coming from the Cartier isomorphism, \autoref{defprop.CartierIsomorphism}, we have an exact sequence:
\[
F_* \Omega_{X/k}^{n-1} \xrightarrow{d} F_* \Omega_{X/k}^{n} \xrightarrow{T} \Omega_{X/k}^n \to 0
\]
The surjective map $T : F_* \Omega_{X/k}^{n} =: F_* \omega_X \rightarrow \omega_X := \Omega_{X/k}^n$ is often called the \emph{trace map} or \emph{Cartier map/operator}.

This map can be constructed in other ways.  With $X$ as above, again set $\omega_X = \Omega_{X/k}^n$.  Then $\omega_X$ is a \emph{dualizing/canonical} module in the sense of \cite[Chapter III, Section 7]{Hartshorne} or more generally, \cite[Chapter V]{HartshorneResidues}.

For any finite dominant map $\pi : Y \to X$ with $Y$ and $X$ smooth, it is a fact (blackboxed for now \cite[Chapter V, Proposition 2.4]{HartshorneResidues}, \cite[Proposition 5.68]{KollarMori}) that $\pi_* \omega_Y \cong \sHom_{\O_Y}(\pi_* \O_Y, \omega_X)$ as a $\pi_* \O_Y$-module.  This is described in greater generality on the next page, see the diagram \autoref{eq.CanonicalUpperShriekDiagram}.  Note that this \emph{completely determines} $\omega_Y$ as well, since $\pi$ is finite and so the data of a coherent $\pi_* \O_Y$-module on $X$ is equivalent to the data of a coherent $\O_Y$-module on $Y$.  Now, we also have the following map:
\begin{equation}
\label{eq.EvalAt1}
\pi_* \omega_Y \cong \sHom_{\O_X}(\pi_* \O_Y, \omega_X) \xrightarrow{\text{eval @ 1}} \omega_X
\end{equation}
This is the map which sends a section $\varphi \in \omega_Y \cong \Gamma(U, \sHom_{\O_Y}(\pi_* \O_Y, \omega_X))$ to the element $\varphi(1) \in \Gamma(U, \omega_X)$.

Now we specialize to the case that $Y = X$ and $\pi = F$ the Frobenius map.

\begin{theorem}
\label{thm.TraceIsCartierMap}
The map described in \autoref{eq.EvalAt1} is the map $T$ described above (up to choice of isomorphism).
\end{theorem}
\begin{proof}[Sketch of proof]
We only show this for $X = \Spec \bF_p[x,y] = \bA^2$.  By considering \autoref{example.CartierIsomorphismOnA2}, we see that the map $T$ sends
\[
F_* f^p x^{p-1} y^{p-1} dx dy \mapsto f dx dy
\]
and everything not of that form to zero.

So we then consider
\[
\xymatrix@R=2pt{
\sHom_{\O_X}(F_* \O_X, \omega_X) \ar@{<->}[r]^-{\simeq} & F_* \omega_X \ar@{<->}[r]^-{\simeq} & F_* \O_X.\\
                                                      & F_* dx dy \ar@{<->}[r] & F_* 1
}
\]
Now, we identify the $\Phi \in \Hom_{\O_X}(F_* \O_X, \omega_X)$ which generates $\sHom_{\O_X}(F_* \O_X, \omega_X)$ as an $F_* \O_X$-module just as in \autoref{ex.BasicAnExample}.  Since $\omega_X = \O_X \cdot (dxdy) \cong \O_X$, we notice that $\Phi$ sends $F_* f^p x^{p-1} y^{p-1} \mapsto f dx dy$ and $\Phi$ sends things not of this form, to zero.

Choosing then $\phi(F_* \blank) = \Phi(F_* c \cdot \blank) \in \sHom_{\O_X}(F_* \O_X, \omega_X)$, we see that the evaluation-at-$1$ map \autoref{eq.EvalAt1} sends $\phi$ to $\Phi(F^e_* c)$.  Making the identification
\[
(F_* \O_X) \cdot (F_* dxdy) = F_* \omega_X \cong \sHom_{\O_X}(F_* \O_X, \omega_X) = (F_* \O_X) \cdot \Phi
\]
we immediately observe that $T$ and the evaluation-at-$1$ map \autoref{eq.EvalAt1} coincide.

The general case for $X \neq \bA^2$ is similar but slightly more technical to write down.  Both the map $T$ and the evaluation-at-1 map can be shown to be a local generator of the same $\sHom$-sheaf.  Thus they coincide up to multiplication by a unit of $\Gamma(X, \O_X)$.
\end{proof}

\subsection{The Trace map for singular varieties}
\label{subsec.TheTraceMapForSingular}
Suppose that $X$ is a normal variety with $U \subseteq X$ the regular locus.  Consider the map $T : F^e_* \omega_U \to \omega_U$ as described above.  This is an element of $\Hom_U(F^e_* \omega_U, \omega_U)$.  However, there is an isomorphism $\Hom_{\O_U}(F^e_* \omega_U, \omega_U) \cong \Hom_{\O_X}(F^e_* \omega_X, \omega_X)$ since $X \setminus U$ is a codimension 2 subset of $X$ and $X$ is normal, see \autoref{sec.reflex}.
Therefore we obtain the following proposition.

\begin{proposition}
\label{prop.TraceMapForNormalVars}
Given any normal variety $X$ there is a trace map $T : F^e_* \omega_X \to \omega_X$ which agrees with, and is completely determined by the map $T$ described in terms of the Cartier isomorphism on the regular locus $U \subseteq X$.
\end{proposition}

Even for non-normal schemes, we can do something similar, we need to work in the derived category.  Suppose that $X$ and $Y$ are schemes of finite type over a field $k$ with a map $f : X \to Y$.  Then there is functor $f^!$ from $D^+_{\coherent}(Y)$ to $D^+_{\coherent}(X)$ (bounded below complexes of $\O_Y$-modules, respectively $\O_X$-modules, with coherent cohomology).  For a precise definition of $f^!$, please see \cite{HartshorneResidues}.  Its abstract existence can nowadays be shown quite formally from general principles, \cf \cite{LipmanDerivedCategoriesAndFunctors}.  Its key property is that it is right adjoint to $\myR f_*$ in the case that $f$ is proper, see \autoref{ex.GrothendieckDuality}.  We will define $f^!$ in two cases which will suffice for our purposes.

 \begin{description}
 \item[Finite]  If $f$ is finite (for example, Frobenius or a closed immersion), then $\sF \in D^b_{\coherent}(X)$ we have an isomorphism of $f_* \O_X$-complexes
\begin{equation}
\label{eq.UpperShriekForFiniteMaps}
f_* f^! \sF = \myR \sHom_{\O_Y}^{\mydot}(f_* \O_X, \sF).
\end{equation}
where $\myR \sHom_{\O_Y}^{\mydot}(f_* \O_X, \sF)$ is the complex obtained by taking an injective resolution of $\sF$ and applying $\sHom_{\O_Y}^{\mydot}(f_* \O_X, \blank)$.  Note that this completely describes $f^!$ since $f$ is finite so that $f_*$ is harmless.
\item[Smooth]  If $f$ is smooth of relative dimension $n$, the for any $\sF \in D^b_{\coherent}(X)$ we have an isomorphism
\[
f^! \sF = (\myL f^* \sF) \otimes (\wedge^n \Omega_{Y/X}^1)[n].
\]
\end{description}

If $f : X \to \Spec k$ is itself the structural map, then we define the \emph{dualizing complex of $X$}, denoted $\omega_X^{\mydot}$ as follows.  View $k \in D^b_{\coherent}(\Spec k)$ as the complex which is trivial except in degree zero where it is $k$.  Then we define $\omega_X^{\mydot} := f^! k$ to be the dualizing complex on $X$.
%Now, $\omega_X^{\mydot}$ is quasi-isomorphic to a bounded complex made up of injective $\O_X$-modules, and because of this, we will view $\omega_X^{\mydot}$ as such a complex.

Consider the following diagram
\begin{equation}
\label{eq.CanonicalUpperShriekDiagram}
\xymatrix{
X \ar[d]_f \ar[r]^{F^e} & X \ar[d]^f \\
\Spec k \ar[r]_{F^e} & \Spec k
}
\end{equation}
where the top row is the absolute $e$-iterated Frobenius on $X$ and the bottom row is the $e$-iterated Frobenius on $k$.  Notice that the bottom row is an isomorphism (although not the identity) and so $(F^e)^! k \cong k$.  The fact that $(f \circ g)^! = f^! \circ g^!$ then implies that $\omega_X^{\mydot}$ is independent of the choice of Frobenius-variant of the $k$-structure on $X$.  In particular, we see that
\begin{equation}
\label{eq.FUpperShriekOfOmega}
\omega_X^{\mydot} \cong f^!k \cong (f \circ F^e)^! k \cong (F^e \circ f)^! k \cong (F^e)^! \omega_X^{\mydot}.
\end{equation}

Now we will apply the \emph{duality functor} $\myR \sHom_{\O_X}^{\mydot}(\blank, \omega_X^{\mydot})$ to the Frobenius map $\O_X \to F^e_* \O_X$.  %Note that if we view $\omega_X^{\mydot}$ as a complex made up of injective $\O_X$-modules, then $\myR \sHom_{\O_X}^{\mydot}(\blank, \omega_X^{\mydot})$ is simply $\sHom_{\O_X}^{\mydot}(\blank, \omega_X^{\mydot})$.
This operation yields
\[
\omega_X^{\mydot} \cong \myR \sHom_{\O_X}^{\mydot}(\O_X, \omega_X^{\mydot}) \leftarrow \myR \sHom_{\O_X}^{\mydot}(F^e_* \O_X, \omega_X^{\mydot}) \cong F^e_* (F^e)^! \omega_X^{\mydot} \cong F^e_* \omega_X^{\mydot}
\]
where the isomorphisms are in the derived category and the final two isomorphisms are due to Equations \autoref{eq.FUpperShriekOfOmega} and \autoref{eq.UpperShriekForFiniteMaps} respectively.  Taking cohomology of this map of complexes gives us maps
\begin{equation}
\label{eq.MapsOnDualizingComplex}
\myH^i \omega_X^{\mydot} \leftarrow F^e_* \myH^i \omega_X^{\mydot} \cong \myH^i F^e_*  \omega_X^{\mydot}
\end{equation}
for each integer $i \in \mathbb{Z}$.

For any equidimensional scheme $X$ over $k$ with dualizing complex $\omega_X^{\mydot} := f^! k$, we define $\omega_X = \myH^{-\dim X}(\omega_X^{\mydot})$ and call it the \emph{canonical module of $X$}.  It follows that \autoref{eq.MapsOnDualizingComplex} induces a map $F^e_* \omega_X \to \omega_X$.  As expected, we then have:

\begin{proposition}
\label{prop.TraceForGeneralMaps}
The map $F^e_* \omega_X \to \omega_X$ coincides with the map $T$ defined previously on the regular locus of $X$.
\end{proposition}

\subsection{Exercises}

\begin{starexercise}
\label{ex.CInverseSurjective}
Suppose that $k$ is a perfect field and that $X = \Spec k[x,y]=\bA^2$, prove that $C^{-1} : \Omega_X^1 \rightarrow \myH^1(F_* \Omega_X^{\mydot} )$ is surjective.
\vskip 3pt
\emph{Hint: } First prove the result for $\bA^1 = \Spec \bF_p[x]$.  Now consider $\sum_j y^j(\alpha_j + \beta_j x^{b_j} dy) = \alpha \in \Omega^1_X$ such that $d \alpha = 0$ where $\alpha_j \in \Omega_{\bA^1}^1$ and $\beta_j \in \bF_p[x]$.  Deduce that $y^{j+1} \alpha_{j+1} + y^j \beta_j dy \in d \Omega_{X}^0$ if $j + 1$ is not divisible by $p$.  Use this to rewrite $\alpha$ and then use the result for $\bA^1$.

This method can be used to do the general proof by induction, see \cite[Theorem 1.3.4]{BrionKumarFrobeniusSplitting}.
\end{starexercise}

\begin{exercise}
\label{ex.TopCartierIsomorphismForA2}
Suppose that $k$ is a perfect field and that $X = \Spec k[x,y]=\bA^2$, prove that $C^{-1} : \Omega_X^2\rightarrow \myH^2(F_* \Omega_X^{\mydot} )$ is an isomorphism.
\end{exercise}

\begin{exercise}
Suppose that $R$ is a regular local ring.  We have seen that $F_* R$ is a flat $R$-module by \autoref{ex.KunzRegularityCriterion}. Consider the evaluation-at-1 map
\[
\xymatrix@R=3pt{
\Hom_R(F_* R, R) \ar@{->}[r]^-e & R\\
\phi \ar@{|->}[r] & \phi(F_* 1)
}
\]
Fix an isomorphism $\gamma : F_* R \to \Hom_R(F_* R, R)$ and consider the composition
\[
e \circ \gamma : F_* R \to R.
\]
Prove that $(e \circ \gamma)$ generates $\Hom_R(F_* R, R)$ as an $F_* R$-module.
\end{exercise}

\begin{exercise}
A variety $X$ is called \emph{Cohen-Macaulay} if $\omega_X^{\mydot} \cong \omega_X[\dim X]$ is a complex with cohomology only in degree $-\dim X$.  Suppose that $H$ is a Cartier divisor on a Cohen-Macaulay scheme $X$.  Prove that $H$ is also Cohen-Macaulay.  Conversely, suppose that $H$ is Cohen-Macaulay, prove that $X$ is Cohen-Macaulay in a neighborhood of $H$.
\vskip 3pt
\emph{Hint: } Apply the \emph{duality functor} to the short exact sequence $0 \to \O_X(-H) \to \O_X \to \O_H \to 0$ and observe that $\myR \sHom_{\O_X}^{\mydot}(\O_H, \omega_X^{\mydot}) \cong \omega_H^{\mydot}$ by \autoref{eq.UpperShriekForFiniteMaps}.  For the converse statement, use Nakayama's Lemma.
\end{exercise}

%\emph{\textsc{Warning}: the final exercises require a bit more homological algebra.}

\begin{starexercise} [Grothendieck duality]  (\emph{For those who wish to learn some homological algebra})
\label{ex.GrothendieckDuality}
Grothendieck duality says the following:
\begin{theorem*}
If $f : X \to Y$ is a proper map of schemes of finite type over a field $k$, then we have an isomorphism in $D^b_{\coherent}(Y)$
\[
\myR \sHom_{\O_Y}^{\mydot}(\myR f_* \sF, \sG) \cong \myR f_* \myR \sHom_{\O_X}^{\mydot}(\sF, f^! \sG)
\]
for $\sF \in D^b_{\coherent}(X)$ and $\sG \in D^b_{\coherent}(Y)$.
\end{theorem*}

Set $Y = \Spec k$ and learn enough about the symbols above to deduce the variant of Serre duality found in Hartshorne, \cite[Chapter III, Section 7]{Hartshorne}.
\end{starexercise}

\section{Connections with divisors}
\label{sec.ConnectionsWithDivisors}

In this section we explain why maps $\varphi \in \Hom_{\O_X}(F^e_* \O_X, \O_X)$ contain roughly the same information as a $\mathbb{Q}$-divisor $\Delta$ such that $K_X + \Delta$ is $\bQ$-Cartier (\ie, such that there exists an integer $n$ such that $n\Delta$ is integral and $nK_X + n\Delta$ is Cartier).  These ideas go back at least to the original papers on Frobenius splittings \cite{MehtaRamanathanFrobeniusSplittingAndCohomologyVanishing,RamananRamanathanProjectiveNormality}. The difference between this section and those original papers is that we normalize our divisors with respect to iterates of Frobenius and thus obtain $\bQ$-divisors.\footnote{Formal sums of codimension 1 subvarieties with rational coefficients.} The statements in this section are somewhat technical.  Therefore, the reader may wish to skim this section for the main idea, and refer back to the numbered bijections as needed throughout the remainder of the article.

Fix $X$ to be a smooth variety of finite type over a perfect field.  Consider an element $\varphi \in \Hom_{\O_X}(F^e_* \O_X, \O_X)$.  We claim that
\begin{equation}
\label{eq.HomIsTwisted}
\Hom_{\O_X}(F^e_* \O_X, \O_X) \cong F^e_* \O_X((1-p^e)K_X).
\end{equation}
Let us prove this claim.  By applying the projection formula as in Equation \eqref{eq.FrobeniusProjectionFormula}, taking global sections and using the fact that $\sHom_{\O_X}(F^e_* \O_X, \O_X(K_X)) \cong F^e_* \O_X(K_X)$, we have
\begin{equation}
\label{eq.HomsAreSections}
\begin{array}{rcl}
\Hom_{\O_X}(F^e_* \O_X, \O_X) & \cong & \Hom_{\O_X}((F^e_* \O_X) \tensor \O_X(K_X), \O_X(K_X)) \\
& \cong & \Hom_{\O_X}(F^e_* (\O_X(p^e K_X)), \O_X(K_X))\\
& \cong & \Hom_{F^e_* \O_X}(F^e_* (\O_X(p^e K_X)), \sHom_{\O_X}(F^e_* \O_X, \O_X(K_X)) )\\
& \cong & F^e_* \Hom_{\O_X}((\O_X(p^e K_X)), \O_X(K_X))\\
& \cong & F^e_* \O_X( (1-p^e)K_X).
\end{array}
\end{equation}
See \autoref{eq.FUpperShriekOfOmega}, \cite[Proposition 5.68]{KollarMori} or \cite{HartshorneResidues}.  Alternately, it follows from Grothendieck duality for the finite map $F : X \to X$ (see \autoref{ex.GRDualityFinite} below).

Therefore, any nonzero map $\phi : F^e_* \O_X \to \O_X$ induces a nonzero global section of $F^e_* \O_X((1-p^e)K_X)$.  By using the fact that $F^e_*$ does not change the underlying structure of sheaves of Abelian groups, we see that there is a bijective correspondence:
\[
\left\{ \begin{array}{c}\text{nonzero elements} \\ \phi \in \Hom_{\O_X}(F^e_* \O_X, \O_X) \end{array} \right\} \longleftrightarrow\left\{ \begin{array}{c}\text{nonzero elements} \\ z \in \Gamma\big(X, \O_X( (1-p^e)K_X) \big) \end{array} \right\}.
\]
Note every nonzero global section of $\O_X( (1-p^e)K_X)$ induces an effective Weil divisor $0 \leq D \sim (1-p^e)K_X$, see \autoref{thm.BijectionBetweenSectionsAndDivisors}.

We notice also that two non-zero elements $z_1, z_2 \in \Gamma\big(X, \O_X( (1-p^e)K_X)\big)$ induce the same divisor if and only if there exists a \emph{unit} $u \in \Gamma\big(U, \O_X \big)$ such that $u z_1 = z_2$.  Therefore, we have the following bijection:
\begin{equation}
\label{eq.DivisorMapCorrespondence1}
\left\{ \begin{array}{c}\text{nonzero $\phi \in$} \\ \Hom_{\O_X}(F^e_* \O_X, \O_X) \end{array} \right\} \Bigg/ \left( \begin{array}{c}\text{multiplication } \\ \text{by units in} \\ \text{ $\Gamma(X, \O_X)$ }\end{array}\right) \longleftrightarrow\left\{ \begin{array}{c}\text{effective divisors } \\ \text{linearly equivalent} \\ \text{to $(1-p^e)K_X$}\end{array} \right\} .
\end{equation}

Now suppose that $X$ is normal but not necessarily smooth.
Of course, the previous argument works fine on $U = X_{\text{reg}} \subseteq X$.   However, Weil divisors are determined off a set of codimension 2.  Likewise $\Gamma\big(X, \O_X( (1-p^e)K_X)\big) = \Gamma\big(U, \O_X( (1-p^e)K_X)\big)$ since $X \setminus U$ has codimension $\geq 2$ \cf \cite[Proposition 2.9]{HartshorneGeneralizedDivisorsOnGorensteinSchemes}.  In particular, we see that
\begin{center}
\fbox{\autoref{eq.DivisorMapCorrespondence1} holds on normal varieties.}
\end{center}

We continue now to work with normal $X$.  Given an effective Weil divisor $D = D_{\phi} \sim (1 - p^e)K_X$ corresponding to $\phi$, set $\Delta = \Delta_{\phi} = {1 \over p^e - 1} D_{\phi}$.  This is an effective $\bQ$-divisor.  Notice that
\[
K_X + \Delta = K_X + {1 \over p^e - 1} D \sim K_X + {1 \over p^e - 1}(1-p^e)K_X = K_X - K_X = 0
\]
In particular, we obtain a bijective correspondence:
\begin{equation}
\label{eq.MapBijectionWithQDivisors1}
\left\{ \begin{array}{c}\text{nonzero  $\phi \in $} \\ \Hom_{\O_X}(F^e_* \O_X, \O_X) \end{array} \right\} \Bigg/ \left( \begin{array}{c}\text{multiplication } \\ \text{by units in} \\ \text{ $\Gamma(X, \O_X)$ }\end{array}\right) \longleftrightarrow\left\{ \begin{array}{c}\text{$\bQ$-divisors } \\ \text{$\Delta \geq 0$ such that} \\ \text{$(p^e - 1)(K_X + \Delta)$ is}\\\text{an integral Weil }\\ \text{divisor linearly} \\ \text{equivalent to 0}\end{array} \right\}
\end{equation}
At this point, it is natural to ask why should one divide by $p^e - 1$.  This division is a normalizing factor as described below.

Suppose that $\phi : F^e_* \O_X \to \O_X$ is an $\O_X$-linear map.  We apply the functor $F^e_*$ and obtain: $F^e_* \phi : F^{2e}_* \O_X \xrightarrow{F^e_* \O_X} F^e_* \O_X$.  Composing this with $\phi$ we obtain $\phi \circ (F^e_* \phi) : F^{2e}_* \O_X \to \O_X$. We use $\phi^2$ to denote this map (note if we view $\phi$ as an honest $p^{-e}$ linear map, then this is really just  $\phi$ composed with itself).  More generally, for each $n \geq 1$, we obtain maps
\begin{equation}
\label{eq.ComposeMaps1}
\phi^n : F^{ne}_* \O_X \to \O_X
\end{equation}
in the same way.

\begin{lemma}  \cite[Theorem 3.11(e)]{SchwedeFAdjunction}
\label{lem.ComposingPhiGivesSameDivisorBase}
Suppose that $X$ is a normal variety.
Then the map $\phi \in \Hom_{\O_X}(F^e_* \O_X, \O_X)$ induces the same $\bQ$-divisor $\Delta$ via \autoref{eq.MapBijectionWithQDivisors1} as does the map $$\phi^n \in \Hom_{\O_X}(F^{ne}_* \O_X, \O_X)$$ for any $n \geq 1$.
\end{lemma}
\begin{proof}
The divisor section correspondence is determined in codimension 1, and so we may assume that $X = \Spec R$ where $(R, \bm)$ is a DVR with $\bm = \langle r \rangle$.  We will simply verify the claim in the Lemma for $n = 2$ and leave the general case to the reader \autoref{ex.ComposingPhiGivesSameDivisorBase}.  Now, since $R$ is regular (and so Gorenstein) and local, $K_X \sim 0$.  Thus
\[
\Hom_{\O_X}(F^e_* \O_X, \O_X) \cong \Gamma(X, F^e_* \O_X( (1-p^e)K_X)) \cong F^e_* R,
\]
we fix $\Phi \in \Hom_{\O_X}(F^e_* \O_X, \O_X)$ corresponding to $F^e_* 1$.  In other words we pick $\Phi$ such that, $D_{\Phi} = 0$ (note that the section $1$ doesn't vanish anywhere).

It is an exercise left to the reader that $\Phi^2 = \Phi \circ (F^e_* \Phi)$ generates $\Hom_R(F^{2e}_* R, R)$ as an $F^{2e}_*$-module, \autoref{ex.CompositionOfGeneratingMaps}.  This is the key point though!

Now consider $\phi(F^e_* \blank) = \Phi(F^e_* u r^a \cdot \blank)$ for some unit $u \in R$ and integer $a \geq 0$.  It follows immediately that $D_{\phi} = a \Div(r)$ and so $\Delta_{\phi} = {a \over p^e - 1} \Div(r)$.

Now we consider $\phi^2$.  We observe that
\[
\phi^2(F^{2e}_* \blank) = \Phi(F^e_* u r^a \Phi(F^e_* ur^a \cdot \blank)) = \Phi(F^e_* \Phi(F^e_* u^{p^e + 1} r^{a(p^e + 1)} \cdot \blank)) = \Phi^2(F^{2e}_* u^{p^e + 1} r^{a(p^e + 1)} \cdot \blank).
\]
Thus $D_{\phi^2} = a(p^e + 1) \Div(r)$ and so that $\Delta_{\phi^2} = {a(p^e + 1) \over p^{2e}-1}\Div(r) = {a \over p^e - 1} \Div(r) = \Delta_{\phi}$ as desired.
\end{proof}

Therefore, we obtain a bijection:
\begin{equation}
\label{eq.MapBijectionWithQDivisors2}
\left\{ \begin{array}{c}\text{nonzero $\phi \in$} \\ \Hom_{\O_X}(F^e_* \O_X, \O_X) \\ \text{as $e \geq 0$ varies}\end{array} \right\} \Bigg/ \left( \begin{array}{c}\text{relation generated} \\ \text{by multiplication} \\ \text{by units in} \\  \text{$\Gamma(X, \O_X)$ and by} \\ \text{composition in \autoref{eq.ComposeMaps1}}\end{array}\right) \longleftrightarrow\left\{ \begin{array}{c}\text{$\bQ$-divisors } \\ \text{$\Delta \geq 0$ such that} \\ \text{$n(K_X + \Delta) \sim 0$}\\\text{for some $n > 0$}\\ \text{with $p$ not} \\ \text{dividing $n$.}\end{array} \right\}
\end{equation}
Here we notice that $(p^e - 1)(K_X + \Delta) \sim 0$ for some $e > 0$ is equivalent to requiring that $n(K_X + \Delta) \sim 0$ for some $n > 0$ which is not divisible by $p$, \autoref{ex.BasicGroupTheory}.

\begin{example}
\label{ex.ZeroDivisorOnAn}
Consider $X = \bA^n = \Spec k[x_1, \ldots, x_n] = \Spec R$ over a perfect field $k$.  Consider $\Phi : F^e_* R \to R$ defined by the following action on monomials
\[
\Phi\left(F^e_* (x_1^{\lambda_1} \cdots x_n^{\lambda_n})\right) = \left\{ \begin{array}{cl} x_1^{{\lambda_1 - (p^e - 1) \over p^e }} \cdots x_n^{{\lambda_n - (p^e - 1) \over p^e }}, & \text{if all ${\lambda_i - (p^e - 1) \over p^e} \in \bZ$} \\ & \\ 0, & \text{otherwise}\end{array} \right.
\]
In other words, $\Phi$ sends $F^e_* (x_1^{p^e - 1} \cdots x_n^{p^e - 1})$ to $1$ and all other elements of the obvious basis of $F^e_* R$ over $R$ to zero.
We already saw in \autoref{ex.BasicAnExample} that $\Phi : F^e_* R \to R$ generates $\Hom_{R}(F^e_* R, R)$ as an $F^e_* R$-module (at least when $e = 1$, but the general case is no different).

But then it immediately follows that the divisor $D_{\Phi}$ is the zero divisor.  In particular, $D_{\Phi}$ corresponds to the element in $\Hom_R(F^e_* R, R) \cong F^e_* R$ that doesn't vanish anywhere.
\end{example}

\subsection{A generalization with line bundles}
\label{sec.GeneralizationOfMapDivisorWithLineBundles}

Previously we considered nonzero maps $\phi \in \Hom_{\O_X}(F^e_* \O_X, \O_X)$.  In this subsection, we generalize this to maps $\phi \in \Hom_{\O_X}(F^e_* \sL, \O_X)$ where $\sL$ is a line bundle on $X$.  This generality actually simplifies some of the statements considered in the previous section.  Indeed, just as in \autoref{eq.HomsAreSections}, it is easy to see that for a smooth variety $X$
\[
\Hom_{\O_X}(F^e_* \sL, \O_X) \cong F^e_* \sL^{-1}( (1-p^e)K_X)
\]
Just as before, this extends to normal varieties as well.  Therefore for any line bundle on a normal variety $X$, we have a bijection of sets.
\begin{equation}
\label{eq.DivisorMapCorrespondence2}
\left\{ \begin{array}{c}\text{nonzero $\phi \in$} \\ \Hom_{\O_X}(F^e_* \sL, \O_X) \end{array} \right\} \Bigg/ \left( \begin{array}{c}\text{multiplication } \\ \text{by units in} \\ \text{ $\Gamma(X, \O_X)$ }\end{array}\right) \longleftrightarrow\left\{ \begin{array}{c}\text{effective Weil} \\ \text{divisors $D$ such} \\ \text{that $\O_X(D) \cong $} \\ \text{$\sL^{-1} ((1-p^e)K_X)$}\end{array} \right\} .
\end{equation}
Thus just as before, $\phi \in \Hom_{\O_X}(F^e_* \sL, \O_X)$ induce $\bQ$-divisors $\Delta_{\phi} = {1 \over p^e - 1}D$ such that $\O_X((p^e - 1)(K_X + \Delta)) \cong \sL^{-1}$.

\begin{definition}
\label{def.QDivisorAssociatedToAMap}
Given $\phi : F^e_* \sL \to \O_X$, we use $\Delta_{\phi}$ to denote the \emph{$\bQ$-divisor associated to $\phi$} as above.
\end{definition}

Finally, consider the data of a line bundle $\sL$ and an $\O_X$-linear map $\phi : F^e_* \sL \to \O_X$.  We will compose $\phi$ with itself in the following way.
We tensor $\phi$ with $\sL$ and use the projection formula to obtain:
\[
F^e_* (\sL^{p^e + 1}) \to \sL
\]
pushing forward by $F^e_*$ we obtain
\[
F^{2e}_* (\sL^{p^e + 1}) \to F^e_* \sL.
\]
Composing with $\phi$ again we obtain a map
\[
F^{2e}_* (\sL^{p^e + 1}) \to \O_X
\]
which we denote by $\phi^2$.  Continuing in this way, we obtain maps
\begin{equation}
\label{eq.ComposeMaps2}
\phi^n : F^{ne}_* (\sL^{p^{(n-1)e} + \cdots + p^e + 1}) \to \O_X
\end{equation}
for all $n \geq 1$.

It is then straightforward to verify that:

\begin{lemma}
\label{lem.QDivisorAssociatedToCompositionTwisted}
The $\bQ$-divisor $\Delta_{\phi}$ induced by $\phi : F^e_* \sL \to \O_X$ is equal to the $\bQ$-divisor $\Delta_{\phi^n}$ induced by $\phi^n : F^{ne}_* (\sL^{p^{(n-1)e} + \dots + p^e + 1}) \to \O_X$.
\end{lemma}
\begin{proof}
Left as an exercise to the reader \autoref{ex.QDivisorAssociatedToCompositionTwisted}.
\end{proof}

In other words, forming the $\bQ$-divisor $\Delta = {1 \over p^e - 1} D$ normalizes the divisor with respect to self composition just as in the case that $\sL = \O_X$.

Given two line bundles $\sL, \sM$, we declare maps $\phi : F^e_* \sL \to \O_X$ and $\psi : F^e_* \sM \to \O_X$ equivalent if there exists a commutative diagram:
\[
\xymatrix{
F^e_* \sL \ar[d]_{\phi} \ar[r]^{\alpha} & F^e_* \sM \ar[d]^{\psi} \\
\O_X \ar[r]_{\id} & \O_X
}
\]
where $\alpha$ is an isomorphism. We also declare $\phi$ and $\phi^n$ to be equivalent.  These relations generate an equivalence relation $\sim$ on pairs $(\sL, \varphi : F^e_* \sL \to \O_X)$.

\begin{theorem}
\label{thm.GeneralBijectionOfSets}
For a normal variety $X$ over a field of characteristic $p > 0$, there is a bijection of sets
\[
\left\{ \begin{array}{c} \text{Line bundles $\sL$ and} \\ \text{maps $\phi : F^e_* \sL \to \O_X$ } \\ \text{modulo equivalence} \end{array}\right\} \longleftrightarrow\left\{ \begin{array}{c}\text{Effective $\bQ$-divisors $\Delta$}\\\text{ such that} \\ \O_X((p^e-1)(K_X + \Delta)) \cong \sL^{-1}  \end{array} \right\}.
\]
\end{theorem}
\begin{proof}
Left to the reader \autoref{ex.GeneralBijectionOfSets}.
\end{proof}

We compute a final example.

\begin{example}
\label{ex.ZeroDivisorOnPn}
Set $X = \bP^n_k$ and consider the line bundle $\sL = \O_X((1-p)K_X) = \O_X((n+1)(p-1))$.  Then
\[
\begin{array}{rl}
& \sHom_{\O_X}(F_* \sL, \O_X)\\
 = & \sHom_{\O_X}(F_* \O_X( (1-p)K_X), \O_X) \\
 = & \sHom_{\O_X}(F_* \O_X(K_X), \O_X(K_X)) \\
 = & F_* \sHom_{\O_X}(\O_X(K_X), \O_X(K_X)) \\
 = & F_* \O_X.
\end{array}
\]
In particular, there is only one non-zero element $\phi \in \Hom_{\O_X}(F_* \sL, \O_X)$ up to scaling by elements of $k$.  In particular, it follows that $D_{\phi} = \Delta_{\phi} = 0$ since a non-zero global section of $F_* \O_X$ doesn't vanish anywhere.  On the affine charts, this element is easily seen to coincide with the map described in \autoref{ex.ZeroDivisorOnAn} (at least for $e = 1$).

On the other hand, there is an obvious map
\[
\psi : F_* \O_X \to \O_X
\]
defined by the rule $\psi(F_* y) = y^{1/p}$, for $y \in \Gamma(U, \O_X)$, if $y^{1/p} \in \Gamma(U, \O_X)$ and $\psi(F_* y) = 0$ otherwise.

It is an exercise left to the reader that $D_{\psi} = (p-1) F$ where $F$ is the union of the various coordinate axes in $\bP^n$.  For example, if $n = 2$ and $X = \Proj k[x,y,z]$, then $F = V(xyz)$.
\end{example}

\subsection{Exercises}

\begin{exercise}[Grothendieck duality for a finite map]
\label{ex.GRDualityFinite}
Suppose that $R \subseteq S$ is a finite inclusion of Cohen-Macaulay local rings and $M$ is an $S$-module.  Grothendieck duality for this inclusion says that there is an isomorphism of $S$-modules:
\[
\Hom_R(M, \omega_R) \cong \Hom_S(M, \omega_S).
\]
Here $\omega_R$ and $\omega_S$ are canonical modules for $R$ and $S$ respectively.  Verify that this is an easy consequence of the formula $\Hom_R(S, \omega_R) = \omega_S$, a formula which was given to you \autoref{eq.FUpperShriekOfOmega}.  %Finally, prove this formula as well using the fact that $\omega_R$ is simply defined to be any module of finite injective dimension such that $\Hom_R(\omega_R, \omega_R) \cong R$ and $\Ext_R^i(\omega_R, \omega_R) = 0$ for $i > 0$ \todo{Karl:  Think about how to simplify this}.
\end{exercise}

\begin{exercise}
\label{ex.CompositionOfGeneratingMaps}
Suppose that $R$ is a ring and $S$ is an $R$-algebra such that $\Hom_R(S, R) \cong S$ as $S$-modules.  Suppose that $M$ is any $S$-module and prove that the natural map:
\[
\Hom_S(M, S) \times \Hom_R(S, R) \to \Hom_R(M, R)
\]
defined by $(\psi, \phi) \mapsto \phi \circ \psi$ is surjective.

In particular, every map in $\Hom_R(M, R)$ can be factored through a map in $\Hom_R(S, R)$.  A solution can be found in \cite[Appendix F] {KunzKahlerDifferentials}
\end{exercise}

\begin{exercise}
\label{ex.ComposingPhiGivesSameDivisorBase}
Prove the general case of \autoref{lem.ComposingPhiGivesSameDivisorBase}.
\end{exercise}

\begin{exercise}
\label{ex.GeneratingMapForGorensteinRings}
Suppose that $R \subseteq S$ is a finite extension of Gorenstein local rings.  Prove that $\Hom_R(S, R)$ is a rank-1 free $S$-module.  Conclude that if $R$ is Gorenstein and local, $\Hom_R(F^e_* R, R)$ is isomorphic to $F^e_* R$ as an $F^e_* R$-module.
\vskip 3pt
\emph{Hint: } Since $R$ is Gorenstein and local (semi-local is good enough), $\omega_R \cong R$.
\end{exercise}

\begin{exercise}
\label{ex.BasicGroupTheory}
Suppose we are given an integer $n > 0$ such that $p$ does not divide $n$, prove that $n | (p^e - 1)$ for some integer $e > 0$.  Conclude that $(p^e - 1)(K_X + \Delta) \sim 0$ for some $e > 0$ if and only if $n(K_X + \Delta) \sim 0$ for some $n > 0$ which is not divisible by $p$.
\end{exercise}

\begin{exercise}
Compute $D_{\psi}$ and $\Delta_{\psi}$ where $\psi$ is as in \autoref{ex.ZeroDivisorOnPn}.
\end{exercise}

\begin{exercise}
\label{ex.QDivisorAssociatedToCompositionTwisted}
Prove \autoref{lem.QDivisorAssociatedToCompositionTwisted}.  See \cite[Theorem 3.11(e)]{SchwedeFAdjunction}.
\end{exercise}

\begin{starexercise}
\label{ex.GeneralBijectionOfSets}
Prove \autoref{thm.GeneralBijectionOfSets}.
\end{starexercise}

\begin{exercise}
\label{ex.TraceMapInducesZeroDivisor}
Suppose that $X$ is a smooth (or Gorenstein) variety and $T : F_* \omega_X \to \omega_X$ is the trace map as described in \autoref{subsec.TheTraceMapForSingular}.  By twisting by $-K_X$ and reflexifying, we obtain a map $\Phi : F_* \O_X( (1-p)K_X) \to \O_X$.  Prove that $\Phi$ corresponds to the zero divisor by \autoref{eq.DivisorMapCorrespondence2}.
\end{exercise}

\begin{exercise}
A normal variety $X$ is called \emph{$\bQ$-Gorenstein} if $\O_X(nK_X)$ is a line bundle for some $n > 0$ (in other words, $nK_X$ is Cartier).  Note that we do not require $\bQ$-Gorenstein varieties to be Cohen-Macaulay.  In this case, the \emph{index of $K_X$} is the smallest $n > 0$ such that $nK_X$ is a Cartier divisor.

Suppose that $X$ is $\bQ$-Gorenstein with index not divisible by $p$.  Suppose that $R = \O_{X,x}$ is the stalk of $R$ at some point $x \in X$.  Prove that we have an isomorphism of $R$-modules, $F^e_* R \cong \Hom_R(F^e_* R, R)$, for all sufficiently divisible $e$.
\end{exercise}

\begin{exercise}
Suppose that $R$ is a normal domain and that $\phi : F^e_* R \to R$ is an $R$-linear map corresponding to a divisor $\Delta_{\phi}$ as in \autoref{def.QDivisorAssociatedToAMap}.  Fix a non-zero $g \in R$.  Form a new map
\[
\psi(F^e_* \blank) = \phi(F^e_* (g \cdot \blank)).
\]
Prove that $\Delta_{\psi} = \Delta_{\phi} + {1 \over p^e - 1} \Div(g)$.
\end{exercise}

\begin{starexercise}
Suppose that $\phi : F^e_* \sL \to \O_X$ and $\psi : F^f_* \sM \to \O_X$ are two $\O_X$-linear maps.  Form the twisted composition $\phi \circ (F^e_* \psi')$ as follows.  Twist $\psi$ by $\sL$ to get $\psi' : F^f_* (\sM \tensor \sL^{p^f}) \to \sL$.  Now pushforward by $F^e_*$ and compose with $\phi$ and obtain:
\[
\phi \circ (F^e_* \psi') : F^{f+e}_* (\sM \tensor \sL^{p^f}) \xrightarrow{F^e_* \psi'} F^e_* \sL \xrightarrow{\phi} \O_X.
\]
Find a relation between $\Delta_{\phi}$, $\Delta_{\psi}$ and $\Delta_{\phi \circ (F^e_* \psi')}$ where the $\Delta$ are $\bQ$-divisors defined as in \autoref{def.QDivisorAssociatedToAMap}.  For a solution, see the proof of \cite[Lemma 4.9(i)]{SchwedeSmithLogFanoVsGloballyFRegular}.
\end{starexercise}

\begin{starexercise}[Non-effective divisors]
\label{ex.NonEffectiveDivisors}
Fix a line bundle $\sL$ on a variety $X$.
There is a bijection between non-zero elements of $\Hom_{\O_X}(F^e_* \sL, \sK(X))$ and (\emph{not necessarily effective}) Weil divisors $D$ such that $\O_X( D) \cong \sL^{-1}( (1-p^e)K_X)$.

Indeed, suppose that $\phi \in \Hom_{\O_X}(F^e_* \sL, \sK(X))$.  Then, working locally if needed, for some sufficiently large Cartier divisor $E \geq 0$, we have that $\phi(F^e_* \sL( (1 - p^e)E)) \subseteq \O_X$.  Set $\psi : F^e_* \sL( (1-p^e)E) \to \O_X$ to be the restriction map.  Then $\psi$ induces a divisor $D_{\psi} > 0$.  Set $D_{\phi} = D_{\psi} + (1-p^e)E$ and prove that $D_{\phi}$ is independent of the choice of $E$.
\end{starexercise}

\section{Frobenius splittings}
\label{sec.FSplittings}

In this section we give a brief introduction to Frobenius splittings.  A more complete treatment can be found in \cite[Chapter 1]{BrionKumarFrobeniusSplitting}.

Suppose that $X$ is a scheme over a perfect field of characteristic $p > 0$.
\begin{definition}
\label{def.FSplitFpure}
We say that $X$ is \emph{Frobenius split} (or \emph{$F$-split}) if the map
\[
\O_X \to F_* \O_X
\]
splits as a map of $\O_X$-modules.  In this case the splitting map $\phi : F_* \O_X \to \O_X$ is called a \emph{Frobenius splitting}.  Of course, there may be multiple different Frobenius splittings $\phi \in \Hom_{\O_X}(F_* \O_X, \O_X)$.

Likewise, we say that a ring $R$ is \emph{Frobenius split} (or \emph{$F$-pure}) if the map
\[
R \to F_* R
\]
splits as a map of $R$-modules.

A scheme $X$ is said to be \emph{$F$-pure} (or \emph{locally $F$-split}) if every point $x \in X$ has a neighborhood which is $F$-split.
\end{definition}

\begin{remark}
Frobenius split varieties were formally introduced in \cite{MehtaRamanathanFrobeniusSplittingAndCohomologyVanishing} (also see \cite{RamananRamanathanProjectiveNormality}), although very closely related concepts were studied in \cite{PeskineSzpiroDimensionProjective,HochsterRobertsFrobeniusLocalCohomology,HartshorneSpeiserLocalCohomologyInCharacteristicP,HaboushAShortProofOfKempf}.  Indeed, Frobenius split affine varieties (\ie{} rings) had been heavily studied by Hochster and his students in the 1970s and 1980s \cf \cite{FedderFPureRat}.
\end{remark}

We shall see below that every regular variety is $F$-pure \autoref{prop.RegularAreFSplitVarieties} but not every regular variety is $F$-split \autoref{lem.FrobeniusSplitImpliesSections}.

\begin{lemma}
\label{lem.FrobeniusSplitOfIterate}
A variety $X$ is Frobenius split if and only if
\begin{itemize}
\item[(a)] the $e$-iterated Frobenius $\O_X \to F^e_* \O_X$ splits for some $e$, \emph{or}
\item[(b)] the $e$-iterated Frobenius $\O_X \to F^e_* \O_X$ splits for all $e$.
\end{itemize}
\end{lemma}
\begin{proof}
This is left as an exercise to the reader, see \autoref{ex.FrobeniusSplitOfIterate}.
\end{proof}

Suppose that $X$ is a variety, we will look for Frobenius splittings inside $\Hom_{\O_X}(F^e_* \O_X, \O_X)$.  Indeed, notice that for any $c \in \Gamma(X, \O_X)$, we have a map $\Hom_{\O_X}(F^e_* \O_X, \O_X) \to \Gamma(X, \O_X)$ defined by evaluation at $c$, in other words, $\phi \mapsto \phi(F^e_* c)$.  Now we observe that:

\begin{lemma}
\label{lem.RFrobeniusSplitIfAndOnlyIfEvaluationAt1Surjects}
A variety $X$ is Frobenius split if and only if the evaluation-at-1 map \[
\Hom_{\O_X}(F^e_* \O_X, \O_X) \to \Gamma(X, \O_X)
\]
 is surjective.
\end{lemma}
\begin{proof}
Left as an exercise to the reader in \autoref{ex.ProofOfFSplitLemma} below.
\end{proof}

Finally, we observe that a normal $X$ is Frobenius split if and only if the regular locus of $X$ is Frobenius split.

\begin{lemma}
Suppose that $X$ is normal and $U \subseteq X$ is the regular locus.  Then $X$ is Frobenius split if and only if $U$ is Frobenius split.
\end{lemma}
\begin{proof}
The natural restriction map $\Hom_{\O_X}(F^e_* \O_X, \O_X) \to \Hom_{\O_U}(F^e_* \O_U, \O_U)$ is an isomorphism since $X \setminus U$ has codimension $\geq 2$ and the $\sHom$ sheaves are reflexive.  See \autoref{sec.reflex} and \cite[Lemma 1.1.7]{BrionKumarFrobeniusSplitting} for additional discussion.
\end{proof}

\subsection{Local properties of Frobenius split varieties}

The easiest property to prove about Frobenius split varieties is that they are reduced.

\begin{lemma}
\label{lem.FrobeniusSplitIsReduced}
Suppose that a scheme $X$ is $F$-pure, then $X$ is reduced.
\end{lemma}
\begin{proof}
Without loss of generality we may assume that $X = \Spec R$ is affine and Frobenius split.  Suppose that $x \in R$ is such that $x^n = 0$.  Then $x^{p^e} = 0$ for some $e > 0$ (where $p$ is the characteristic of $R$).  Therefore $x = x \phi(F^e_* 1) = \phi(F^e_* x^{p^e}) = \phi(F^e_* 0) = 0$.
\end{proof}

First we identify some Frobenius split varieties.

\begin{proposition}[Regular affine varieties are Frobenius split]
\label{prop.RegularAreFSplitVarieties}
Suppose that $X = \Spec R$ is a regular affine variety.  Then $X$ is Frobenius split.
\end{proposition}
\begin{proof}
We prove the result for $R_{\bm} = \O_{X,x}$, the stalk of $X$ at a closed point $x \in X$.  The global case is \autoref{ex.FSplitAtPointsImpliesFSplit}.  Let $\hat R$ denote the completion of $R_{\bm}$ at the maximal ideal $\bm$.  Now consider the evaluation-at-1 map $\Phi : \Hom_{R_{\bm}}(F^e_* R_{\bm}, R_{\bm}) \to R_{\bm}$.  Tensoring with $\hat{R}$ gives us a map
\[
\hat{\Phi} : \Hom_{\hat R}(F^e_* \hat{R}, \hat{R}) \cong \Hom_{R_{\bm}}(F^e_* R_{\bm}, R_{\bm}) \tensor_{R_{\bm}} \hat{R} \to R_{\bm} \tensor_{R_{\bm}} \hat{R} \cong \hat{R}.
\]
Here we have used \autoref{ex.LocalizationCompletionAndFeLowerStar}.  Note that by the Cohen-structure theorem, \cite[Theorem 28.3]{MatsumuraCommutativeRingTheory}, we have that $\hat{R} = k\llbracket x_1, \dots, x_n \rrbracket$.  It follows then from the argument of \autoref{ex.FeLowerStarIsFreeForPolynomial} that $F^e_* \hat{R}$ is free as an $\hat{R}$-module and in particular, that there is a splitting of $\hat{R} \to F^e_* \hat{R}$.  In particular, $\hat{\Phi}$ is surjective.  But therefore $\Phi$ is surjective as well since tensoring with $\hat{R}$ is faithfully flat.  Thus by \autoref{lem.RFrobeniusSplitIfAndOnlyIfEvaluationAt1Surjects}, we are done.
\end{proof}

Of course, not all Frobenius split varieties are regular.

\begin{lemma}[Simple normal crossings are $F$-split]
\label{lem.NormalCrossingsAreFSplit}
The ring \[
R = k[x_1, \ldots, x_n]/\langle x_1 \cdot x_2 \cdots x_n \rangle = S/J
\]
is Frobenius split.
\end{lemma}
\begin{proof}
Observe we have an ``obvious'' Frobenius splitting $\phi : F^e_* k[x_1, \dots, x_n] \to k[x_1, \dots, x_n]$ coming from \autoref{ex.FeLowerStarIsFreeForPolynomial}, which sends the basis element corresponding to $F^e_* 1$ to $1$ and sends all the other basis elements $x_1^{\lambda_1} \dots x_n^{\lambda_n}$ to $0$.  We want to consider what this map does to the ideal $\langle x_1 \cdot x_2 \cdots x_n \rangle = J$.
Consider any monomial in ${\bf x}^{\alpha} = x_1^{\alpha_1} \cdots x_n^{\alpha_n} \in \langle x_1 \cdot x_2 \cdots x_n \rangle = J$.  Then $\phi(F^e_* {\bf x}^{\alpha}) \neq 0$ if and only if $p^e | \alpha_i$ for each $i$.  In particular, this means that $\phi(F^e_* {\bf{x}}^{\alpha}) = x_1^{\beta_1} \cdots x_n^{\beta_n}$ with each $\beta_i \geq 1$.  Therefore, $\phi(F^e_* {\bf x}^{\alpha}) \in  J$.  Since every element of $J$ is a sum of such monomials, we have that $\phi(F^e_* J) \subseteq J$.

But now consider the commutative diagram:
\begin{equation}
\label{eq.SplittingDiagram}
\xymatrix{
F^e_* J \ar[r]^{\phi|_J} \ar@{^{(}->}[d] & J\ar@{^{(}->}[d]\\
F^e_* R \ar@{->>}[d] \ar[r]^{\phi} & R\ar@{->>}[d] \\
F^e_* (R/J) \ar[r]_{\phi/J} & R/J\\
}
\end{equation}
Since $\phi$ sends $1$ to $1$, so does $\phi/J$.
\end{proof}

In the next section, we will introduce a highly effective tool, based upon similar analysis, which can be used to test whether an affine variety is Frobenius split -- Fedder's criterion.

\begin{definition}
Suppose that $\phi : F^e_* \O_X \to \O_X$ is a Frobenius splitting, then an ideal sheaf $J \subseteq \O_X$ is called \emph{compatibly ($\phi$-)split} if $\phi (F^e_* J) \subseteq J$.  If the subscheme $Y = V(J) \subseteq X$, then we also say that $Y$ is \emph{compatibly ($\phi$-)split}.
\end{definition}

Note that in \autoref{lem.NormalCrossingsAreFSplit}, we showed that the coordinate hyperplanes were compatibly split with the obvious Frobenius splitting on $X = \Spec k[x_1, \dots, x_n]$.  Indeed, consider the following proposition:

\begin{proposition}[Properties of compatibly split varieties]
\label{prop.PropertiesOfCompatiblySplitSubvarieties}
Suppose that $\phi : F^e_* \O_X \to \O_X$ is a Frobenius splitting.  Then:
\begin{enumerate}
\item If $J \subseteq \O_X$ is compatibly $\phi$-split, then $V(J)$ is Frobenius split as well.  In particular, $J$ is a radical ideal.
\item If $J \subseteq \O_X$ is compatibly $\phi$-split, then $\phi(F^e_* J) = J$ (instead of just contained in).
\item If $I, J \subseteq \O_X$ is compatibly $\phi$-split, then so are $I+J$ and $I \cap J$.
\item If $Q$ is a minimal prime over $J$, then $Q$ is also compatibly $\phi$-split.
\item If $I \subseteq \O_X$ is compatibly $\phi$-split, then so is $I : K$ for any ideal sheaf $K \subseteq \O_X$.
\item A prime ideal sheaf $Q$ is compatibly $\phi$-split if and only if $Q \cdot \O_{X,Q}$ is compatibly $\phi_Q$ split where $\phi_Q$ is the map induced on the stalk $\phi_Q : F^e_* \O_{X,Q} \to \O_{X,Q}$.
\end{enumerate}
\end{proposition}
\begin{proof}
This is left as an exercise to the reader in \autoref{ex.PropertiesOfCompatSplitSubvarieties}.
\end{proof}

\begin{remark}
Suppose that $\phi : F^e_* \O_X \to \O_X$ is a Frobenius splitting.  It is easy to see that a sort of converse to \autoref{prop.PropertiesOfCompatiblySplitSubvarieties}(a) holds.  In particular, suppose there is a commutative diagram
\[
\xymatrix{
F^e_* \O_X \ar@{->>}[d] \ar[r]^{\phi} & \O_X \ar@{->>}[d] \\
F^e_* (\O_X/J) \ar[r]_{\phi/J} & \O_X/J\\
}
\]
then $J$ is $\phi$-compatibly split (simply take the kernel of the vertical arrows).
\end{remark}

One important point about Frobenius splittings are that compatibly split subvarieties intersect normally.  In particular:

\begin{corollary}
\label{cor.IntersectionsOfCompatiblySplitVarietiesAreReduced}
If $\phi : F^e_* \O_X \to \O_X$ is a Frobenius splitting, if $I$ and $J$ are compatibly $\phi$-split, then $I+J$ is a radical ideal.
\end{corollary}
\begin{proof}
Combine properties (a) and (c) from \autoref{prop.PropertiesOfCompatiblySplitSubvarieties}.
\end{proof}

Also see \autoref{ex.FSplitAreWeaklyNormal} below.

\subsection{Global properties of Frobenius split varieties}

Now we turn to projective (or more generally complete) Frobenius split varieties.  First we introduce another definition.

\begin{definition}
\label{def.FrobeniusSplitAlongADivisor}
Suppose that $D$ is an effective Weil divisor on a normal variety $X$.  Then we say that $X$ is \emph{$e$-Frobenius split relative to $D$} if the composition:
\[
\O_X \to F^e_* \O_X \hookrightarrow F^e_* (\O_X(D))
\]
is split.
\end{definition}

Notice that if $X$ is $e$-iterated Frobenius split relative to $D$, then $X$ is Frobenius split.
We mentioned earlier that regular affine varieties are Frobenius split \autoref{prop.RegularAreFSplitVarieties}, but not every smooth projective variety is Frobenius split.  We prove that now.

\begin{lemma}
\label{lem.FrobeniusSplitImpliesSections}
If $X$ is proper, Frobenius split and normal, then $H^0(X, \O_X(-nK_X)) \neq 0$ for some $n > 0$.  In particular $X$ is not of general type.
Even more, if $X$ is $e$-Frobenius split relative to an ample divisor $A$, then $-K_X$ is big.
\end{lemma}
\begin{proof}
The fact that $X$ is Frobenius split implies that there is some non-zero element $\phi \in \Hom_X(F^e_* \O_X, \O_X) \cong H^0(X, F^e_* \O_X( (1-p^e)K_X))$ by \autoref{sec.ConnectionsWithDivisors}.  In particular, $H^0(X, F^e_* \O_X( (1-p^e)K_X)) \neq 0$.  But $F^e_* \O_X( (1-p^e)K_X)$ is isomorphic to $\O_X( (1-p^e)K_X)$ as an Abelian group and so the result follows for $n = (p^e - 1)$.

For the second statement, we notice that we have a section $\phi \in \Hom_X(F^e_* \O_X(D), \O_X) \cong H^0(X, F^e_* \O_X( (1-p^e)K_X - A))$ and so there is an effective divisor $H \sim (1 - p^e) K_X - A$ and thus $(1-p^e)K_X \sim A + H = \text{``ample + effective''}$ and so $K_X$ is big\footnote{On a projective variety $X$, you can take the definition of \emph{big} to be a divisor which has a multiple which is linearly equivalent to an ample divisor plus an effective divisor \cite[Corollary 2.2.7]{LazarsfeldPositivity1}.}.
\end{proof}

Our next goal is to prove vanishing theorems for Frobenius split varieties.  First however, we need the following Lemma.

\begin{lemma}
\label{lem.FrobeniusSplitImpliesPowersRaised}
If $X$ is $e$-Frobenius split relative to $D$, then for any integer $n > 0$, $X$ is $ne$-Frobenius split relative to $(p^{(n-1)e} + \dots + p^e + 1)D$.
\end{lemma}
\begin{proof}
Suppose that $\O_X \to F^e_* \O_X \to F^e_* \O_X(D) \xrightarrow{\phi} \O_X$ is the Frobenius splitting.  By tensoring this with $D$, taking the reflexification of the sheaves, and applying the functor $F^e_*$, we obtain a splitting
\[
F^e_* (\O_X(D)) \to F^e_* (\O_X(p^e D)) \to F^{2e}_* (\O_X(D + p^e D)) \xrightarrow{F^e_* \phi(D)} F^e_* (\O_X(D)).
\]
But now composing with Frobenius and $\phi$ on the left and right sides respectively, we obtain our desired splitting
\[
\O_X \to F^e_* (\O_X(D)) \to F^e_* (\O_X(p^e D)) \to F^{2e}_* (\O_X(D + p^e D)) \xrightarrow{F^e_* \phi(D)} F^e_* (\O_X(D)) \xrightarrow{\phi} \O_X.
\]
Continuing in this way yields the desired result.
\end{proof}

\begin{theorem}[Vanishing Theorems for Frobenius split varieties]
\label{thm.VanishingForFrobeniusSplit}
Suppose that $X$ is a projective Frobenius split variety.  Then:
\begin{itemize}
\item[(a)]  $H^i(X, \sL) = 0$ for any ample line bundle $\sL$ and any $i > 0$.
\item[(b)]  $H^i(X, \sL \tensor \omega_X) = 0$ for any ample line bundle $\sL$ and any $i > 0$.
\item[(c)]  If $X$ is normal and $e$-Frobenius split relative to an ample Cartier divisor $D$, then we have $H^i(X, \sL) = 0$ for any nef line bundle $\sL$ and any $i > 0$.
\item[(d)]  If $X$ is normal and $e$-Frobenius split relative to an ample Cartier divisor $D$ such that $X \setminus D$ is regular, then $H^i(X, \sL \tensor \omega_X) = 0$ for any big and nef line bundle $\sL$ and any $i > 0$.
\end{itemize}
\end{theorem}
\begin{proof}
For (a), notice that we have a splitting of $\sL \cong \O_X \tensor \sL \to (F^e_* \O_X) \tensor \sL \cong F^e_* \sL^{p^e}$.  Thus $H^i(X, \sL) \hookrightarrow H^i(X, F^e_* \sL^{p^e})$ injects.  On the other hand $H^i(X, F^e_* \sL^{p^e}) \cong H^i(X, \sL^{p^e})$ as Abelian groups, and the latter vanishes for $i > 0$ and $e \gg 0$ by Serre vanishing.

For (b), notice that an application of $\sHom_{\O_X}(\blank, \omega_X)$ to the splitting $\O_X \to F^e_* \O_X \to \O_X$ induces a splitting:
\[
\omega_X \xleftarrow{T} F^e_* \omega_X \hookleftarrow \omega_X.
\]
Twisting by $\sL$ and applying the projection formula gives us
\[
\omega_X \tensor \sL \xleftarrow{T} F^e_* (\omega_X \tensor \sL^{p^e}) \hookleftarrow \omega_X \tensor \sL.
\]
Taking cohomology for $i > 0$ we obtain maps
\[
H^i(X, \omega_X \tensor \sL) \xleftarrow{T} H^i(X, F^e_* (\omega_X \tensor \sL^{p^e})) \hookleftarrow H^i(X, \omega_X \tensor \sL)
\]
whose composition is an isomorphism.  But the middle term vanishes by Serre vanishing since we may take $e \gg 0$.

For (c), we first notice that by using \autoref{lem.FrobeniusSplitImpliesPowersRaised} we may assume that $D$ is as ample as we wish (at the expense of increasing $e$).  Thus, using the same strategy as in (a), it is sufficient to prove that $H^i(X, \O_X(D) \tensor \sL^{p^e}) = 0$ for all $i > 0$.  But this follows from Fujita's vanishing theorem \cite{FujitaVanishingTheorems}.

Part (d) is left as a somewhat involved exercise to the reader \autoref{ex.VanishingForFrobeniusSplit}.
\end{proof}

Finally, we notice that sections on Frobenius split subvarieties often extend to sections on the ambient spaces.

\begin{theorem}
Suppose that $Y \subseteq X$ is compatibly Frobenius split.  Then the natural maps:
\[
H^0(X, \sL) \to H^0(Y, \sL|_Y)
\]
are surjective for any ample line bundle $\sL$.
\end{theorem}
\begin{proof}
By composition of the Frobenius splitting with itself, we have the following diagram for any $e > 0$.
\[
\xymatrix{
H^0(X, F^e_* (\sL^{p^e})) \ar@{->>}[d] \ar@{->>}[r]^{\beta} & H^0(X, F^e_* ( \sL^{p^e}|Y)) \ar[r] \ar@{->>}[d] & H^1(X, F^e_* (I_Y \tensor \sL^{p^e}) ) = 0\\
H^0(X, \sL) \ar[r]_{\alpha} & H^0(X, \sL|_Y) \\
}
\]
Note we have the top-right vanishing by Serre vanishing which implies that $\beta$ is surjective.  The vertical maps are surjective because they are obtained from twisting the Frobenius splitting $F^e_* \O_X \to \O_X$ by $\sL$.  The diagram then implies that $\alpha$ is surjective, this completes the proof.
\end{proof}

\subsection{Tools for proving proper varieties are Frobenius split}

There are two common tools for proving that proper varieties are Frobenius split.  The first involves a study of the singularities of sections of $H^0(X, \O_X((1-p^e)K_X))$.  The second is a general fact that images of Frobenius split varieties often remain Frobenius split.  In many applications, these tools are combined.

\begin{theorem}\cite{MehtaRamanathanFrobeniusSplittingAndCohomologyVanishing} \cite[Section 1.3]{BrionKumarFrobeniusSplitting}
\label{thm.MRFSplitProof}
Suppose $X$ is a proper normal $d$-dimensional variety of finite type over an algebraically closed field of characteristic $p > 0$.  Further suppose that there is an effective divisor $D$, linearly equivalent to $(1-p^e)K_X$ for some $e$, that satisfies the following condition:
\begin{itemize}
\item{} There exists a smooth point $x \in X$ and divisors $D_1, D_2, \dots, D_d$ intersecting in a simple normal crossings divisor at $x \in X$ such that $D = (p^e - 1) D_1 + \dots + (p^e - 1) D_d + G$ for some effective divisor $G$ not passing through $x \in X$.
\end{itemize}
Then $X$ is Frobenius split by a map $\phi : F^e_* \O_X \to \O_X$ which corresponds to $D$ as in \autoref{eq.DivisorMapCorrespondence1}.
\end{theorem}
\begin{proof}
There are two main ideas in this proof.
\begin{itemize}
\item[(a)]  $D$ corresponds to some map, $\phi : F^e_* \O_X \to \O_X$ by \autoref{eq.DivisorMapCorrespondence1}.  Thus $\phi(F^e_* 1) = \lambda \in H^0(X, \O_X) = k$ is a constant.  If we can show that $\lambda \neq 0$, then by rescaling $\phi$ we are done.
\item[(b)]  The value of $\phi(F^e_* 1)$ can be detected at any point.  In particular, we can try to compute it at the stalk of $x \in X$.
\end{itemize}

For simplicity, we denote the stalk at $x$ by $R := \O_{X,x}$ and we use $\bm$ to denote the maximal ideal.

Fix $\phi$ corresponding to $D$ as in \autoref{eq.DivisorMapCorrespondence1} and consider $\phi_x : F^e_* R \to R$.  Suppose that
\[
D|_{\Spec R} = V(f_1^{p^e-1} \cdots f_d^{p^e - 1}) = V(f)
 \]
 where the $f_i$ are the local equations for $D_i$ near $x$.

 Set $\widehat{R}$ to be the completion of $R = \O_{X,x}$.  We know that $\phi$ corresponds to $D$, so it can be factored as:
 \[
 F^e_* \O_X( (1-p^e)K_X - D) \hookrightarrow F^e_* \O_X( (1-p^e)K_X) \to \O_X.
 \]
 Taking the completion of this factorization, we obtain:
\[
\xymatrix@C=40pt{
F^e_* \widehat{R} \ar@/_4ex/@{->>}[rr]_{\widehat\phi} \ar@{^{(}->}[r]^{\cdot (F^e_* f)} & F^e_* \widehat{R} \ar[r]^-{\psi} & \widehat{R}
}
\]
By construction, $\psi$, viewed as an element of $M = \Hom(F^e_* \widehat{R} , \widehat{R})$, generates $M$ as an $F^e_* \widehat{R}$-module (use \autoref{ex.TraceMapInducesZeroDivisor}).

On the other hand, $\widehat{R} = k\llbracket f_1, \dots, f_d \rrbracket$ and so the map $\Psi : F^e_* \widehat{R} \to \widehat{R}$ which sends $f = f_1^{p^e-1} \cdots f_d^{p^e -1}$ to $1$ and the other basis monomials $\{ f_1^{a_1} \cdots f_d^{a_d} \neq f \,|\, 0 \leq a_i \leq p^e - 1\}$ to zero also generates $M$ as an $F^e_* \widehat{R}$-module by \autoref{ex.ZeroDivisorOnAn}.

It follows that $\psi(F^e_* \blank) = \Psi(F^e_* (c \cdot \blank) )$ for some invertible element $c \in \widehat{R}$.  But notice that $c$ is invertible, so it has a non-zero constant term $c_0 \in k$ where $c = c_0 + c'$, $c' \in \langle f_1, \dots, f_d \rangle_{\widehat{R}}$.  Thus
\[
\begin{array}{rcl}
\lambda & = & \phi_x(F^e_* 1)\\
& = & \widehat\phi(F^e_* 1)\\
&  = & \psi(F^e_* f) \\
&  = & \Psi(F^e_* (c \cdot f)) \\
&  = & \Psi(F^e_* (c_0 \cdot f)) + \Psi(F^e_* (c' \cdot f)) \\
&  = & c_0^{1/p^e} + \Psi(F^e_* (c'\cdot f)).
\end{array}
\]
But $\Psi(F^e_* (c'\cdot f) ) \in \langle f_1, \dots, f_d \rangle_{\widehat{R}}$ by our choice of $\Psi$ (note that $c' \cdot f \in \langle f_1^{p^e}, \dots, f_d^{p^e} \rangle$).  Since $c_0^{1/p^e} + \Psi(F^e_* (c'\cdot f) ) = \lambda \in k$ is a constant, we see that $\Psi(F^e_* (c'\cdot f) ) = 0$.  Thus $\lambda = c_0^{1/p^e} \neq 0$ as desired.
\end{proof}

\begin{remark}
A more general, simpler and more conceptual version of the above result is described in \autoref{ex.SimpleFSplitProof} in the next section.  We lack the language to describe it here however.
\end{remark}

Now we study the behavior of Frobenius splittings under maps between varieties.  We will study some complementary constructions later in \autoref{sec.ChangeOfVariety}.

\begin{theorem}\cite{HochsterRobertsFrobeniusLocalCohomology, MehtaRamanathanFrobeniusSplittingAndCohomologyVanishing}
\label{thm.ImageOfFSplitIsFSplit}
Suppose that $\pi : Y \to X$ is a map of varieties such that $\O_X \to \pi_* \O_Y$ splits as a map of $\O_X$-modules (for example, if $\pi_* \O_Y = \O_X$).  Then if $Y$ is Frobenius split, so is $X$.
\end{theorem}

Before proving the theorem, we point out just how common the condition that $\O_X \to \pi_* \O_Y$ splits is.  Indeed, if $\pi : Y \to X$ is a proper surjective map between normal varieties with connected fibers, then $\pi_* \O_Y = \O_X$.  Alternately, if $\pi : Y \to X$ is proper, dominant, generically finite, $Y$ and $X$ are normal, and $p$ does not divide $[\sK(Y) : \sK(X)] = n$, then the normalized field trace ${1 \over n} \Tr : \sK(Y) \to
\sK(X)$ restricts to a map $\pi_* \O_Y \to \O_X$ which sends $1$ to $1$.

\begin{proof}[Proof of \autoref{thm.ImageOfFSplitIsFSplit}]
Set $\phi : F^e_* \O_Y \to \O_Y$ to be the Frobenius splitting of $Y$ and fix $\alpha : \pi_* \O_Y \to \O_X$ to be the splitting of $i : \O_X \to \pi_* \O_Y$.  Pushing down $\phi$ we obtain:
\[
(\pi_* \phi) : \pi_* F^e_* \O_Y \to \pi_* \O_Y
\]
Now, we simply form the composition:
\[
F^e_* \O_X \xhookrightarrow{F^e_* i} F^e_* \pi_* \O_Y = \pi_* F^e_* \O_Y \xrightarrow{\pi_* \phi} \pi_* \O_Y \xrightarrow{\alpha} \O_X.
\]
By chasing through the composition, we see that $F^e_* 1$ is sent to $1$ and that $X$ is $F$-split.
\end{proof}

\subsection{Exercises}

\begin{exercise}
\label{ex.FrobeniusSplitOfIterate}
Prove Lemma \ref{lem.FrobeniusSplitOfIterate}.  \\ \emph{Hint:} Compose Frobenius and Frobenius splittings by using the \emph{functor} $F^e_*$.
\end{exercise}

\begin{exercise}
\label{ex.ProofOfFSplitLemma}
Prove Lemma \ref{lem.RFrobeniusSplitIfAndOnlyIfEvaluationAt1Surjects}.
\end{exercise}

\begin{exercise}
\label{ex.FSplitAreWeaklyNormal}
A domain $R$ containing a field of characteristic $p > 0$ is said to be \emph{weakly normal} if any $r \in K(R)$ satisfying $r^p \in R$ also satisfies $r \in R$ as well, see \cite[Lemma 3]{YanagiharaWeaklyNormal} and \cite[Section 3]{VitulliWeakNormalitySeminormalitySurvey}.  Show that any $F$-pure/split $R$ is weakly normal.  You can find a solution in \cite[Proposition 1.2.5]{BrionKumarFrobeniusSplitting}, \cf \cite[Proposition 5.31]{HochsterRobertsFrobeniusLocalCohomology}.
\end{exercise}

\begin{exercise}
Suppose that $X$ is Frobenius split relative to a Cartier divisor $D$ such that $X \setminus D$ is Cohen-Macaulay.  Prove that $X$ is Cohen-Macaulay.
\vskip 3pt
\emph{Hint: } Working locally we may assume that $X = \Spec R$ and $D = V(f)$.  Fix a maximal ideal $\bm \in \Spec R$ and consider the composition $H^i_{\bm}(R) \to H^i_{\bm}(F^e_* R) \xrightarrow{\cdot (F^e_* f)} H^i_{\bm}(F^e_* R)$ recalling that a variety can be proven to be Cohen-Macaulay by examining its local cohomology modules as in \cite[Chapter III, Exercises 3.3 and 3.4]{Hartshorne}.
\end{exercise}

\begin{exercise}
\label{ex.SurjectiveMapImpliesLocalSplitting}
Suppose that $F^e_* R \cong R \oplus M$ as $R$-modules where $M$ is some arbitrary $R$-module.  Prove that $R$ is Frobenius split.  More generally, prove the same result if there is any surjective map $F^e_* R \to R$.
\end{exercise}

\begin{exercise}
\label{ex.FSplitAtPointsImpliesFSplit}
Suppose that $X = \Spec R$ is an affine variety and suppose that for every maximal ideal $\bm \in \Spec R$, we have that $R_{\bm}$ is $F$-split.  Prove that $X$ is $F$-split.
\vskip 3pt
\emph{Hint: } The given splittings definitely do not glue.  However consider the evaluation-at-1 map $\Hom_R(F_* R, R) \to R$.
\end{exercise}

\begin{exercise}[Toric varieties]
Suppose that $X$ is a normal toric variety.  Consider the map $\Psi : F_* \O_X \to \O_X$ defined as follows.  We define
\[
\Psi(F_* {\bf x}^{\lambda}) = \left\{ \begin{array}{cl} {\bf x}^{\lambda / p} & \text{ if $\lambda/p$ has integer entries}\\ 0 & \text{ otherwise} \end{array}\right.
\]
acting on each affine toric chart (where ${\bf x}^{\lambda}$ is a monomial).
Show that this induces a Frobenius splitting on $X$ which compatibly splits all the torus invariant divisors.  What is the $\Delta_{\Psi}$ (as defined as in \autoref{eq.MapBijectionWithQDivisors1})?
\end{exercise}

\begin{exercise}[Affine section rings]
\label{ex.FSplitSectionRings}
Suppose that $X$ is a projective algebraic variety with ample line bundle $\sA$.  Consider
\[
S := \bigoplus_{i \in \bZ} H^0(X, \sA^i),
\]
the section right with respect to $\sA$.  Prove that $X$ is Frobenius split if and only if $S$ is Frobenius split.  For additional discussion of related topics, see \cite{SmithGloballyFRegular}.
\end{exercise}

\begin{exercise}
\label{ex.PropertiesOfCompatSplitSubvarieties}
Prove \autoref{prop.PropertiesOfCompatiblySplitSubvarieties}.  \vskip 3pt \emph{Hint: } For part (a), use a diagram similar to the one in \autoref{lem.NormalCrossingsAreFSplit}.  For solutions, see \cite[Chapter 1]{BrionKumarFrobeniusSplitting}.
\end{exercise}

\begin{starexercise}
\label{ex.VanishingForFrobeniusSplit}
Prove \autoref{thm.VanishingForFrobeniusSplit}(d). \vskip 3pt
\emph{Hint: } This is somewhat involved.  There exists a Cartier divisor $B$ such that $\sL^n(-B)$ is ample for all $n \gg 0$ since $\sL$ is big and nef.  For some $m \gg 0$, we also know that $mD + B$ is still ample.  First show that $X$ is $r$-Frobenius split relative to $mD + B$ for some integer $r \gg 0$ (this is hard).  Then notice we have a composition
\[
\omega_X \tensor \sL \xleftarrow{T} F^r_* (\omega_X \tensor \sL) \hookleftarrow F^r_* (\omega_X(-B) \tensor \sL) \hookleftarrow F^r_* (\omega_X(-mD - B) \tensor \sL) \leftarrow \omega_X
\]
which is an isomorphism (we type this with the arrows going backwards to suggest that this arises by duality).  Now, by composing the map  $\omega_X \leftarrow F^r_* (\omega_X(-B) \tensor \sL)$ with itself as in \autoref{eq.ComposeMaps2}, we can obtain the desired vanishing.  For a solution, see \cite[Theorem 6.8]{SchwedeSmithLogFanoVsGloballyFRegular}.
\end{starexercise}

\section{Frobenius non-splittings }
\label{sec.NonSplittings}
Our goal in this section is to develop a theory for $p^{-1}$-linear maps generalizing the theory of Frobenius split varieties demonstrated in the previous section. First we start with a definition.

\begin{definition}
Suppose that we are given a line bundle $\sL$ on a variety $X$.  Consider an $\O_X$-linear map $\phi : F^e_* \sL \to \O_X$.  We say that an ideal $J$ is \emph{$\varphi$-compatible} if we have that $\phi(F^e_* (J \cdot \sL)) \subseteq J$.  If $Y = V(J) \subseteq X$, then we say that $Y$ is \emph{$\phi$-compatible} if $J$ is.
\end{definition}

For example, if $\sL = \O_X$ and $\phi$ is a Frobenius splitting, then any $\phi$-compatibly split ideal is $\phi$-compatible.  We also have a slight variation on this definition.

\begin{definition}
\label{def.FPureCenter}
Given $\Delta$ corresponding to  $\phi : F^e_* \sL \to \O_X$ as in \autoref{eq.DivisorMapCorrespondence2}.
A subvariety $Y \subseteq X$ is called an \emph{$F$-pure center of $(X, \Delta)$} if $Y$ is $\phi$-compatible and $\phi_{\eta}$ is surjective where $\eta$ is the generic point of $Y$.
\end{definition}

\begin{lemma}
\label{lem.RestrictNearSplitting}
If $J \subseteq \O_X$ is an ideal sheaf, then $J$ is $\phi : F^e_* \sL \to \O_X$ compatible if and only if $\phi$ induces a map $\phi_Y : F^e_* (\sL|_Y) \to \O_Y$.
\end{lemma}
\begin{proof}
Left as an exercise to the reader \autoref{ex.RestrictNearSplitting}.
\end{proof}

We explore compatibility after composing maps as in \autoref{eq.ComposeMaps2}.

\begin{lemma}
Suppose that $J \subseteq \O_X$ is $\phi : F^e_* \sL \to \O_X$-compatible.  Then $J$ is \[
\phi^n : F^{ne}_* \sL^{p^{e(n-1)} + \dots + 1} \to \O_X
\]
 compatible for all $n > 0$.  Conversely, suppose that $\phi$ is surjective.  If $J$ is $\phi^n$-compatible then $J$ is also $\phi$-compatible.
\end{lemma}\begin{proof}
The statement is local so we may as well only check this at the stalks $\O_{X,x}$ and in particular assume that $\sL \cong \O_{X,x}$.  The first statement is obvious and will be left to the reader.  For the second statement, we sketch the idea of the proof.

The first step is to show that any $J \subseteq \O_{X,x}$ which is $\phi : F^e_* \O_{X,x} \twoheadrightarrow \O_{X,x}$-compatible is also radical, see \autoref{ex.WhichPropertiesOfSplittingsExtendToSurjectiveMaps}.  One can then show it is sufficient to verify the statement at the minimal primes of $J$.  In particular, we can assume that $J$ is the maximal ideal of $\O_{X,x}$ by localizing.

Now then, suppose that $J$ is $\phi^n$ compatible but not $\phi$-compatible.  Then $\phi(F^e_* J) = \O_{X,x}$ (since otherwise, it would be in the maximal ideal, which coincides with $J$).  But then it is easy to see that $\phi^2(F^{2e}_* J) = \phi(F^e_* \phi(F^e_* J)) = \O_{X,x}$ as well.  Continuing in this way, we obtain a contradiction.
\end{proof}

We also generalize the notion of $F$-pure to non-Frobenius splittings and to pairs.

\begin{definition}
\label{def.FPurePair}
Suppose that $X$ is a normal variety and that $\Delta$ is a $\bQ$-divisor such that
\begin{equation}
\label{eq.KXDeltaIsQCartierWithIndex} \tag{$\dagger$}
\text{$K_X + \Delta$ is a $\bQ$-Cartier divisor with index not divisible by $p$.}
\end{equation}
We say that $(X, \Delta)$ is \emph{sharply $F$-pure} if the map $\phi : F^e_* \sL \to \O_X$, corresponding to $\Delta$ as in \autoref{eq.MapBijectionWithQDivisors2} is surjective as a map of $\O_X$-modules.

If we do not satisfy \eqref{eq.KXDeltaIsQCartierWithIndex}, then we say that $(X, \Delta)$ is \emph{sharply $F$-pure} if for every point $x \in X$, there exists a neighborhood $U$ of $x \in X$ and a divisor $\Delta_U$ on $U$ such that $\Delta_U \geq \Delta|_U$ and such that $(U, \Delta_U)$ is sharply $F$-pure in the above sense.
\end{definition}

It is an exercise below, \autoref{ex.NotionsOfFPureCoincide}, that the definition of sharply $F$-pure above and the definition given in \autoref{def.FSplitFpure} coincide.

\subsection{Global considerations}
\label{subsec.GlobalConsiderationsForNonSplittings}
In this subsection, we briefly demonstrate that some of the \emph{global} methods from the Frobenius splitting section can still bear fruit, even if the actual vanishing theorems do not hold.

Our first goal is to consider a generalization of a proof due to D.~Keeler \cite{KeelerFujita} (also independently obtained by N.~Hara [unpublished]).  Related results were first proven by \cite{SmithFujitasFreeness} and also \cite{HaraACharacteristicPAnalogOfMultiplierIdealsAndApplications}.  Before doing that, we recall a Definition and a Lemma.

\begin{definition}[Castelnuovo-Mumford Regularity]\cite[Section 1.8]{LazarsfeldPositivity1}
Suppose that $\sF$ is a coherent sheaf on a projective variety $X$ and that $\sA$ is a globally generated ample divisor on $X$.  Then $\sF$ is called \emph{$0$-regular (with respect to $\sA$)} if $H^i(X, \sF \tensor \sA^{-i}) = 0$ for all $i > 0$.
\end{definition}

\begin{lemma}[Mumford's Theorem] \cite[Theorem 1.8.5]{LazarsfeldPositivity1}
If $\sF$ is $0$-regular with respect to a globally generated ample line bundle $\sA$, then $\sF$ is globally generated.
\end{lemma}

Now we are in a position to prove that certain sheaves are globally generated.

\begin{theorem} \cite{KeelerFujita,SchwedeCanonicalLinearSystem}
\label{thm.GlobalGenerationOfImages}
Suppose that $\phi : F^e_* \sL \to \O_X$ is a \emph{surjective} $\O_X$-linear map and $\sL$ is a line bundle.  Additionally suppose that $\sA$ is a globally generated ample line bundle and that $\sM$ is any other line bundle such that $\sL \tensor \sM^{p^e - 1}$ is ample (for example, if $\sL$ is itself ample, then we may take $\sM = \O_X$).  In this case, the line bundle
\[
\sM \tensor \sA^{\dim X}
\]
is globally generated.
\end{theorem}
\begin{proof}
Choose $n \gg 0$.  Then we have a surjective map:
\[
\phi^n : F^{ne}_* \sL^{p^{(n-1)e} + \dots + p^e + 1} \to \O_X
\]
from \autoref{eq.ComposeMaps2}.  Twisting by $\sM \tensor \sA^{\dim X}$ we obtain a surjective map:
\[
F^{ne}_* (\sL^{p^{(n-1)e} + \dots + p^e + 1} \tensor \sM^{p^{ne} } \tensor \sA^{p^{ne} \dim X}) \rightarrow \sM \tensor \sA^{\dim X}.
\]
It is sufficient to show that the left side is globally generated as an $\O_X$-module since then the right side is a quotient of a globally generated module and thus globally generated itself.  Note it is definitely \emph{not} sufficient to show that the left side is globally generated as an $F^{ne}_* \O_X$-module.  We will proceed by proving that the left side is $0$-regular as an $\O_X$-module.
Note
\[
\sL^{p^{(n-1)e} + \dots + p^e + 1} \tensor \sM^{p^{ne} } = (\sL \tensor \sM^{p^e - 1})^{p^{(n-1)e} + \dots + p^e + 1} \tensor \sM.
\]
But now we have
\[
\begin{array}{rl}
& H^i\Big(X, F^{ne}_* \big( (\sL \tensor \sM^{p^e - 1})^{p^{(n-1)e} + \dots + p^e + 1} \tensor \sM \tensor \sA^{p^{ne} \dim X}\big) \tensor \sA^{-i}\Big)\\
= & H^i\Big(X, F^{ne}_* \big( (\sL \tensor \sM^{p^e - 1})^{p^{(n-1)e} + \dots + p^e + 1} \tensor \sM \tensor \sA^{p^{ne} (\dim X - i)}\big)\Big)\\
= & H^i\Big(X, F^{ne}_* \big( (\sL \tensor \sM^{p^e - 1} \tensor \sA^{(\dim X - i)(p^{e} - 1)})^{p^{(n-1)e} + \dots + p^e + 1} \tensor (\sM \tensor \sA^{\dim X - i})\big)\Big).
\end{array}
\]
We already have the vanishing for $i > \dim X$.  Now the $F^{ne}_*$ does not effect the vanishing or non-vanishing of the cohomology since it doesn't change the underlying sheaf of Abelian groups.  Therefore, the above cohomology groups vanish by Serre vanishing, since $\sL \tensor \sM^{p^e - 1} \tensor \sA^{(\dim X - i)(p^{e} - 1)}$ is ample and each of the \emph{finitely many} $\sM \tensor \sA^{\dim X - i}$ are coherent sheaves.
\end{proof}

\begin{example}
If $X$ is smooth (or even $F$-pure), then there is always a \emph{surjective} map $F^e_* \O_X( (1-p^e) K_X) \to \O_X$.  It follows that if $M$ is a divisor such that $M - K_X$ is ample, and $\sA$ is any globally generated ample line bundle, then $\O_X(M) \tensor \sA^{\dim X}$ is globally generated.
\end{example}

\begin{remark}
It is worth pointing out that not only is $\sM \tensor \sA^{\dim X}$ globally generated, one even has that it is globally generated by the image of the map
\[
H^0\Big(X, F^{ne}_* \big(\sL^{p^{(n-1)e} + \dots + p^e + 1} \tensor \sM^{p^{ne} } \tensor \sA^{p^{ne} \dim X}\big)\Big) \rightarrow H^0\big(X, \sM \tensor \sA^{\dim X}\big).
\]
This special sub-vector space of global sections also behaves well with respect to restriction to compatible subvarieties as we shall see shortly.\end{remark}

Similar arguments to those in the proof \autoref{thm.GlobalGenerationOfImages} also yield the following result.

\begin{proposition}\cite{SmithFujitasFreeness,HaraACharacteristicPAnalogOfMultiplierIdealsAndApplications,KeelerFujita}
\label{prop.FujitaGlobalGen}
If $X$ is any $F$-pure variety, $\sA$ is a globally generated ample line bundle and $\sM$ is any other ample line bundle then
\[
\omega_X \otimes \sA^{\dim X} \otimes \sM
\]
is globally generated.
\end{proposition}
\begin{proof}
The proof is left to the reader in \autoref{ex.FujitaGlobalGen}.
\end{proof}

Finally, we also remark that compatible ideals also play a special role with regards to lifting of sections.

\begin{theorem}
Suppose that $\phi : F^e_* \sL \to \O_X$ is an $\O_X$-linear map and that $J \subseteq \O_X$ is $\phi$-compatible.  Set $Y = V(J)$ and set $\phi_Y : F^e_* (\sL|_Y) \to \O_Y$ to be the map $\phi$ restricted to $Y$ as in \autoref{lem.RestrictNearSplitting}.  Suppose that $\sH$ is a line bundle on $X$ such that $\sH^{p^e - 1} \tensor \sL$ is ample and also such that the map induced by $\varphi_Y$
\begin{equation}
\label{eq.ExtendSections}
H^0\big(Y, F^{ne}_* ((\sL^{p^{(n-1)e} + \dots + p^e + 1} \tensor \sH^{p^{ne}} )|_Y) \big) \xrightarrow{\gamma} H^0(Y, \sH|_Y)
\end{equation}
is non-zero for some $n \gg 0$.  Then $H^0(X, \sH) \neq 0$ as well.  Even more, the sections in the image of $\gamma$ all extend to sections on $H^0(X, \sH)$.
\end{theorem}

Before starting the proof, let us note some conditions under which the map $\gamma$ is non-zero.  For example, if $\sL|_Y = \O_Y$ and $\phi_Y$ is a Frobenius splitting, then $\gamma$ is in fact surjective (for example, if $Y$ is a point and $\phi_Y$ is non-zero).  Alternately, if $\phi_Y$ is surjective and also $\sH|_Y = \sA^{\dim Y} \tensor \sM$ where $\sA$ is a globally generated ample line bundle on $Y$ and $\sM^{p^e - 1} \tensor \sL|_Y$ is ample on $Y$, then we can apply \autoref{thm.GlobalGenerationOfImages}.  In the case that $Y$ is a curve, see \autoref{ex.ImagesTraceForCurves}

\begin{proof}
We fix $n \gg 0$, for simplicity of notation set $\eta = p^{(n-1)e} + \dots + p^e + 1$ and consider the following diagram.
\[
\xymatrix@C=12pt{
H^0\big(X, F^{ne}_* (\sL^{\eta} \tensor \sH^{p^{ne}} \big) \ar[d]_{\phi} \ar[r] & H^0\big(Y, F^{ne}_* ((\sL^{\eta} \tensor \sH^{p^{ne}} )|_Y)  \big) \ar[d]^{\phi_Y} \ar[r] & H^1\big(X, F^{ne}_* (J \tensor \sL^{\eta} \tensor \sH^{p^{ne}} )\big) \ar[d] \\
H^0\big(X, \sH) \ar[r] & H^0(Y, \sH|_Y) \ar[r] & H^1(X, J \tensor \sH).
}
\]
However, note that
\[
H^1\big(X, F^{ne}_* (J \tensor \sL^{\eta} \tensor \sH^{p^{ne}} )\big) = H^1\big(X, F^{ne}_* (J \tensor \sH \tensor (\sL \tensor \sH^{p^e -1})^{\eta}\big) = 0
\]
by Serre vanishing since the $F^{ne}_*$ does not effect the underlying sheaf of Abelian groups.
\end{proof}

\subsection{Fedder's Lemma}
\label{subsec.Fedder}

We now delve into the \emph{local} theory of $p^{-e}$-linear maps and in particular state \emph{Fedder's Lemma}.  This is a particularly effective tool for explicitly writing down these maps and also for identifying which of them are surjective.

Suppose that $S = k[x_1, \dots, x_n]$ and $R = S/I$ for some ideal $I \subseteq R$.
The point is that if $R = S/I$, then maps $\bar\phi : F^e_* R \to R$ come from maps $\phi : F^e_* S \to S$, which Fedder's Lemma \emph{precisely} identifies.  Set $\Phi_S : F^e_* S \to S$ to be the map which generates $\Hom_S(F^e_* S, S)$ as an $F^e_* S$-module as identified in \autoref{ex.ZeroDivisorOnAn}.  Recall that $\Phi_S$ sends the monomial $F^e_* (x_1^{p^e - 1} \dots x_n^{p^e - 1})$ to $1$ and all other basis monomials to zero.

\begin{lemma}[Fedder's Lemma] \cite[Lemma 1.6]{FedderFPureRat}
\label{lem.FeddersLemma}
With $S \supseteq I$, $R$ and $\Phi_S$ as above, then
\[
\left\{ \begin{array}{c} \text{Maps $\phi \in \Hom_S(F^e_* S, S)$} \\ \text{compatible with $I$ } \end{array}\right\} = \left\{ \begin{array}{c} \phi \;|\; \phi(F^e_* \blank) = \Phi_S(F^e_* (z \cdot \blank) ), \text{for some } z \in I^{[p^e]} : I \end{array}\right\}.
\]
More generally, there is an isomorphism of $S$-modules:
\[
\Hom_R(F^e_* R, R) \longleftrightarrow {\big(F^e_* (I^{[p^e]} : I)\big) \cdot \Phi_S \over \big(F^e_* I^{[p^e]}\big) \cdot \Phi_S}
\]
induced by restricting $\psi \in (F^e_* (I^{[p^e]} : I) ) \cdot \Phi_S \subseteq \Hom_S(F^e_* S, S)$ to $R = S/I$ as in \autoref{lem.RestrictNearSplitting}.

Finally, for any point $\bq \in V(I) \subseteq \Spec S$, there exists a map $\phi \in \Hom_R(F^e_* R, R)$ which is surjective at $\bq/I \in \Spec R$ if and only if $I^{[p^e]} : I \nsubseteq \bq^{[p^e]}$.  In other words, $R$ is $F$-pure in a neighborhood of $\bq$ if and only if $I^{[p]} : I \nsubseteq \bq^{[p]}$.
\end{lemma}
\begin{proof}
There are a lot of statements here.  First we notice that any map of the form $\phi(F^e_* \blank) = \Phi_S(F^e_* (z \cdot \blank) )$ for some  $z \in I^{[p^e]} : I$ is clearly compatible with $I$ since
\[
\Phi_S(F^e_* (z \cdot I)) \subseteq \Phi_S(F^e_* (I^{[p^e]} : I) \cdot I) \subseteq \Phi_S(F^e_* I^{[p^e]}) = I \cdot \Phi_S(F^e_* S) = I.
\]
This gives us the containment $\supseteq$ in the first equality.  For the other containment, we first prove the following claim.
\begin{claim} For ideals $I, J \subseteq S$ we have
\[
\Phi_S(F^e_* J) \subseteq I
\]
if and only if $J \subseteq I^{[p^e]}$.
\end{claim}
\begin{proof}[Proof of claim]
Certainly the \emph{if} direction is obvious, so suppose then that $\Phi_S(F^e_* J) \subseteq I$.  This implies that $\phi(F^e_* J) \subseteq I$ for every $\phi \in \Hom_S(F^e_* S, S)$ since $\Phi_S$ generates that set as an $F^e_* S$-module.   But $F^e_* S$ is a free $S$-module of rank $p^{en}$, so we see that
\[
F^e_* J \subseteq \underbrace{I \oplus \dots \oplus I}_{p^{ne}-\text{times}}
\]
since we could take the $\phi$ as the various projections.  Now, $I \oplus \dots \oplus I = I \cdot (F^e_* S) = F^e_* I^{[p^e]}$.  This proves the claim.
\end{proof}
Now we return to the proof of Fedder's Lemma.  We observe that if $\phi(F^e_* \blank) = \Phi_S(F^e_* (z \cdot \blank))$ is $I$-compatible, then $z \cdot I \subseteq I^{[p^e]}$ by the claim, which proves that $z \in I^{[p^e]} : I$ and so the equality is proven.

Now we come to the bijection.  We certainly have a natural map
\[
\Lambda : (F^e_* (I^{[p^e]} : I)) \cdot \Phi_S \to \Hom_R(F^e_* R, R)
\]
induced by sending $F^e_* z$ first to $(F^e_* z) \cdot \Phi_S( \blank) = \Phi_S((F^e_* z) \cdot \blank))$ and then second, inducing a map in $\Hom_R(F^e_* R, R)$ as in \autoref{lem.RestrictNearSplitting}.  The kernel of $\Lambda$ is $(F^e_* I^{[p^e]}) \cdot \Phi_S$ by the claim, and so we only need to show that this map is surjective.

Given $\phi \in \Hom_R(F^e_* R, R) = \Hom_S(F^e_* R, R)$, consider the following diagram of $S$-linear maps where the horizontal maps are the canonical surjections:
\[
\xymatrix{
F^e_* S \ar@{.>}[d]_{\exists \psi} \ar@{->>}[r] & F^e_* (R/I) \ar[d]^{\phi}\\
S \ar@{->>}[r] & (R/I)
}
\]
Because $F^e_* S$ is a free (and so projective) $S$-module, the dotted map $\psi$ exists.  By construction, $\psi$ is compatible with $I$.  By the earlier parts of the theorem, $\psi$ corresponds to a $z \in I^{[p^e]} : I$ which restricts to $\phi$, completing the proof of the bijection.

The last part of the theorem is left as an exercise to the reader.
\end{proof}

\begin{remark}[Regular local rings are fine]
\label{rem.RegLocalRingsAreFineForFedder}
The proof given above goes through without change if one assumes that $S$ is a regular local\footnote{or even semilocal} ring \emph{instead} of assuming that $S$ is a polynomial ring.
\end{remark}

One of the most important corollaries of this is the following.

\begin{corollary}
\label{cor.HypersurfaceProof}
Given $f \in k[x_1, \dots, x_n] = S$, then $S/\langle f \rangle$ is $F$-split in a neighborhood of the origin if and only if $f^{p - 1} \notin \langle x_1^p, \dots, x_n^p \rangle$.
\end{corollary}
\begin{proof}
Note that $S/\langle f \rangle = R$ is $F$-split if and only if there exists a surjective map $\phi \in \Hom_R(F^e_* R, R)$ by \autoref{ex.SurjectiveMapImpliesLocalSplitting}.  The result then follows from Fedder's Lemma since $\langle f^{p} \rangle : \langle f \rangle = \langle f^{p - 1} \rangle$.
\end{proof}

We now apply Fedder's Lemma in a number of examples of hypersurface singularities:

\begin{example}
We consider $S$ to be a polynomial ring in the following examples.
\begin{description}
\item[Node]  Consider the ring $S = k[x,y]$ and $R = k[x,y]/\langle xy \rangle$.  Then $R$ is $F$-split near the origin since
\[
(xy)^{p-1} = x^{p-1}y^{p-1} \notin \langle x^p, y^p \rangle.
\]
\item[Cusp]  Consider the ring $S = k[x,y]$ and $R = k[x,y]/\langle x^3 - y^2 \rangle$.  Then we claim that $R$ is \emph{not} $F$-split near the origin since (for odd primes).  To see this observe that for some constant $c$
    \[
    (x^3 - y^2)^{p-1} = x^{3(p-1)} + \dots + c x^{3(p-1)/2} y^{p-1} + \dots + x^{2(p-1)} \in \langle x^p, y^p \rangle.
    \]
    The computation for $p = 2$ is similar (or follows from the work below).
\item[Pinch point]  Consider the ring $S = k[x,y,z]$ and $R = k[x,y,z]/\langle xy^2 - z^2 \rangle$.  If $p \neq 2$, this is $F$-split near the origin since
    \[
    \begin{array}{rl}
    & (xy^2 - z^2)^{p-1} \\
    = & x^{p-1} y^{2(p-1)} + \dots + {p-1 \choose (p-1)/2} (-1)^{(p-1)/2} x^{(p-1)/2} y^{p-1} z^{p-1} + \dots + z^{(p-1)/2} \\
    \notin & \langle x^p, y^p, z^p \rangle,
    \end{array}
    \]
    noting that $p$ does not divide ${p-1 \choose (p-1)/2}$.
\item[Characteristic $2$]  If $R = k[x_1, \dots, x_n]/\langle f \rangle$ and $\Char k = 2$, then $R$ is $F$-split near the origin if and only if $f \notin \langle x_1^2, \dots, x_n^{2} \rangle$.  In particular, it is immediate that the cusp and the pinch point are also not $F$-split near the origin in characteristic $2$.
\item[Characteristic $3$]  Just like characteristic $2$, if $R = k[x_1, \dots, x_n]/\langle f \rangle$ and $\Char k = 3$, then $R$ is $F$-split near the origin if and only if $f^2 \notin \langle x_1^3, \dots, x_n^{3} \rangle$.
\end{description}
\end{example}

Finally, we point out that complete intersection singularities are nearly as easy to compute as hypersurfaces.

\begin{proposition}
\label{prop.FedderCriterionForCI}
Suppose that $f_1, \dots, f_m \subseteq \langle x_1, \dots, x_n \rangle \subseteq k[x_1, \dots, x_n] = S$ is a regular sequence.\footnote{This means that $f_i$ is not a zero divisor in $S/\langle f_1, \dots, f_{i-1} \rangle$ for all $i > 0$.}  Set $I = \langle f_1, \dots, f_m \rangle$. Then
\[
(I^{[p^e]} : I) = \langle  f_1^{p^e-1} \cdots f_n^{p^e-1} \rangle + I^{[p^e]}
\]
In particular, $S/I$ is $F$-split near the origin $\bm = \langle x_1, \dots, x_n \rangle$ if and only if the product
\[
f_1^{p^e-1} \cdots f_n^{p^e-1} \notin \bm^{[p^e]}
\]
for some $e > 0$.
\end{proposition}
\begin{proof}
The containment $\supseteq$ is trivial.  The converse direction is left as \autoref{ex.FedderCriterionForCI}.
\end{proof}

\subsection{Exercises}

\begin{exercise}
\label{ex.RestrictNearSplitting}
Prove \autoref{lem.RestrictNearSplitting}.
\end{exercise}

\begin{exercise}
\label{ex.ImagesTraceForCurves}
Suppose that $C$ is a smooth curve and that $\sL$ is a line bundle of degree $\geq 2$.  Prove that the image of the map
\[
H^0\big(X, F^e_* (\omega_X \tensor \sL^{p^e} ) \big) \to H^0(X, \omega_X \tensor \sL)
\]
globally generates $\omega_X \tensor \sL$ for any $e \gg 0$.
\vskip 3pt
\emph{Hint: } Mimic the proof in \cite[Chapter IV, Proposition 3.1]{Hartshorne}.  For a solution, see \cite[Theorem 3.3]{SchwedeCanonicalLinearSystem}.
\end{exercise}

\begin{starexercise}
\label{ex.WhichPropertiesOfSplittingsExtendToSurjectiveMaps}
Consider a map $\phi : F^e_* \sL \to \O_X$ for some line bundle $\sL$ and $e > 0$.  Formulate analogs of the properties from \autoref{prop.PropertiesOfCompatiblySplitSubvarieties} and \autoref{cor.IntersectionsOfCompatiblySplitVarietiesAreReduced} for such a map (and $\phi$-compatible ideals / subvarieties).  Which of these properties
hold for all $\phi$?  Which hold for surjective $\phi$?  Prove those that do and give counterexamples to those that do not.  Some of the answers can be found in \cite{SchwedeCentersOfFPurity,SchwedeFAdjunction}.
\end{starexercise}

\begin{starexercise}
Suppose that $X$ is a Frobenius split normal variety.  Suppose that $X$ embeds into $\bP^n$ as a closed subvariety.  Prove that $X$ is compatibly $F$-split by a Frobenius splitting of $\bP^n$ if and only if the embedding $X \subseteq \bP^n$ is projectively normal, \cf \cite[Chapter II, Exercise 5.14]{Hartshorne}.
\vskip 3pt
\emph{Hint: } Projective normality can be detected by the difference between the affine cone and the section ring as in \autoref{ex.FSplitSectionRings}.  Develop then a ``graded variant'' of Fedder's Lemma that will allow you to prove the result.
\end{starexercise}

\begin{exercise}
\label{ex.NotionsOfFPureCoincide}
We can define $X$ to be $F$-pure if $(X, 0)$ is sharply $F$-pure in the sense of \autoref{def.FPurePair}.  Show that this coincides with the definition of $F$-pure given in \autoref{def.FSplitFpure}.
\end{exercise}

\begin{exercise}
Suppose that $\sL$ is an ample line bundle on a smooth variety $X$.  Prove that $H^0(X, F^e_* (\omega_X \tensor \sL^{mp^e})) \to H^0(X, \omega_X \tensor \sL^m)$ is surjective for all $m \gg 0$.  For one solution, see \cite[Lemma 3.1]{SchwedeCanonicalLinearSystem}.
\end{exercise}

\begin{exercise}
\label{ex.FujitaGlobalGen}
Use the method of \autoref{thm.GlobalGenerationOfImages} to prove \autoref{prop.FujitaGlobalGen}.
\vskip 3pt
{\emph{Hint: }} Dualize a local splitting $\O_U \to F_* \O_U \to \O_U$ to obtain a surjective map $T : F_* \omega_U \to \omega_U$.  Use $T$ instead of $\phi$ in the proof of \autoref{thm.GlobalGenerationOfImages}.
\end{exercise}

\begin{exercise}
Consider $\overline{\bF_5}[x,y,z] = S$ and $f = x^4 + y^4 + z^4$.  Consider the map $\Phi_S : F_* S \to S$ which sends $F_* x^4 y^4 z^4$ to $1$ and sends all the other monomials $x^i y^j z^k$ to $0$ for $0 \leq i,j,k \leq 4$ as in \autoref{ex.ZeroDivisorOnAn}.  Consider the map $\phi : F_* S \to S$ defined by
\[
\phi(F_* \blank) = \Phi_S(F_* (f^4 \cdot \blank)).
\]
\begin{itemize}
\item[(a)] Prove that $\langle f \rangle$ is $\phi$-compatible and let $\overline{\phi} : F_* R \to R$ be the induced map on $R = S/\langle f \rangle$ as in \autoref{lem.RestrictNearSplitting}.
\item[(b)*] Set $\bm = \langle x, y, z \rangle \in S$.  Fix $a,b,c \in \bF_{5^2} \setminus \bF_5$.  Show that $J = \bm^2 + \langle ax + by + cz \rangle$ is $\phi^2$-compatible.  However, show that $J$ is not $\phi$-compatible.
\end{itemize}
\end{exercise}

\begin{exercise}
With $\Phi_S$ as in \autoref{subsec.Fedder}, fix $f \in S$ and consider the map $\phi$ defined by the rule $\phi(F^e_* \blank) = \Phi_S(F^e_* (f \cdot \blank))$.  Show that $\phi$ is compatible with an ideal $J \subseteq S$ if and only if $f \in J^{[p^e]}: J$.
\end{exercise}

\begin{exercise}
\label{ex.FeddersLemma}
Complete the proof of Fedder's Lemma by proving the following.  For any point $\bq \in V(I) \subseteq \Spec S$, there exists a map $\phi \in \Hom_R(F^e_* R, R)$ which is surjective at $\bq/I \in \Spec R$ if and only if $I^{[p^e]} : I \nsubseteq \bq^{[p^e]}$.  In other words, $R$ is $F$-pure in a neighborhood of $\bq$ if and only if $I^{[p^e]} : I \nsubseteq \bq^{[p^e]}$.
\vskip 3pt
\emph{Hint: } Note that a map $\phi_{\bq} : F^e_* (R_{\bq}) \to R_{\bq}$ is surjective if and only if $\Image(\phi_{\bq}) \nsubseteq \bq R_{\bq}$.
\end{exercise}

\begin{exercise}
\label{ex.EasyFedderCheckForPairs}
Suppose that $X = \Spec R$ is a regular ring and $\Delta = {1 \over p^e-1} \Div_X(f)$ is a $\bQ$-divisor on $X$.  Show that $(X, \Delta)$ is sharply $F$-pure near a point $\bm \in \Spec R = X$ if and only if $f^{p^e - 1} \notin \bm^{[p^e]}$.
\vskip 3pt
\emph{Hint: } Use Fedder's lemma in the form of \autoref{rem.RegLocalRingsAreFineForFedder}.
\end{exercise}

\begin{exercise}
\label{ex.SimpleFSplitProof}
Suppose that $X$ is a proper variety and that $\phi : F^e_* \O_X \to \O_X$ is a map that is compatible with $\bm$, the ideal of a closed point $x \in X$.  Further suppose that $(X, \Delta_{\phi})$ is sharply $F$-pure in a neighborhood of $\bm$.  Prove that $0 \neq \phi(F^e_* 1) \in k$ and so in particular $X$ is $F$-split.  This generalizes \autoref{thm.MRFSplitProof} by the following argument.

Given a $D = (p^e - 1)D_1 + \dots + (p^e -1)D_d + G$ and $\phi$ as in \autoref{thm.MRFSplitProof}, set $\Delta = {1 \over p^e -1} D$.  Observe that $\bm_x$, the maximal ideal of $x$ is $\phi$-compatible since each $D_i$ is $\phi$-compatible, \cf \autoref{lem.NormalCrossingsAreFSplit}.  Use exercise \autoref{ex.EasyFedderCheckForPairs} to conclude that $\phi$ is surjective in a neighborhood of $x \in X$.

Obtain a new proof of \autoref{thm.MRFSplitProof} by combining the above.
\end{exercise}

\begin{exercise}
Suppose that $X = \Spec k\llbracket x,y \rrbracket$ where $k$ has characteristic $7$ and that $\Delta = {1 \over 2} \Div_X(y^2 - x^3) + {1 \over 3} \Div_X(x) + {1 \over 2} \Div_X(y)$.  Prove that $(X, \Delta)$ is sharply $F$-pure at the origin $\bm$ and also that if $\phi$ corresponds to $\Delta$, then $\bm$ is $\phi$-compatible.

Now suppose that $Y$ is a smooth projective variety with a $\bQ$-divisor $\Theta \geq 0$ such that
\begin{itemize}
\item{} $(p-1)(K_Y + \Theta) \sim 0$ and,
\item{} $(Y, \Theta)$ has a point $y \in Y$ analytically isomorphic to $(X, \Delta)$ above.
\end{itemize}
Show that $Y$ is Frobenius split using \autoref{ex.SimpleFSplitProof}.
\end{exercise}

\begin{exercise}
\label{ex.Normalization}
Suppose that $R$ is an integral domain with normalization $R^{\textnormal{N}}$ in $K(R)$, the field of fractions of $R$.  In this exercise, we will prove that every map $\phi:F^e_* R \to R$ induces an $R^{\textnormal{N}}$-linear map $\phi^{\textnormal{N}} : F^e_* R^{\textnormal{N}} \to R^{\textnormal{N}}$ which is compatible with the conductor ideal $\bc := \Ann_R(R^{\textnormal{N}})$, an ideal in both $R$ and $R^{\textnormal{N}} $.
We do this in two steps.

\begin{itemize}
\item[(a)]  Prove that $\phi$ is compatible with $\bc$ (when viewed as an ideal in $R$).
\item[(b)]  Notice that $\phi$ induces a map $\phi_0 : F^e_* K(R) \to K(R)$ by localization.  Prove that $\phi_0(R^{\textnormal{N}}) \subseteq R^{\textnormal{N}}$ which proves that we can take $\phi^{\textnormal{N}} = \phi_0|_{R^{\textnormal{N}} }$.
    \vskip 3pt
\emph{Hint: } Recall that $x \in K(R)$ is integral over $R$ if there exists a non-zero $c \in R$ such that $cx^n \in R$ for all $n \gg 0$, see \cite[Exercise 2.26]{HunekeSwansonIntegralClosure}.
\end{itemize}
\end{exercise}

\begin{exercise}
\label{ex.FedderCriterionForCI}
Prove \autoref{prop.FedderCriterionForCI}.
\vskip 3pt
\emph{Hint: } A very easy proof (pointed out to us by Alberto Fernandez Boix), follows from \cite[Corollary 1]{HartshorneAPropertyOfASequences}.
Alternately, the $\supseteq$ containment is easy.  For the reverse proceed by induction on the number of $f_i$.  Notice that $\Hom_{S/I}(F^e_* S/I, S/I)$ is a free $F^e_* S/I$-module of rank 1.  Thus a generator of that module corresponds to an element $h \in I^{[p^e]} : I$.

For a generalization to Gorenstein rings (instead of just complete intersections), see \cite[Corollary 7.5]{SchwedeFAdjunction}.
\end{exercise}

\begin{exercise}[Macaulay2 Fedder's criterion]
The following Macaulay2 code, written by Mordechai Katzman and available at
\begin{center}
{\tt http://katzman.staff.shef.ac.uk/FSplitting/}
\end{center}
can be quite useful.
\begin{quote}
{\tt
frobeniusPower=method();\\
\\
frobeniusPower(Ideal,ZZ) := (I,e) ->(\\
R:=ring I;\\
p:=char R;\\
local u;\\
local answer;\\
G:=first entries gens I;\\
if (\#G==0) then answer=ideal(0\_\,R) else answer=ideal(apply(G, u->u\^{ }(p\^{ }e)));\\
answer\\
);
}
\end{quote}
This takes an ideal $I$ and raises it to the $p^e$th Frobenius power, $I \mapsto I^{[p^e]}$.  Using this as a starting place, implement within Macaulay2 a method which determines whether a given ring is $F$-pure near the origin.  Check your method against the following examples:
\begin{itemize}
\item[(a)]  $R = k[x,y,z]/\langle xy, xz, yz \rangle$ in whatever characteristics you feel like.
\item[(b)]  $R = k[w,x,y,z]/\langle xy, z^2 + wx^2, yz \rangle$ in characteristic $2$ and $3$.
\item[(c)]  $R = k[x,y,z]/\langle x^3 + y^3 + z^3 \rangle$ in characteristics $7, 11$ and $13$.
\item[(d)]  $R = k[x,y,z]/\langle x^2 + y^3 + z^5 \rangle$ in characteristics $2,3,5,7$ and $11$.
\end{itemize}
\end{exercise}

\begin{starexercise}
Use Fedder's criterion to determine for which $p > 0$, the ring $k[x,y,z]/\langle x^3 + y^3 + z^3 \rangle$ is $F$-pure near the origin.  For some related computations, see \cite[Chapter V, Section 4]{SilvermanArithmeticOfEllipticCurves}.
\end{starexercise}

\begin{starexercise}  If $(R, \bm)$ is a regular local ring and $0 \neq f \in \bm$, then
the \emph{$F$-pure threshold $\fpt_{\bm}(f)$} of $f \in k[x_1, \dots, x_n]$, at the origin $\bm = \langle x_1, \dots, x_n \rangle$, is defined as follows:
\[
\lim_{e \to \infty} { \max \{ l \,|\, f^l \notin \bm^{[p^e]} \} \over p^e}.
\]
Prove that this limit exists in general and then show that $\fpt_{\bm}(x^3 - y^2) = {5 \over 6}$ if $p = 7$.  See \cite{MustataTakagiWatanabeFThresholdsAndBernsteinSato} for solutions, \cf \cite{TakagiWatanabeFPureThresh} .
\end{starexercise}

\section{Change of variety}
\label{sec.ChangeOfVariety}

In this section, we describe how $p^{-e}$-linear maps change under common change of variety operations.

\subsection{Closed subschemes}

We have already studied the behavior of $p^{-e}$-linear maps for subschemes extensively.  Indeed, suppose that $\phi : F^e_* R \to R$ is an $R$-linear map which is compatible with an ideal $I \subseteq R$.  Then we have an induced map $\phi_{R/I} : F^e_* (R/I) \to (R/I)$.  It is natural to ask what the divisor associated to $R/I$ is.

\begin{lemma}
\label{lem.InversionOfFadjunction}
Suppose that $R$ is a normal Gorenstein local ring, and that $D = V(f)$ is a normal Cartier divisor on $X = \Spec R$.  Fix $\Phi : F^e_* R \to R$ to be map generating $\Hom_R(F^e_* R, R)$ as an $F^e_* R$-module as in \autoref{ex.GeneratingMapForGorensteinRings}.  Set $\phi(F^e_* \blank) = \Phi(F^e_* (f^{p^e - 1} \cdot \blank))$.  Then $\phi$ is compatible with $D$ and furthermore, $\phi_{D}$ generates $\Hom_{R/\langle f \rangle}(F^e_* (R/\langle f \rangle), R/\langle f \rangle)$ as an $F^e_* R/\langle f \rangle$-module.
It follows that the $\bQ$-divisor $\Delta$ on $D$ associated to $\phi_{D}$, as in \autoref{eq.MapBijectionWithQDivisors1} is the zero divisor.
\end{lemma}
\begin{proof}
See \autoref{ex.DivisorAssociatedToCartierIsZero}.
\end{proof}

However, things are not always nearly so nice.  In particular the divisor associated to $\phi_D$ need not always be zero.

\begin{example}
\label{ex.ConeFDifferent}
Consider $S = k[x,y,z]$ with $p = \text{char} k \neq 2$, set $R := k[x,y,z]/\langle xy - z^2 \rangle$ and fix $D = V(\langle x, z\rangle)$.  Set $\Phi_S \in \Hom_S(F^e_* S, S)$ to be the $F^e_*S$-module generator as in \autoref{ex.ZeroDivisorOnAn}.  We notice that by Fedder's Lemma, \autoref{lem.FeddersLemma}, that $\Psi(F^e_* \blank) = \Phi(F^e_* ((xy-z^2)^{p^e - 1} \cdot \blank))$ induces the generator of $\Hom_R(F^e_* R, R)$ by restriction.  Notice that $\O_X(-2nD) = \langle x^n \rangle$ and consider the map
\[
\phi(F^e_* \blank) = \Psi(F^e_* (x^{p^e - 1 \over 2} \cdot \blank)) = \Phi(F^e_* (x^{p^e - 1 \over 2}(xy-z^2)^{p^e - 1} \cdot \blank)).
\]
If we set $X = \Spec R$, then the induced map $\phi_X \in \Hom_R(F^e_* R, R)$ corresponds to the divisor $(p^e - 1)D$.

However, it is easy to see that $\phi_X$ also is compatible with $D$.  Thus we obtain $\phi_D$.  To compute the divisor associated to $D$, we need only read off the term containing $x^{p^e - 1}z^{p^e - 1}$ in
\[
(x^{p^e - 1 \over 2})(xy-z^2)^{p^e - 1} = x^{3(p^e - 1) \over 2} y^{p^e -1} + \dots + {p^e - 1 \choose {p^e - 1 \over 2}} x^{p^e - 1} y^{p^e -1 \over 2} z^{p^e - 1}+ \dots + z^{2(p^e - 1)}
\]
Again, the reason this works is because the map $\Phi_S(F^e_* (x^{p^e - 1}z^{p^e-1} \cdot \blank))$ induces the generator on $\Hom_{\O_D}(F^e_* \O_D, \O_D)$.  But ${p^e - 1 \choose {p^e - 1 \over 2}} \neq 0 \text{ mod } p$ and so if $\Phi_D : F^e_* k[y] \to k[y]$ is the map generating $\Hom_{\O_D}(F^e_* \O_D, \O_D)$, then $\phi_D$ (which is just $\phi_X$ restricted to $D$) is defined by the rule
\[
\phi_{D}(F^e_* \blank) = \Phi_D(F^e_* y^{p^e -1 \over 2} \cdot \blank)
\]
at least up to multiplication by an element of $k$.  Thus, in the terminology of \autoref{eq.MapBijectionWithQDivisors1},
\[
\Delta_{\phi_D} = {1 \over p^e -1} \Div(y^{p^e - 1 \over 2}) = {1 \over 2} \Div(y).
\]
In particular, in contrast with \autoref{lem.InversionOfFadjunction}, $\Delta_{\phi_D} \neq 0$.
\end{example}

\begin{theorem}[$F$-adjunction]
If $X$ is a normal variety, $\Delta \geq 0$ is a $\bQ$-divisor on $X$ such that $K_X + \Delta$ is $\bQ$-Cartier with index not divisible by $p$.  Suppose that $Y$ is an $F$-pure center (see \autoref{def.FPureCenter}) of $(X, \Delta)$ and that $\phi$ corresponds to $\Delta$ as in \autoref{eq.MapBijectionWithQDivisors2}. Then there exists a canonically determined $\bQ$-divisor $\Delta_Y \geq 0$ such that:
\begin{itemize}
\item[(a)]  $(K_Y + \Delta)|_Y \sim_{\bQ} K_Y + \Delta_Y$
\item[(b)]  $(X, \Delta)$ is sharply $F$-pure near $Y$ if and only if $(Y, \Delta_Y)$ is sharply $F$-pure.
\end{itemize}
\end{theorem}
\begin{proof}
Set $\phi_Y$ to be the restriction of $\phi$ to $Y$ as in \autoref{lem.RestrictNearSplitting}.  Set $\Delta_Y$ to be the $\bQ$-divisor associated to $\phi_Y$ as in \autoref{eq.MapBijectionWithQDivisors2}.  The first result then follows easily.  The second follows since $\phi$ is surjective near $Y$ if and only if $\phi_Y$ is surjective.
\end{proof}

\begin{remark}
The previous result should be compared with subadjunction and inversion of adjunction in birational geometry.  See for example \cite{KawamataSubadjunction2,KawakitaInversion,HaconLogCanonicalInversionOfAdj} and \cite[Chapter 5, Section 4]{KollarMori}.
\end{remark}

\subsection{Birational maps}

Suppose that $X$ is a normal variety, $\sL$ is a line bundle on $X$ and $\phi : F^e_* \sL \to \O_X$ is an $\O_X$-linear map corresponding to the $\bQ$-divisor $\Delta$ as in \autoref{eq.MapBijectionWithQDivisors2}.  Suppose $\pi : \tld X \to X$ is a birational map with $\tld X$ normal.  Fix $K_{\tld X}$ and $K_X$ which agree wherever $\pi$ is an isomorphism.  We can write
\[
K_{\tld X} + \Delta_{\tld X}  = \pi^*(K_X + \Delta)
\]
where now $\Delta_{\tld X}$ is uniquely determined.  Notice that $\Delta_{\tld X}$ need not be effective.  The main result of this section is the following:

\begin{lemma}
\label{lem.PhiOnBirational}
The map $\phi : F^e_* \sL \to \O_X$ induces a map $\tld \phi: F^e_* (\pi^* \sL) \to \sK({\tld X})$ where $\sK(\tld X)$ is the fraction field sheaf of $\tld X$ (which we can also identify with the fraction field on $X$ since $\pi$ is birational).  Furthermore, $\tld \phi$ agrees with $\phi$ wherever $\pi$ is an isomorphism.

Even more, using the fact that maps to the fraction field correspond to possibly non-effective divisors via \autoref{ex.NonEffectiveDivisors}, we have that $\Delta_{\tld \phi} = \Delta_{\tld X}$.
\end{lemma}
\begin{proof}
We construct $\tld \phi$ as follows.  We note that $\sL = \O_X( (1 - p^e)(K_X + \Delta))$ by \autoref{eq.MapBijectionWithQDivisors2}, and so after fixing $K_X$, we obtain an embedding of $\sL \subseteq \sK(X)$.  In particular, for each affine open set $U$, we have an embedding $\Gamma(U, \sL) \subseteq K(X)$.  But then we also obtain for each affine open set $V \subseteq \tld X$, an embedding $\Gamma(V, \pi^* \sL) \subseteq K(\tld X) = K(X)$.

Now, by taking the map $F^e_* \sL \to \O_X$ at the generic point $\eta$ of $X$, we obtain $\phi_{\eta} : F^e_* K(X) \to K(X)$ (note that our embedding of $\sL \subseteq \sK(X)$ fixes the isomorphism $\sL_{\eta} \cong K(X)$).  But we identify $\eta$ with the generic point $\tld \eta$ of $\tld X$ (since they have isomorphic neighborhoods) and so we have a map $\phi_{\tld \eta} : F^e_* K(\tld X) \to K(\tld X)$.  By restricting $\phi_{\tld \eta}$ to $\Gamma(V, \pi^* \sL)$ for each open set $V$, we obtain a map $\tld \phi : F^e_* \pi^* \sL \to \sK(\tld X)$.

By construction, $\tld \phi$ agrees with $\phi$ wherever $\pi$ is an isomorphism.  For the statement $\Delta_{\tld \phi} = \Delta_{\tld X}$ we proceed as follows.  We notice that $\Delta_{\tld \phi}$ and $\Delta_{\tld X}$ already agree wherever $\pi$ is an isomorphism so that $\Delta_{\tld \phi} - \Delta_{\tld X}$ is $\pi$-exceptional.  Furthermore, by the construction done in \autoref{ex.NonEffectiveDivisors} $\O_{\tld X}( (1-p^e)(K_{\tld X} + \Delta_{\tld \phi})) \cong \pi^* \sL \cong \pi^* \O_X((1-p^e)(K_X + \Delta))$.  Thus $\Delta_{\tld \phi} \sim_{\bQ} \Delta_{\tld X}$ and so $$\Delta_{\tld \phi} - \Delta_{\tld X} \sim_{\bQ} 0$$ is $\pi$-exceptional.  Therefore $\Delta_{\tld \phi} = \Delta_{\tld X}$ as desired, \cf \cite{KollarMori}.
\end{proof}

We now come to the definition of log canonical singularities (in arbitrary characteristic).

\begin{definition}
Suppose that $X$ is a normal variety and that $\Delta$ is a $\bQ$-divisor such that $K_X + \Delta$ is $\bQ$-Cartier. Then we say that $(X, \Delta)$ is \emph{log canonical} if the following condition holds.  For every proper birational map $\pi : \tld X \to X$ with $\tld X$ normal, when we write
\[
\sum a_i E_i = K_{\tld X} - \pi^*(K_X + \Delta)
\]
each $a_i$ is $\geq -1$.
\end{definition}

\begin{theorem}\cite[Main Theorem]{HaraWatanabeFRegFPure}
\label{thm.FPureImpliesLC}
If $(X, \Delta)$ is sharply $F$-pure, then $(X, \Delta)$ is log canonical.
\end{theorem}
\begin{proof}
The statement is local on $X$ and so we may assume that $\sL = \O_X$ and that $X = \Spec R$ is affine.
We only prove the case where the index of $K_X + \Delta$ is not divisible by $p$.  To reduce to this case, use \autoref{ex.FpureImpliesLogCanonicalGeneral} below.  Set $\phi : F^e_* R \to R$ to be a map corresponding to $\Delta$.  Thus there exists an element $c \in R = \Gamma(X, \sL)$ such that $\phi(F^e_* c) = 1$ since $(X, \Delta)$ is sharply $F$-pure.

Set $\pi : \tld X \to X$ a proper birational map with $\tld X$ normal and write $\sum a_i E_i = K_{\tld X} - \pi^*(K_X + \Delta)$.  Suppose that some $a_i < -1$ (with corresponding fixed $E_i$).  Then in particular $a_i \leq 0$.  Set $\eta_i$ to be the generic point of $E_i$.  It follows that $-a_i$, the $E_i$-coefficient of $\Delta_{\tld \phi}$, is positive and so we have a factorization:
\[
F^e_* \O_{\tld X, \eta_i} \subseteq F^e_* \O_{\tld X, \eta_i}( (1-p^e)a_i E_i) \xrightarrow{\tld\phi} \O_{\tld X, \eta_i},
\]
where $\tld \phi$ is as in \autoref{lem.PhiOnBirational}.
But now it is easy to see that if $a_i < -1$, then $(1-p^e)a_i \geq p^e$ so that we have the factorization
\[
F^e_* \O_{\tld X, \eta_i} \subseteq F^e_* \O_{\tld X, \eta_i}( p^e E_i) \xrightarrow{\tld \phi|_{F^e_* \O_{\tld X, \eta_i}( p^e E_i)}} \O_{\tld X, \eta_i}.
\]
which sends $F^e_* c \in F^e_* R \subseteq  F^e_* \O_{\tld X, \eta_i}$ to $1$.  But that is impossible since if $d \in \O_{\tld X, \eta_i}$ is the local parameter for $E_i$, then $\tld \phi$ sends $F^e_* (c/d^{p^e}) \in F^e_* \O_{\tld X, \eta_i}( p^e E_i)$ to $1/d \notin \O_{\tld X, \eta_i}$.
\end{proof}

\subsection{Finite maps}
\label{subsec.FiniteMaps}

Finally, suppose that $\pi : Y \to X$ is a finite surjective map of normal varieties.  Then there is an inclusion $\O_X \subseteq \pi_* \O_Y$.  Given a line bundle $\sL$ on $X$ and a map \mbox{$\phi : F^e_* \sL \to \O_X$,} it is natural to ask when $\phi$ can be extended to a map $F^e_* \pi_* (\pi^* \sL) \to \pi_* \O_Y$.  Since $\pi$ is finite, the $\pi_*$ is harmless and so we can ask when $\phi$ can be extended to a map $\phi_Y : F^e_* \pi^* \sL \to \O_Y$.

The local version of this statement is as follows.  Suppose that $R \subseteq S$ is a finite extension of semi-local normal rings and suppose that $\phi : F^e_* R \to R$ is a finite map.  Then when does there exist a commutative diagram as follows?
\[
\xymatrix{
F^e_* S \ar[r]^{\phi_S} & S\\
F^e_* R \ar@{^{(}->}[u] \ar[r]_{\phi} & R \ar@{^{(}->}[u]
}
\]
It is easy to see that the answer is not always.

\begin{example}
Consider $k[x^2] \subseteq k[x]$ with $p = \text{char} k \neq 2$.  Consider the map $\phi : F_* k[x^2] \to k[x^2]$ which sends $F_* x^{2(p-1)}$ to $1$ and other monomials $F_* x^{2i}$, for $0 \leq i < p-1$  to zero.  Note $\Delta_{\phi} = 0$.

Suppose this map extended to a map $\psi : F_* k[x] \to k[x]$.  Then $\phi(F_* x^{2(p-1)}) = 1$ and so since $\phi$ and $\psi$ are the same on $k[x^2]$, we have
\[
1 = \psi(F_* x^{2(p-1)}) = \psi(F_* x^p x^{p-2}) = x \psi(F_* x^{p-2})
\]
which implies that $x$ is a unit.  But that is a contradiction.

On the other hand, consider the map $\alpha : F_* k[x^2] \to k[x^2]$ which sends $F_* x^{2(p-1)/2} = F_* x^{p-1}$ to $1$ and the other monomials $F_* x^{2i}$ to $0$ for $0 \leq i \leq p-1$ to zero.  Note $\Delta_{\alpha} = {1 \over 2} \Div(x^2)$.

We will show that $\alpha$ extends to a map $\beta : F_* k[x] \to k[x]$.  It is in fact easy to show that $\alpha$ extends to a map on the fraction field $\beta : F_* k(x) \to k(x)$, see \autoref{ex.SeparableMapsExtendToFractionFields}.  Therefore, it is enough to show that $\beta(F_* x^j) \in k[x]$ for each $0 \leq j \leq p-1$.  Fix such a $j$.  If $j$ is even, then there is nothing to do since $\beta(F_* x^j) = \alpha(F_* x^j) \in k[x^2] \subseteq k[x]$.   Therefore, we may suppose that $j$ is odd. But then $j + p$ is even and $p \leq j+ p \leq 2(p-1)$.  Thus
\[
\beta(F_* x^j) = {1 \over x} \beta(F_* x^{j + p}) = {1 \over x} \alpha(F_* x^{j+p}) = {1 \over x} \cdot 0 = 0 \in k[x].
\]
This proves that $\beta$ exists and is well defined.
\end{example}

\begin{theorem} \cite{SchwedeTuckerTestIdealFiniteMaps}
\label{thm.ExtendingOverFiniteMaps}
Fix $\pi : Y \to X$ as above.  Fix a nonzero map $\phi : F^e_* \sL \to \O_X$ as above.  If $\pi$ is inseparable then $\phi$ never extends to $\phi_Y$.  If $\pi$ is separable, then there exists a map $\phi_Y : F^e_* \pi^* \sL \to \O_Y$ extending $\phi$ if and only if $\Delta_{\phi}$ is bigger than or equal to the ramification divisor of $\pi : Y \to X$.
\end{theorem}
\begin{proof}
We won't prove this but we will sketch the main steps and leave the details as an exercise.  We first work in the separable case.

\begin{description}
\item[Step 1]  The statement is local on $X$ and so we may suppose that $X = \Spec R$, $Y = \Spec S$ and $\sL = \O_X$.  In fact, we may even assume that $R$ is a DVR and that $S$ is a Dedekind domain.
\item[Step 2]  There is a map $\phi_S$ and a commutative diagram:
\[
\xymatrix{
F^e_* S \ar[r]^{\phi_S} & S\\
F^e_* R \ar@{^{(}->}[u] \ar[r]_{\phi} & R \ar@{^{(}->}[u]
}
\]
if and only if there exists a map $\phi_S$ and a commutative diagram.
\[
\xymatrix{
F^e_* S \ar[d]_{F^e_* \Tr} \ar[r]^{\phi_S} & S \ar[d]^{\Tr} \\
F^e_* R \ar[r]_{\phi} & R
}
\]
where $\Tr : S \to R$ is simply the restriction of the field trace $\Tr : K(S) \to K(R)$ to $S$.
\item[Step 3]  $\Hom_R(S, R)$ is isomorphic to $S$ as an $S$-module.  The map $\Tr : S \to R$ is a section of this and so corresponds to a divisor $D$ on $\Spec R$.  This divisor is the ramification divisor $\text{Ram}_{\pi}$ of $\pi : \Spec S \to \Spec R$.
\item[Step 4]  Supposing $\phi_S$ exists, compute the divisor corresponding to $\Tr \circ \phi_S = \phi \circ (F^e_* \Tr)$.  This gives one direction of the if and only if. Working with the fraction fields, as in \autoref{ex.NonEffectiveDivisors}, yields the other direction.
\end{description}

For the inseparable case, it turns out that the only map that can extend is the zero map, see \autoref{ex.InseparableMapExtensions}.
\end{proof}

\subsection{Exercises}

\begin{exercise}
In the setting of \autoref{lem.InversionOfFadjunction}, prove that the divisor $D$ is $F$-pure (as a variety) if and only if $\phi$ is surjective.
\end{exercise}

\begin{starexercise}
\label{ex.DivisorAssociatedToCartierIsZero}  Prove \autoref{lem.InversionOfFadjunction}.\vskip 3pt
\emph{Hint: } Consider the map $\langle \phi \rangle_{F^e_* R} \to \Hom_{R/\langle f \rangle}(F^e_* {R/\langle f \rangle}, {R/\langle f \rangle})$ and prove it is surjective at the codimension $1$ points of ${R/\langle f \rangle}$.  For a solution, see \cite[Proposition 7.2]{SchwedeFAdjunction}.
\end{starexercise}

\begin{starstarexercise}[The $F$-different]
Suppose that $X$ is a normal variety and $D$ is an effective normal Weil divisor such that $K_X + D$ is $\bQ$-Cartier with index not divisible by $p$.  Thus there exists a map $\phi_D : F^e_* \sL \to \O_X$ as in \autoref{eq.MapBijectionWithQDivisors2} corresponding to $D$ for any $e$ such that $(p^e - 1)(K_X + D)$ is Cartier.  It is easy to see that this map is compatible with $D$ and so it induces a map:
\[
\phi_D : F^e_* \sL|_D \to \O_D.
\]
This map corresponds to a $\bQ$-divisor $\Delta_D$ on $D$, again by \autoref{eq.MapBijectionWithQDivisors2}, which is called the \emph{$F$-different}.   Verify all the statements made above.

It is an open question whether or not the $F$-different always coincides with the different, as described in \cite[Chapter 17]{KollarFlipsAndAbundance} or \cite[10.6]{ShokurovThreeDimensionalLogFlips}.  Prove that it either does or does not and write a paper about it, and then tell the authors of this survey paper what you found (this is why the problem gets $**$).  For more discussion see \cite[Remark 7.6]{SchwedeFAdjunction}. \end{starstarexercise}

\begin{starexercise}
Consider the family of cones over elliptic curves:
\[
X = \Spec k[x,y,z,t]/\langle y^2 - x(x-1)(x-t)\rangle \to \bA^1 = \Spec k[t]
\]
Set $\Phi \in \Hom_X(F_* \O_X, \O_X)$ to be the map generating $\Hom_X(F_* \O_X, \O_X)$ as an $F_* \O_X$-module.  Show that $\Phi$ is compatible with the ideal $J = \langle x,y, z\rangle$.  Consider $\Phi_J = \Phi/J$, the map obtained by restricting $\Phi$ to $V(J) \cong \bA^1$.  Show that $\Delta_{\Phi_J}$ is supported exactly at those points whose fibers correspond to supersingular elliptic curves.
\end{starexercise}

\begin{exercise}
Using the notation of \autoref{lem.PhiOnBirational}, suppose that $\Delta_{\phi}$ is the effective divisor associated to $\phi$.  Show that there is a map $$\phi' : F^e_* \big((\pi^* \sL)(\lceil K_{\tld X} - \pi^*(K_X + \Delta_{\phi})\rceil)\big) \to \O_{\tld X}(\lceil K_{\tld X} - \pi^*(K_X + \Delta_{\phi})\rceil )$$ that agrees with $\phi$ wherever $\pi$ is an isomorphism.
\vskip 3pt
{\emph Hint: } It is sufficient to show that there is a map $\phi'' : F^e_* \big((\pi^* \sL)(\lceil K_{\tld X} - \pi^*(K_X + \Delta_{\phi})\rceil - p^e \lceil K_{\tld X} - \pi^*(K_X + \Delta_{\phi})\rceil )\big) \to \O_{\tld X}$.  Now, use the roundings to your advantage and the fact that $\pi^* \sL = \O_{\tld X}( \pi^* (1 - p^e)(K_X + \Delta_{\phi}))$.
\end{exercise}

\begin{exercise}
\label{ex.FpureImpliesLogCanonicalGeneral}
Suppose that $(X, \Delta)$ is sharply $F$-pure.  Prove that for every point $x \in X$ there exists a divisor $\Delta_U$ on a neighborhood $U$ of $x$ such that $\Delta_U \geq \Delta|_U$, such that $(U, \Delta_U)$ is sharply $F$-pure \emph{and} such that $K_U + \Delta_U$ has index not divisible by $p$.  Conclude that \autoref{thm.FPureImpliesLC} holds in full generality.  For a solution, see \cite[Theorem 4.3(ii)]{SchwedeSmithLogFanoVsGloballyFRegular}.
\end{exercise}

\begin{exercise}
\label{ex.SeparableMapsExtendToFractionFields}  Suppose that $R \subseteq S$ is an extension of integral domains with induced \emph{separable} extension of fraction fields $K(R) \subseteq K(S)$.  Fix $\phi : F^e_* R \to R$ to be an $R$-linear map.  Prove that there is always a map $\psi : F^e_* K(S) \to K(S)$ such that $\psi|_R = \phi$.
\vskip 3pt
\emph{Hint: } First form $\phi_{\eta} : F^e_* K(R) \to K(R)$ by localization.  Then tensor this map with $K(S)$ and use the fact that $K(R) \subseteq K(S)$ is separable (unlike $K(R) \subseteq F^e_* K(R) \cong (K(R))^{1/p^e}$).
\end{exercise}

\begin{starexercise}
Prove the separable case of \autoref{thm.ExtendingOverFiniteMaps} by filling in the details of the Steps 1 through 4.  Step 3 is somewhat involved, see for example \cite{MoriyaTheorieDerDerivationen,SchejaStorchUberSpurfunktionen,deSmitTheDifferentAndDifferentials}. On the other hand, see \cite{SchwedeTuckerTestIdealFiniteMaps} for a complete proof.
\end{starexercise}

\begin{exercise}
\label{ex.InseparableMapExtensions}
Prove the inseparable case of \autoref{thm.ExtendingOverFiniteMaps} as follows.  First suppose that $K \subseteq L$ is a purely inseparable extension of fields.  Suppose that
$\phi : F_* K \to K$ is a $K$-linear map that extends to an $L$-linear map $\phi_L : F_* L \to L$.  Prove that $\phi = 0$.

Use the above to prove that now if $L \supseteq K$ is any inseparable map, the only map $\phi : F_* K \to K$ that extends to $\phi_L : F_* L \to L$ is the zero map.  For a complete solution, see \cite[Proposition 5.2]{SchwedeTuckerTestIdealFiniteMaps}.
\end{exercise}

\section{Cartier modules}
\label{sec.CartierModules}
Perhaps the most natural example of a $p^{-e}$-linear map is the trace of the Frobenius $F_*\omega_X \to \omega_X$ on the canonical sheaf of a normal variety as discussed in detail in \autoref{subsec.TraceOfFrob}. In generalizing one is lead to consider the category consisting of (coherent) $\O_X$-modules $\sF$ equipped with a $p^{-e}$-linear map $\kappa \colon F_*^e \sF \to \sF$. We will outline here the resulting theory in a slightly more general setting than considered in \cite{BlickleBoeckleCartierModulesFiniteness}.
\begin{definition}
If $\sL$ is a line bundle on $X$, then a \emph{$(\sL,p^e)$--Cartier module} is a coherent $\O_X$-module $\sF$ equipped with an $\O_X$-linear map
\[
\kappa \colon F^e_*(\sF \tensor \sL) \to \sF\, .
\]
(or equivalently, equipped with a $p^{-e}$ linear map $\sF \tensor \sL \to \sF$). If $\sL \cong \O_X$, we call these objects mostly just Cartier modules.
\end{definition}

\begin{remark}
Cartier modules as originally defined in the work of \cite{BlickleBoeckleCartierModulesFiniteness} were always defined with $\sL \cong \O_X$.  The addition of the $\sL$ adds little to the complication of the basic theory (which generally reduces to the local case where $\sL$ is trivialized).  Although admittedly, it does add some notational complications.  However, this generalization does show up naturally.  Regardless, little will be lost if the reader always assumes that $\sL =\O_X$.
\end{remark}

A morphism of $(\sL,p^e)$--Cartier modules $(\sF,\kappa_{\sF})$ and $(\sG,\kappa_{\sG})$ is an $\O_X$-linear map $\phi \colon \sF \to \sG$ such that the diagram
\[
\xymatrix{
F^e_*(\sF \tensor \sL) \ar[r]^-{\kappa_{\sF}} \ar[d]_{F^e_*(\phi \tensor \id)}& \sF\ \ar[d]^\phi \\
F^e_*(\sG \tensor \sL) \ar[r]_-{\kappa_{\sG}} & \sG
}
\]
commutes. If $(\sF,\kappa)$ is a $(\sL,p^e)$--Cartier module, then we can apply $F^e_*$ to $\kappa \tensor \sL$ to obtain -- using the projection formula -- a map
\[
   \kappa^2 \colon F^{2e}_*(\sF \tensor \sL \tensor \sL^{p^e}) \cong F^e_*(F^e_*(\sF \tensor \sL) \tensor \sL) \to[F^e_*(\kappa \tensor \sL)] F^e_*( \sF \tensor \sL) \to[\kappa] \sF
\]
which equips $\sF$ with the structure of a $(\sL^{1+p^e},p^{2e})$--Cartier module. Iterating this construction in the obvious way (similar to \autoref{eq.ComposeMaps2}) we obtain morphisms
\[
    \kappa^e \colon F^{ne}_* \left(\sF \tensor \sL^{1+p^e+p^{2e}+\cdots+p^{(n-1)e}}\right) \to \sF
\]
for all $n \geq 1$, making $\sF$ into a $(\sL^{\frac{p^{ne}-1}{p^e-1}},p^{ne})$-Cartier module.

\begin{proposition}
The category of (coherent) $(\sL,p^e)$-Cartier modules is an Abelian category. The kernel and cokernel of the underlying quasi-coherent sheaves carry an obvious Cartier module structure and are the kernel and cokernel in the category of Cartier modules.
\end{proposition}
\begin{proof}
This is easy to verify since $\usc \tensor \sL$ as well as $F^e_*\usc$ are exact functors. Alternatively, we may view $(\sL,p^e)$-Cartier modules as the right module category over a certain (non-commutative) sheaf of rings, see \autoref{ex.RFdef} below, which immediately implies that the category is Abelian.
\end{proof}

Compared to a Frobenius splitting, which is nothing but a Cartier module structure on the coherent sheaf $\O_X$, the advantages of working in this larger category of Cartier modules are manifold. For one, there are a number of natural examples of Cartier modules, most prominently the canonical sheaf $\omega_X$ together with the trace of Frobenius as Cartier module structure. Furthermore one has in this category methods to construct new Cartier modules by functorial operations. Most notably there is the notion of a push-forward for proper maps (in the case that $\sL \cong \O_X$), localization and \'etale pullback, and even an extraordinary pullback $f^!$ can be defined \cite{BliBoe.CartierCrys,BlickleBoeckleCartierModulesFiniteness}. We conclude this subsection by illustrating some of these concepts in special cases.  First however, we state some examples.

\begin{examples}[Examples of Cartier modules]
$\text{ }$
\begin{enumerate}
\item The canonical sheaf $\omega_X$ is a Cartier module with structural map $\kappa \colon F_* \omega_X \to \omega_X$ given by the trace map.  More generally, if $\omega_X^{\mydot}$ is the dualizing complex of $X$, then the trace of Frobenius is a map (in the derived category) $F_*\omega_X^{\mydot} \to \omega_X^{\mydot}$. This induces for each $i$ the structure of a Cartier module on the cohomology $\myH^i \omega_X^{\mydot}$.
\item Suppose that $D$ is a Cartier divisor on $X$, then the map
\[
F^e_* (\omega_X(p^e D)) \xrightarrow{\Tr} \omega_X(D)
\]
equips $\omega_X(D)$ with the structure of an $\O_X((p^e - 1)D)$-Cartier module.
\item Suppose that $D$ is an effective integral divisor on $X$, then the composition
\[
F_* \omega_X(D) \hookrightarrow F_* \omega_X(p D) \xrightarrow{\Tr} \omega_X(D)
\]
equips $\omega_X(D)$ with the structure of a Cartier module as well.
\item Suppose that $\pi : Y \to X$ is a proper map of varieties.  Then $R^i \pi_* \omega_Y$ is a Cartier module for any $i \geq 0$.  This is because $F_* R^i \pi_* \omega_Y = R^i \pi_* F_* \omega_Y$.
\item  Set $X = \mathbb{A}^2$ and let $\pi : Y \to X$ be the blowup at the origin with exceptional divisor $E$.  Thus we have $\Tr_Y : F_* \omega_Y \to \omega_Y$ as the trace on $Y$.  Now, $\omega_Y \cong \O_Y(E)$.  Thus by twisting by $-E$ ,we have a $\O_Y( (1-p)E)$-Cartier module structure on $\O_Y$.  Namely, a map $\Tr : F_* (\O_Y( (1-p)E)) \to \O_Y$.
\end{enumerate}
\end{examples}
Since localization at any multiplicative set commutes with pushforward along the Frobenius (see \autoref{ex.LocalizationCompletionAndFeLowerStar} and \autoref{ex.localizationFrob}) we observe that localization preserves the Cartier module structure.
\begin{lemma}
Let $S \subseteq R$ be a multiplicative system and $\sF$ a $(\sL,p^e)$--Cartier module on $X=\Spec R$. Then the map
\[
    F^e_{S^{-1}R *}(S^{-1}\sF \tensor_{S^{-1}R} S^{-1}\sL) \cong S^{-1}F^e_*(\sF \tensor_R \sL) \to[S^{-1}\kappa_{\sF}] S^{-1}\sF
\]
is a $(S^{-1}\sL,p^e)$--Cartier module structure on $S^{-1}\sF$.
\end{lemma}
In particular, if $j \colon U \subseteq X = \Spec R$ is the inclusion of a basic open subset $U = \Spec R_f$ for some $f \in R$, then the pullback $j^*$ induces a functor from $\sL$--Cartier modules on $X$ to $j^*\sL$--Cartier modules on $U$. Using a \Cech-complex construction, this globalizes to an arbitrary open immersion $U \subseteq X$. Even more generally this holds for any essentially \'etale\footnote{essentially \'etale means essentially of finite type and formally \'etale, \ie~a morphism that can be factored as a localization followed by a finite type \'etale morphism} morphism $j \colon U \to X$, see \cite{BliBoe.CartierCrys}.

\begin{proposition}
Let $j \colon U \to X$ be essentially \'etale and let $\sF$ be a $\sL$-Cartier module on $X$. Then the pullback $j^*\sF$ carries a natural functorial structure of a $j^*\sL$-Cartier module on $U$. The structural map is given by
\[
    F_{U*}(j^*\sF \tensor j^*\sL) \cong F_{U*}j^*(\sF \tensor \sL) \cong j^*F_{X*}(\sF \tensor \sL) \to[j^*\kappa] j^*\sF\, .
\]
\end{proposition}
\begin{proof}
The key point is the fact that for an essentially \'etale morphism $j \colon U \to X$ the diagram
\[
\xymatrix{
    U \ar[r]^j \ar[d]_{F_Y} & X \ar[d]^{F_X} \\
    U \ar[r]^j & X
}
\]
is Cartesian and that the base change morphism $j^* F_{X*} \cong F_{U_*}j^*$ is an isomorphism since $j$ is flat, see \cite{HochsterHunekeTC1}. This justifies the definition of the Cartier structure on $j^*\sF$.
\end{proof}

For a closed immersion $i \colon Y \to X$, the pullback $i^*$ does not give a functor on Cartier modules. The reason is precisely that the above diagram is not Cartesian in this case. However there is an exotic restriction functor one can define. For concreteness, let $X=\Spec R$ be affine and let $Y = \Spec R/I$ for some ideal $I \subseteq R$. Then, for an $R$-module $M$, the $R/I$ submodule $i^\flat(M) \colonequals \Hom_R(R/I,M)= \{ m \in M | Im=0 \}$ is just the $I$-torsion submodule $M[I] \subseteq M$. Note that $F_*(M[I]) \subseteq F_*(M[I^{[p]}]) = (F_*M)[I]$ which shows that $F_*(M[I])$, is contained in the $I$-torsion submodule $(F_* M)[I]$ of $F_* M$. Hence, if $\kappa \colon F_*M \to M$ is a Cartier module structure on $M$, then we have that this restricts to a map
\[
    \kappa \colon F_*(M[I]) \to M[I]
\]
giving $M[I]$ a natural Cartier module structure. The same construction works globally and more generally for $(\sL,p^e)$-Cartier modules:
\begin{proposition}
Let $i \colon Y \into X$ be a closed immersion given by a sheaf of ideals $I$ of $\O_X$, and let $\sF$ be a $(\sL,p^e)$-Cartier module on $X$. Then the $\O_Y$-module (via action on the first argument) $i^\flat(\sF) = \Hom_{\O_X}(i_*\O_Y,\sF)=\sF[I]$ carries a natural functorial structure of a $(\sL|_Y,p^e)$-Cartier module on $Y$. The structural map is given by
\[
    F^e_*(\sF[I] \tensor_{\O_Y} \sL|_Y)
    \subseteq F^e_*((\sF \tensor_{\O_X} \sL)[I^{[p^e]}])
    = (F_*^e(\sF \tensor_{\O_X} \sL))[I] \to[\kappa_{\sF}] \sF[I] = i^\flat \sF\, .
\]
\end{proposition}

Finally, let us consider a proper morphism of varieties $\pi \colon Y \to X$. Since the Frobenius commutes with any morphism one has a natural isomorphism of functors $F_{X*}^e \circ \pi_* \cong \pi_* \circ F_{Y*}^e$ which implies that the pushforward induces a functor on Cartier modules as well.
\begin{proposition}
Let $\pi \colon Y \to X$ be a proper morphism and $\kappa \colon F^e_* \sF \to \sF$ a Cartier module on $Y$. Then the map
\[
    F^e_*(\pi_*(\sF)) \cong \pi_*(F^e_* \sF) \to[\pi_*(\kappa)] \pi_*\sF
\]
is a Cartier module structure on $\pi_*\sF$. The same construction also holds for the higher derived images $R^i\pi_*\sF$.
\end{proposition}
Note, however, that if $\sF$ is a $(\sL,p^e)$--Cartier module there is no obvious way to equip $\pi_*\sF$ with such a structure unless $\sL$ is of the form $\pi^*\sL'$ for some invertible sheaf on $X$. In this case, using the projection formula, one obtains
\[
    F^e_*(\pi_*\sF \tensor \sL') \cong F^e_*(\pi_*(\sF \tensor \pi^*\sL')  \cong \pi_*(F^e_*(\sF \tensor \sL) \to[\pi_*(\kappa)] \pi_*\sF
\]
as a Cartier structure on $\pi_*\sF$.

\begin{example}
Let $\kappa \colon F^e_*\sF  \to \sF$ be a Cartier module, then the pushforward along the Frobenius (which is an affine map) equips $F_*\sF$ with the Cartier module structure
\[
    F^e_*\kappa \colon F^e_* (F^e_*(\sF)) \to F^e_*\sF
\]
making $\kappa \colon F^e_*\sF \to \sF$ into a map of Cartier modules.
\end{example}

\subsection{Finiteness results for Cartier modules.}\label{sec.AlgebrasAndCrystals}
In this section we state, and outline the proofs of two key structural results which make the category of Cartier modules interesting. But first we introduce the basic concept of nilpotence of a Cartier module and recall some elementary constructions, starting with the following simple Lemma whose verification we leave to the reader in \autoref{ex.ProofOfDescendingChain}.
\begin{lemma}\label{lem.ImagesKappaAreCartier}
Let $\kappa \colon F^e_*(\sF \tensor \sL) \to \sF$ be a Cartier module. Then the images $\sF_n \colonequals \kappa^n(F^{ne}_*(\sF \tensor \sL^{1+p^e+\cdots+p^{(n-1)e}}) \subseteq \sF$ are Cartier submodules of $\sF$, and satisfy the properties:
\begin{enumerate}
\item $\sF_n \supseteq \sF_{n+1}$.
\item $\kappa (F^e_*(\sF_n \tensor \sL)) = \sF_{n+1}$.
\item If $S \subseteq \O_X$ is a multiplicative set, then $S^{-1}\sF_n = (S^{-1}\sF)_n$.
\item The sequence of closed subsets $Y_n \colonequals \Supp \sF_n/\sF_{n+1}$ is descending.
\end{enumerate}
\end{lemma}

An important notion in the theory of Cartier modules, and in particular, for its applications to finiteness results for local cohomology for local rings, is the notion of nilpotence.
\begin{definition}
Let $\sF$ be a coherent Cartier module on $X$. We say that $\sF$ is \emph{nilpotent} if for some $n \geq 0$ the $n$th power $\kappa^n$ of the structural map $\kappa$ is zero.
\end{definition}
Some basic properties of this notion are collected in the following Lemma:
\begin{lemma}
Let $\kappa \colon F^e_*(\sF \tensor \sL) \to \sF$ be a Cartier module. Denote by $\sF^n \subseteq \sF$ the Cartier submodule of $\sF$ consisting of all local sections $s$ such that $\kappa^n( F^{ne}_*(\O_C \cdot s \tensor \sL^{1+p^e+\cdots +p^{(n-1)e}}) = 0$. Then
\begin{enumerate}
\item $\sF^n \subseteq \sF^{n+1}$ for all $n \geq 0$.
\item $\kappa (F_*^e(\sF^{n+1} \tensor \sL)) \subseteq \sF^n$.
\item If $S \subseteq \O_X$ is a multiplicative set, then $S^{-1}\sF^n = (S^{-1}\sF)^n$.
\item If $\sF$ is coherent, then the ascending sequence stabilizes and the stable member $\sF_{\nil}=\bigcup_n \sF^n$ is the maximal nilpotent Cartier submodule of $\sF$.
\end{enumerate}
\end{lemma}

Nilpotent Cartier modules form a Serre subcategory of all coherent Cartier modules, \ie~they form an Abelian subcategory which is closed under extension. The only non-trivial part here is the non-closedness under extensions, see \autoref{ex.NilpotentExtensions}.

The first structural result for Cartier modules we will show is that the descending sequence of iterated images stabilizes.  This result was first proved in \cite[Lemma 13.1]{Gabber.tStruc}.  In fact, this result is essentially Matlis dual to a famous result of Hartshorne and Speiser \cite[Proposition 1.11]{HartshorneSpeiserLocalCohomologyInCharacteristicP} and generalized by G.~Lyubeznik \cite{LyubeznikFModulesApplicationsToLocalCohomology}, also \cf \cite{SharpTightClosureTestExponentsForCertain,SharpOnTheHSLThm} and \cite{Blickle.MinimalGamma}.

\begin{proposition}
\label{prop.CartierImagesStabilize}
Let $(\sF,\kappa)$ be a coherent $(\sL,p^e)$--Cartier module. Then the descending sequence of images \[
    \sF_n \colonequals \kappa^n(F^{ne}_*(\sF \tensor \sL^{1+p^e+\cdots + p^{(n-1)e}}) \subseteq \sF
\]
stabilizes. In particular, the stable image $\sigma(\sF) \subseteq \sF$ is the largest $(\sL,p^e)$-Cartier submodule with the property that the structural map $\kappa$ is surjective.
\end{proposition}
\begin{proof}
To show the stabilization of a sequence of subsheaves on a Noetherian scheme $X$ can be done on an affine open cover. Choosing the open sets of the cover sufficiently small we may assume that $\sL$ is trivial. Hence we may assume that $X = \Spec R$ and $M$ is a finitely generated $R$ module equipped with a $p^{-e}$-linear map $\kappa \colon M \to M$. And we have to show that the descending sequence of Cartier submodules of $M$
\[
    M \supseteq \kappa(M) \supseteq \kappa^2(M) \supseteq \cdots
\]
stabilizes. The sets
\[
    Y_n \colonequals \Supp (\kappa^n(M)/\kappa(\kappa^n(M)))
\]
form a descending sequence of closed subsets of $X$, by \autoref{lem.ImagesKappaAreCartier}. Since $X$ is Noetherian, the descending sequence must stabilize. After truncating we may assume that $Y=Y_n=Y_{n+1}$ for all $n$. We have to show that $Y$ is empty. Assuming otherwise, let $\bp$ be the generic point of a component of $Y$. Localizing at $\bp$ we may assume that $R$ is local with maximal ideal $\bp$ and that $Y = \{ \bp \} = \Supp (\kappa^n(M)/\kappa(\kappa^n(M)))$ for all $n$. In particular, for $e=0$ we obtain that there is an integer $k$ such that $\bp^k M \subseteq \kappa(M)$. Hence, for any $x \in \bp^k$
\[
    x^2M \subseteq x\bp^kM \subseteq x\kappa(M) = \kappa(x^{p^e}M) \subseteq \kappa(x^2M)
\]
and iterating we get $x^2M \subseteq \kappa^n(M)$ for all $n$. Hence $\bp^k(b-1) \subseteq \kappa^n(M)$ for all $e$ where $b$ is the number of generators of $\bp^k$. Hence the original chain stabilizes if and only if the chain $\kappa^n(M)/\bp^{k(b-1)}M$ does. But the latter is a chain in the finite length module $M/\bp^{k(b-1)}M$.
\end{proof}
A characterization of this stable image is as follows. $\sigma(\sF) \subseteq \sF$ is the smallest Cartier submodule of $\sF$ such that on the quotient $\sF/\sigma(\sF)$ some power of the structural map is zero. If this property is satisfied for some Cartier submodule $\sN \subseteq \sF$, then it is also satisfied for its image. The minimality now implies that for $\sigma(\sF)$ the structural map
\[
    F^e_*(\sigma(\sF) \tensor \sL) \to \sigma(\sF)
\]
is surjective. The Cartier modules with surjective structural map play an important role in the theory. For example one can see immediately (\autoref{ex.KappaSurjImpliesReduced}), that for such Cartier module $\kappa \colon F^e_*(\sF \tensor \sL) \onto \sF$ its annihilator $\Ann \sF$ is a sheaf of radical ideals, \ie~$\sF$ has reduced support. This may be viewed as a generalization of the reduced-ness of Frobenius split varieties alluded to earlier. It is also a key ingredient in the following Kashiwara-type equivalence which will be used repeatedly below (the easy but rewarding proof is left to the reader as \autoref{ex.BasicKashiwaraCartier}, see also \cite[Proposition 2.6 and Section 3.3]{BlickleBoeckleCartierModulesFiniteness}):

\begin{proposition}\label{prop.BasicKashiwaraCartier}
Let $\sF$ be a coherent Cartier module on $X$ with surjective structural map $\kappa_{\sF}$ (\ie~$\sigma(\sF) = \sF$). Then $I=\Ann_{\O_X} \sF$ is a sheaf of radical ideals and hence $\sF=\sF[I]=i^\flat(\sF)$ is a Cartier module on $Y=\Supp \sF$, the closed reduced subset of $X$ given by $I$.

More precisely, if $i \colon Y \to X$ denotes a closed immersion, then the functors $i^\flat$ and $i_*$ induce a (inclusion preserving) bijection between
\[
\left\{ \begin{array}{c}\text{coherent $\sL$--Cartier modules on $X$} \\ \text{with surjective structural map}  \\ \text{and $\Supp \sF \subseteq Y$}\end{array} \right\}  \longleftrightarrow \left\{ \begin{array}{c}\text{coherent $\sL|_Y$-Cartier modules on $Y$} \\ \text{with surjective structural map} \end{array} \right\}
\]
\end{proposition}

The most important structural result for Cartier modules is the following theorem which asserts that for a coherent Cartier module $\sF$, the lattice of Cartier submodules with surjective structural map satisfies the ascending and descending chain conditions.

\begin{theorem}\label{thm.CartierFiniteLength}
Let $X$ be a scheme and $\kappa \colon F^e_*(\sF \tensor \sL) \to \sF$ a coherent Cartier module. Then any chain of Cartier submodules
\[
    \cdots \sF_i \supseteq \sF_{i+1} \supseteq \sF_{i+2} \supseteq \cdots
\]
each of whose structural map $\kappa_{\sF_i}$ is surjective, is eventually constant (in both directions).
\end{theorem}

\begin{proof}
The ascending chain stabilizes simply because the underlying $\O_X$-module is coherent and our schemes are Noetherian. So it remains to show the descending chain condition. One way to proof this result is to show that there is a unique smallest Cartier submodule $\tau(\sF) \subseteq \sF$ which agrees with $\sigma(\sF)$ on each generic point of $X$, \ie~$\tau(\sF)_\eta = \sigma(\sF)_\eta$ for each $\eta$ the generic point of an irreducible component of $X$. This is a generalization of the notion of a \emph{test ideal} which will be discussed in some detail below \autoref{subsec.AlgebrasOfMaps}.

Assuming the existence of $\tau(\sF)$ for now, the proof can be outlined as follows: We show that the chain
\[
    \sF_0 \supseteq \sF_1 \supseteq \sF_2 \supseteq \cdots
\]
stabilizes by induction on $\dim X$, the case $\dim X = 0$ being clear. Since a chain stabilizes if it stabilizes after restriction to each of the finitely many irreducible components of $X$, we may assume that $X$ is irreducible. Since $X$ is Noetherian, the descending sequence of supports $\Supp \sF_i$ stabilizes. After truncating we may assume that $Y = \Supp \sF_i$ for all $i$. Since the structural map of each $\sF_i$ is surjective, we have by \autoref{prop.BasicKashiwaraCartier} that $\sF_i$ is annihilated by the ideal sheaf defining the reduced structure of $Y$. Hence we may view the $\sF_i$ as $(\sL|_Y,p^e)$ Cartier modules on $Y$. If $\dim Y < \dim X$ then we are done by induction. So let us assume otherwise, that $\dim X = \dim Y$. Further truncating the sequence $\sF_i$ we may assume that all $\sF_i$'s have the same generic rank. Now, by definition, $\tau(\sF_0)$ is contained in $\sF_i$ for all $i$ (in fact $\tau(\sF_0)=\tau(\sF_i)$) such that it is enough to show the stabilization of the sequence
\[
    \sF_0/\tau(\sF_0) \supseteq \sF_1/\tau(\sF_0) \supseteq \sF_2/\tau(\sF_0) \supseteq \cdots\, .
\]
But since $\tau(\sF_0)$ generically agrees with each $\sF_i$ this is a sequence of Cartier modules $\sF_i/\tau(\sF_0)$ whose entries have strictly smaller support than $X$. As above, we are done by induction.
\end{proof}
A corollary of the proof is the following result.
\begin{proposition}
Let $\sF$ be a coherent Cartier module on $X$ with surjective structural map. Then the set
\[
    \{ \supp \sF/\sG \,|\, \sG \subseteq \sF \text{ a Cartier submodule} \}
\]
is a finite set of reduced subschemes that is closed under finite unions and taking irreducible components.
\end{proposition}
\begin{proof}
We only prove the finiteness and leave the rest as an exercise \autoref{ex.SupportsCartier}. We proceed by induction on $\dim X$. By \autoref{prop.BasicKashiwaraCartier} we may view $\sF$ as a Cartier module on $\supp \sF$, hence we may assume that $\supp \sF=X$. Since $X$ is Noetherian it has only finitely many irreducible components so we may assume that $X$ itself is irreducible. If $\supp \sF/\sG \neq X$ then $\sF$ and $\sG$ agree on the generic point of $X$. Hence the test module $\tau(\sF) \subseteq \sG$. Therefore
\[
    \supp \sF/\sG \subseteq \supp \sF/\tau(\sF) =: Y
\]
and $Y$ is a proper closed subset of $X$. Again using \autoref{prop.BasicKashiwaraCartier} we can apply the induction hypothesis to the Cartier module $\sF/\tau(\sF)$ on $Y$ whose dimension is strictly less than $\dim X$.
\end{proof}
This yields the following corollary which was obtained in \cite{KumarMehtaFiniteness} and also independently obtained by the second author in \cite{SchwedeFAdjunction}.  In the case that $X = \Spec R$ and $R$ is local, proofs of this fact were first obtained in \cite{SharpGradedAnnihilatorsOfModulesOverTheFrobeniusSkewPolynomialRing} and \cite{EnescuHochsterTheFrobeniusStructureOfLocalCohomology}.
\begin{corollary}
\label{cor.FinitelyManyCompatiblySplitIdeals}
Let $X$ be Frobenius split, then $\O_X$ has only finitely many ideals with are compatible with the splitting.
\end{corollary}
\begin{proof}
If $\phi \colon F^e_* \O_X \to \O_X$ is the splitting of Frobenius, note that $\phi$ is surjective. The Cartier submodules of $\O_X$ are just the ideals which are $\phi$-compatible. Since $\Ann(\O_X/I)=I$ there is a one-to-one correspondence between the set of $\phi$-compatible ideals, and the set $\supp \O_X/I$ for $I$ a Cartier submodule of $\O_X$. The latter set is finite by the preceding proposition.
\end{proof}

\subsection{Cartier Crystals}
The finiteness results for Cartier modules of the preceding section receive a more natural formulation if one deals with the notion of nilpotence in a more systematic manner. This is done by localizing the category of coherent Cartier modules at its Serre subcategory\footnote{\ie~a full Abelian subcategory which is closed under extensions, see \cite{BliBoe.CartierCrys}.} of nilpotent Cartier modules. That is, we invert morphisms which are \emph{nil-isomorphisms}, \ie~maps of Cartier modules $\phi \colon \sF \to \sG$ whose kernel and cokernel are nilpotent. For the formal definition, see \cite{Miyachi,Gabriel-AbelianCat}, but roughly speaking the localization is defined as follows:

\begin{definition}\label{def.CartCrys}
Let $X$ be a scheme. The \emph{category of $\sL$--Cartier crystals} has as objects the coherent Cartier modules on $X$. A morphism $\phi \colon \sF \to \sG$ of Cartier crystals is an equivalence class (left fraction) of diagrams of morphisms of the underlying Cartier modules
\[
    \phi \colon \sF \leftarrow \sF' \to[\phi'] \sG
\]
where $\sF'$ is some Cartier module and $\sF \leftarrow \sF'$ is a nil-isomorphism. More precisely, $$\displaystyle{\Hom_{Crys}(\sF,\sG) = \operatorname{colim}_{\sF' \to \sF} \Hom_{Cart}(\sF',\sF)}$$ where $\sF' \to \sG$ ranges over all nil-isomorphisms.
\end{definition}
It follows from general principles that the category of Cartier crystals on $X$ is again Abelian. Using this point of view the preceding result can be phrased (and extended) as follows, see \cite[Theorem 4.17 and Corollay 4.7]{BlickleBoeckleCartierModulesFiniteness}:

\begin{theorem}\label{thm.CrystalsFiniteness}
Let $X$ be a scheme.
\begin{enumerate}
\item Each Cartier crystal $\sF$ has finite length in the category of Cartier crystals.
\item Hom-sets in the category of Cartier crystals are finite sets (finite dimensional $\mathbb{F}_{p^e}$ vector spaces).
\item Each Cartier crystal $\sF$ has only finitely many Cartier sub-crystals.
\end{enumerate}
\end{theorem}
\begin{proof}
The first statement follows from \autoref{thm.CartierFiniteLength} above by noting that $\sF$ and $\sigma(\sF)$ are isomorphic as Cartier crystals (\ie~nil-isomorphic as Cartier modules). The second statement is shown in \cite[Theorem 4.17]{BlickleBoeckleCartierModulesFiniteness} (but see \autoref{ex.FiniteHoms} below for an idea why such a statement may hold), and the last one follows formally from the other two.
\end{proof}

In \cite{BliBoe.CartierCrys} the category of Cartier crystals (for $\sL \cong \O_X$) is thoroughly studied on an arbitrary Noetherian scheme such that $F : X \to X$ is finite. In particular it is shown that half of Grothendieck's six operations, namely $f^!, Rf_*$ and an exotic tensor product, can be defined on a suitable derived category of Cartier crystals. In particular the construction of the functors $f^!$ and $Rf_*$ is rather subtle and bears some interesting insights. This greatly extends the examples of the pullback for open and closed immersions and the proper push-forward that was discussed in the preceding section.

If $f \colon Y \to X$ is a proper morphism, then $R^if_*$ induces a functor on (coherent) Cartier modules, which can be shown to preserve nilpotence. Hence it descends to a functor on Cartier crystals. However, if $f$ is not proper, then already $f_*\sF$ of a coherent sheaf is no longer coherent. It is a crucial observation in \cite{BliBoe.CartierCrys} that if $\sF$ is a coherent Cartier crystal on $Y$, then $R^if_*\sF$ is a \emph{locally nil-coherent} Cartier crystal on $X$. Nil-coherent for a Cartier module $\sF$ means that $\sF$ has a coherent Cartier submodule $ \sE \subseteq \sF$ such that the quotient $\sF/\sE$ is locally nilpotent, \ie~is the union of nilpotent Cartier submodules. This implies the following result:
\begin{theorem}
For an arbitrary finite type morphism $f \colon Y \to X$, the usual push-forward functor $Rf_*$ on quasi-coherent sheaves induces an exact functor
\[
    Rf_* \colon D_\crys^b(\op{QCrys}(Y)) \to D_\crys^b(\op{QCrys}(X))\, ,
\]
where $D_\crys^b(\op{QCrys}(\usc))$ denotes the bounded derived category of quasi-coherent Cartier crystals whose cohomology is locally nil-coherent.
\end{theorem}
The proof of this result, though not difficult, is somewhat subtle, so we won't attempt it here but instead refer to \cite{BliBoe.CartierCrys}. However the basic idea is already present in \autoref{ex.CartierOpenImmersion}

The situation with the functor $f^!$ is similar but more subtle. As we have already seen in \autoref{subsec.TheTraceMapForSingular}, on quasi-coherent sheaves the construction of the functor $f^!$ is generally quite involved.  Already in the finite case, in particular for a closed immersion $Y \subseteq X$ with $X$ not smooth, one sees that $f^!$ does not have bounded cohomological dimension, hence does not preserve the bounded derived category. However, in \cite{BliBoe.CartierCrys} it is shown quite generally that $f^!$ preserves local nilpotence, and hence induces a functor on quasi-coherent Cartier crystals. The induced functor on Cartier crystals preserves boundedness up to local nilpotence.
\begin{theorem}
If $f \colon Y \to X$ is essentially of finite type, then the twisted inverse image functor $f^!$ on quasi-coherent sheaves induces an exact functor
\[
    f^! \colon D^{b}_\crys(\op{QCrys}(X)) \to D^{b}_\crys(\op{QCrys}(Y))\, .
\]
of bounded cohomological dimension.
\end{theorem}
Besides a number of obvious compatibilities between these functors which are induced from the corresponding ones of the underlying quasi-coherent sheaves, there are two adjointness statements which are important in the theory.
\begin{proposition}
\begin{enumerate}
\item Let $f \colon Y \to X$ be a proper morphism. Then as functors on categories $D^b_\crys(\op{QCrys}(\usc))$ the functor $Rf_*$ is naturally left adjoint to $f^!$.
\item If $j \colon Y \to X$ is an open immersion, then $j_*$ is naturally right adjoint to $j^!=j^*$.
\end{enumerate}
\end{proposition}
For an open immersion $j\colon U\into X$ and a closed complement $i\colon Z\into X$ the above adjunction yields natural isomorphisms $i_*i^!\to\id$ and $\id\to j_*j^*$. This yields the following technically important result regarding their combination:
\begin{theorem}
In $D^b_\crys(\op{QCrys}(X))$, there is a natural exact triangle
\[
    i_*i^!\to \id \to Rj_*j^* \to[+1]
\]
\end{theorem}
This in turn yields a very general form of the Kashiwara equivalence that was alluded to in \autoref{prop.BasicKashiwaraCartier} above.
\begin{theorem}
Let $i \colon Y \to X$ be a closed immersion. Then $i^!$ and $i_*$ define natural isomorphisms
\[
\xymatrix{  D^b_\crys(\op{QCrys}(Y))   \ar@<.5ex>[r]^-{i_*} &   D^{b}_{\crys,Y}(\op{QCrys}(X)) \ar@<.5ex>[l]^-{i^!} }
\]
where the right hand category consists of bounded complexes of quasi-coherent Cartier crystals on $X$ whose cohomology is coherent and supported in $Y$.
\end{theorem}

\subsection{Arithmetic aspects of $p^{-e}$-linear maps}
We conclude with a brief discussion of connections between Cartier crystals and more arithmetic constructions.  What follows is much less explicit than previous sections of the paper, so if the terms used are not familiar to you, we suggest the reader use this as a place jump off for further reading.

The finite length result for Cartier crystals in \autoref{thm.CrystalsFiniteness} suggests -- in analogy with the Riemann-Hilbert correspondence for $D$-modules (\ie modules of the ring of differential operators) on smooth complex manifolds -- a connection of Cartier crystals with a category of constructible sheaves. Indeed, in \cite{Gabber.tStruc} Gabber introduces a family of $t$-structures on the derived category of bounded complexes of constructible $\mathbb{F}_p$-vector spaces on the \'etale site of $X$. He shows that for the middle perversity the heart of this $t$-structure (\ie~the perverse sheaves with respect to this $t$-structure) form an Abelian category which also is Noetherian and Artinian. The connection between Cartier crystals and constructible $\mathbb{F}_p$-vector spaces is a combination of \cite{BoPi.CohomCrys} and \cite{BliBoe.CartierCrys} and yields an equivalence of derived categories:
\[
    D^b_\crys(\op{QCrys}(X)) \to[\cong] D^b_c(X_{et},\mathbb{F}_p)
\]
where the right hand side is the category of constructible sheaves of $\mathbb{F}_p$-vector spaces on $X_{et}$. This correspondence is a two step procedure: First is a Grothendieck-Serre duality between Cartier crystals (coherent $\O_X$-modules with a \emph{right} action of Frobenius) with the category of $\tau$-crystals (coherent $\O_X$-modules with a \emph{left} Frobenius action) of \cite{BoPi.CohomCrys} and was largely motivated by our desire to understand the precise connection of the theory in \cite{BoPi.CohomCrys} with the work of Emerton and Kisin \cite{EmKis.Fcrys} and Lyubeznik \cite{LyubeznikFModulesApplicationsToLocalCohomology}. This Grothendieck-Serre duality is the main result of \cite{BliBoe.CartierCrys}. The step from $\tau$-crystals to constructible sheaves is just by taking Frobenius fix-points, \ie~the Artin-Schreier sequence, see \cite{BoPi.CohomCrys}.

The first author's PhD student Tobias Schedlmeier has shown in his upcoming thesis that the equivalence is given directly by the functor $\Sol(\usc) \colonequals \RHom_{\crys}(\usc,\omega^\mydot_X)$ and proved that the image of the Abelian subcategory of Cartier crystals under $\Sol$ is precisely Gabbers category of perverse sheaves $\op{Perv}(X_{et},\mathbb{F}_p)$ for the middle perversity.

\subsection{Exercises}
\begin{exercise}
\label{ex.ProofOfDescendingChain}
Prove \autoref{lem.ImagesKappaAreCartier}.
\end{exercise}

\begin{exercise}\label{ex.KappaSurjImpliesReduced}
Show that the annihilator of any coherent Cartier module $\sF$ on $X$ with surjective structural map is a sheaf of radical ideals, \ie~its support is reduced.
\end{exercise}

\begin{exercise}\label{ex.localizationFrob}
Let $R$ be a ring and $S \subseteq R$ a multiplicative set. Then for any module $M$ show that $S^{-1}(F_R)_*M \cong (F_{S^{-1}R})_* S^{-1}M$. \\ Hint: Localize with respect to the multiplicative set $S^p$ is the as with respect to $S$.  This generalizes \autoref{ex.LocalizationCompletionAndFeLowerStar}.
\end{exercise}
\begin{exercise}\label{ex.RFdef}
Let $X$ be a scheme and $\sL$ a line bundle. We define a sheaf of rings $\O_X^\sL[F^e]$ as
\[
    \O_X \oplus (\sL \cdot F^e) \oplus (\sL^{1+p^e} \cdot F^{2e}) \oplus (\sL^{1+p^e+p^{2e}} \cdot F^{3e}) \oplus \cdots
\]
where $F^{ne}$ are formal symbols and the multiplication of homogeneous elements $lF^{ne}$ and $l'F^{ne'}$ is defined as $lF^{ne}l'F^{n'e} = l(l')^{p^{n'e}}F^{(n+n')e}$.
\begin{enumerate}
\item Show that this defines the structure of a sheaf of rings on $\O_X^\sL[F^e]$.
\item Show that the category of $(\sL,p^e)$--Cartier modules is equivalent to the category of (sheaves of) right $\O_X^\sL[F^e]$-modules.
\end{enumerate}
Hint: Do the case of $\sL \cong \O_X$ first and then attempt the general case.
\end{exercise}

\begin{exercise}\label{ex.NilpotentExtensions}
If $0 \to \sF' \to \sF \to \sF'' \to 0$ is an exact sequence of coherent Cartier modules. Show that $\sF',\sF''$ are nilpotent (of order $\leq e,e'$) if and only if $\sF$ is nilpotent (of order $\leq e+e'$).
\end{exercise}

\begin{starexercise}\label{ex.SupportsCartier}
Let $\sF$ be a quasi-coherent Cartier module with surjective structural map. Show that the collection
\[
\{ \supp (\sF/\sG)\, |\, \sG \subseteq \sF \text{ a Cartier submodule} \}
\]
is a collection of reduced subschemes that is closed under finite unions and taking irreducible components.
\end{starexercise}

\begin{starexercise}\label{ex.CartierOpenImmersion}
Let $X =\Spec R$ be an affine scheme and $U = \Spec R_f$ a basic open subset with $f \in R$, and denote the open inclusion $U \subseteq X$ by $j$. Let $\sF$ be a coherent Cartier module on $U$. Show that $j_*\sF$ has a coherent Cartier submodule $F$ such that the quotient $j_*\sF/F$ is locally nilpotent, \ie~the union of its nilpotent Cartier submodules.
\end{starexercise}

\begin{exercise}\label{ex.FiniteLengthImpliesTest}
Let $\sF$ be a coherent Cartier module on $X$. The \emph{test submodule} $\tau(\sF)$ is defined as the smallest Cartier submodule $\sG \subseteq \sF$ which agrees with $\sigma(\sF)$ for each generic point of $X$. Show that \autoref{thm.CartierFiniteLength} implies the existence and uniqueness of $\tau(\sF)$.
\end{exercise}

\begin{exercise}
Suppose that $R$ is a ring and $(M, \phi)$ is a Cartier module on $M$.  Suppose further that $R \to S$ is a finite ring homomorphism.  Prove that $\Hom_R(S, M)$ has the structure of a Cartier module induced by $\phi$ and by the Frobenius map $S \to F_* S$.
\end{exercise}

\begin{exercise}
\label{ex.BasicKashiwaraCartier}
Prove \autoref{prop.BasicKashiwaraCartier}.
\end{exercise}

\begin{exercise}
\label{ex.FiniteHoms}
Let $R$ be a regular $F$-finite ring with dualizing sheaf $\omega_R$ with its standard Cartier structure $T : F_* \omega_R \to \omega_R$ (see \autoref{subsec.TraceOfFrob}). Show that the homomorphisms of Cartier modules $\Hom_{\operatorname{Cart}}(\omega_R,\omega_R)=R^F=\mathbb{F}_p$
is just the Frobenius fixed points of the action of $F$ on $R$. In particular, this Hom-set is finite.
\end{exercise}

\begin{starexercise}
\label{ex.FinitelyManyCompatiblySplitIdeals}
Suppose that $R = k[x_1, \dots, x_4]_{\langle x_1, \dots, x_4 \rangle}$ and that $\phi : F^e_* R \to R$ is a Frobenius splitting.  In \autoref{cor.FinitelyManyCompatiblySplitIdeals}, it was shown that there are at most finitely many $\phi$-compatible ideals.

Prove that there at most ${4 \choose d}$ prime ideals $Q$ which are compatibly split by $\phi$ such that $\dim (R/Q) = d$.
\vskip 3pt
\emph{Hint: } Prove it for $d = 0$ first (very easy), then $d = 1$ (use the fact that compatibly split subvarieties must intersect normally, \autoref{cor.IntersectionsOfCompatiblySplitVarietiesAreReduced}, but we only have 4 ``directions'' in $\Spec R$, which is just the origin in $\bA^4$).  For $d = 2, 3$, simply consider all possibilities exhaustively (keeping in mind the normal intersections).  For a complete proof for any $\bA^n$ (not just $n = 4$), see \cite{SchwedeTuckerNumberOfFSplit}.
\end{starexercise}

\begin{exercise}
Suppose that $(\sF, \kappa)$ is an $(\sL, p^e)$-Cartier module on a projective variety $X$ such that the structural map $\kappa : F^e_* (\sL \tensor \sF) \to \sF$ is surjective.  Further suppose that $\sA$ is a globally generated ample line bundle and that $\sN$ is another line bundle such that $\sN^{p^e - 1} \tensor \sL$ is ample.  Prove that
\[
\sF \tensor \sA^{\dim X} \tensor \sN
\]
is a globally generated sheaf.
\vskip 3pt
\emph{Hint: } Use the same strategy as in \autoref{thm.GlobalGenerationOfImages}.
\end{exercise}

\section{Applications to local cohomology and test ideals}
\label{sec.CartierModLocalCohom}

In this section we discuss in detail the relation of the theory of Cartier modules to other theories of modules with a Frobenius action, with an emphasize on applications to local cohomology. Then we discuss a simple but interesting degree-reducing property of Cartier linear maps, which allows an elementary treatment of the theory of Cartier modules in the case that $X$ is of finite type over a perfect field. We use this approach to study the test ideals and show the discreteness of their jumping numbers.

\subsection{Cartier modules and local cohomology}\label{sec.LocalProperties}

The category of Cartier modules, besides enjoying some extraordinary finiteness conditions, is useful due to its connection to other categories which are studied, in particular in connection with local cohomology. Besides the connection to constructible $p$-torsion sheaves that we hinted at above, we show the relation to two further categories which are particularly important in the study of the local cohomology of rings in positive characteristic. Our goal is to explain the following diagram of categories and to derive a number of finiteness results for local cohomology from the above finiteness result for Cartier modules.
\[
\scriptsize
\left\{\begin{array}{c} \text{cofinite $R$-modules} \\ \text{with left Frobenius action} \\ \text{($R$ complete local ring)} \end{array} \right\} \leftrightarrow
\left\{\begin{array}{c} \text{coherent Cartier modules on $X$} \\ \text{($X$ Noetherian and $F$-finite)} \end{array} \right\}
\to
\left\{\begin{array}{c} \text{Lyubeznik's $F$-modules over $R$} \\ \text{($R$ regular, Noetherian ring)} \end{array} \right\}
\]

The parenthetical parts indicate in what generality the categories are defined and the arrows are defined when both assumption holds, for example the first double arrow holds for complete local and $F$-finite rings. The left double arrow is an equivalence of categories given by Matlis duality $\Hom_R(\usc,E_{R/\bm})$ where $E_{R/\bm}$ is an injective hull of the perfect residue field of $R$. The right arrow is a functor which gives an equivalence after inverting Cartier modules  at nil-isomorphisms, that is it induces an equivalence of categories from Cartier crystals to $F$-finite modules. Lyubeznik's $F$-finite modules and this equivalence will be explained in detail below.

Let us begin with Matlis duality. Let $(R,\bm)$ be complete and local and denote by $E=E_R$ an injective hull of the prefect residue field of $R$. Since $R$ is $F$-finite one has that $F^e_* F^!E_R \colonequals \Hom_R(F_*R,E_R)\cong E_{F_*R}$ which we identify with $E_R$ since $R$ and $F_*R$ are isomorphic as rings. We fix hence an isomorphism $F^!E \cong E$. If we denote by $(\usc)^\vee = \Hom_R(\usc,E_R)$ the Matlis duality functor, we have the following lemma whose proof we leave as \autoref{ex.FrobeniusCommutesMatlis}.
\begin{lemma}\label{lem.FrobeniusCommutesMatlis}
For $(R,\bm)$ local and $F$-finite there is a (functorial) isomorphism $F_*(\usc)^\vee \cong (F_*\usc)^\vee$.
\end{lemma}
This immediately implies the first of the equivalences above, also \cf \cite{SharpYoshinoRightLeftModulesOverTheFrobeniusSkewPolynomialRing}.
\begin{proposition}
Let $(R,\bm)$ be complete, local and $F$-finite. Then Matlis duality induces an equivalence between the categories of
\[
\left\{\begin{array}{c} \text{co-finite $R$-modules} \\ \text{with left Frobenius action} \end{array} \right\} \leftrightarrow
\left\{\begin{array}{c} \text{finitely generated $R$-modules} \\ \text{with right Frobenius action} \end{array} \right\}
\]
Of course the $R$-modules with right Frobenius action are just the coherent Cartier modules on $X=\Spec R$. The equivalence preserves nilpotence.
\end{proposition}
\begin{proof}
A left action of Frobenius on $M$ is an $R$-linear map $\phi \colon M \to F_*M$. Applying Matlis duality and the preceding lemma this yields a map
\[
    F_*(M^\vee) \cong (F_*M)^\vee \to[\phi^\vee] M^\vee
\]
which is the desired Cartier structure (=right Frobenius action) on the dual $M^\vee$. The same construction works in the opposite direction and one immediately checks that this induces an equivalence of categories.
\end{proof}
With this result we can translate the finiteness theorems for Cartier modules obtained above to the setting of cofinite $R$-modules with a left Frobenius action. In particular the results hold for local cohomology modules $H^i_{\bm}(R)$ with support in the maximal ideal $\bm$ of $R$.
\begin{theorem}
Let $N$ be a cofinite $R$ module equipped with a $p$-linear map $F \colon N \to N$ (\ie~$F$ is additive and $F(rm)=r^pF(m)$).
\begin{enumerate}
\item The ascending chain of submodules $\ker F \subseteq \ker F^2 \subseteq \ker F^3 \subseteq \cdots$ stabilizes (\cite[Proposition 1.1]{HartshorneSpeiserLocalCohomologyInCharacteristicP}).
\item Any chain $ \cdots \subseteq N_i \subseteq N_{i+1} \subseteq N_{i+2} \subseteq \cdots $ of submodules $N_i \subseteq N$ which are stable under $F$ (\ie~$F(N_i) \subseteq N_i$) has eventually $F$-nilpotent quotients (\cite[Theorem 4.7]{LyubeznikFModulesApplicationsToLocalCohomology})
\item $N$ has up to nilpotent action of $F$, only finitely many $F$-stable submodules. Concretely, there are only finitely many $F$ stable submodules $N'$ for which the action of $F$ on the quotient $N/N'$ is injective.
\end{enumerate}
\end{theorem}
\begin{proof}
These are just the Matlis dual statements of \autoref{prop.CartierImagesStabilize}, \autoref{thm.CartierFiniteLength}, and \autoref{thm.CrystalsFiniteness} part (c).
\end{proof}
An immediate consequence of these observations is the following result originally obtained by Enescu and Hochster \cite{EnescuHochsterTheFrobeniusStructureOfLocalCohomology}, see \cite{MaFiniteLocalCohom} for a recent extension showing that $F$-split alone is sufficient in the assumptions below.
\begin{proposition}
If $R$ is quasi-Gorenstein (\ie~$H^d_\bm(R)\cong E_R$) and $F$-split, then the top local cohomology module $H^d_\bm(R)$ with its left action of the Frobenius has only finitely many $F$-stable submodules.
\end{proposition}
\begin{proof}
The existence of a splitting $\phi \colon R \to S$ implies that the Cartier module $(R,\phi)$ has only finitely many Cartier submodules. Hence, by the above duality result its dual $(r^\vee=H^d_\bm(R),\phi^\vee)$ has only finitely many submodules stable under the action of $\phi^\vee$. But $H^d_\bm(R)$ also has a natural Frobenius action $F_H$ induced by the Frobenius on $R$ by functoriality of $H^d_\bm(\usc)$. One can show (\autoref{ex.FrobOnLocalCohom}) that there is a $r \in R$ such that $\phi^\vee = r F_H$. Hence all submodules which are stable under $F_H$ are also stable under $\phi^\vee$, but of the latter there are only finitely many as just argued.
\end{proof}

The connection of Cartier modules with Lyubeznik's $F$-finite modules also relies on a certain commutation of functors which we recall first. Lyubeznik's theory \cite{LyubeznikFModulesApplicationsToLocalCohomology} is phrased for a regular ring $R$, and even though there is an extension to schemes by Emerton and Kisin \cite{EmKis.Fcrys}, we will stick to this setting and assume from now on that $X=\Spec R$, with $R$ regular (and such that the Frobenius morphism $F: R \to R$ is finite).
\begin{lemma}\label{lem.ShreickCommutesTensor}
Let $f \colon Y \to X$ be a finite flat morphism and $M$ a $\O_X$ module, then there is a functorial isomorphism
\[
    f^!\O_X \tensor_{\O_Y} f^*M  \cong f^!M\, ,
\]
where $f^!(\usc) = \Hom_{\O_X}(f_*\O_X,\usc)$.
\end{lemma}
\begin{proof}
See \autoref{ex.ShrieckCommutesTensorRegular}.
\end{proof}
Applying this to the case of the Frobenius on the regular scheme $X$ and $M = \omega_X$ the dualizing sheaf (which is invertible!) we obtain an isomorphism
\[
    F^!\O_X \cong F^!\omega_X \tensor F^*\omega_X^{-1}\, .
\]
Further using that the adjoint of the map $F_*\omega_X \to \omega_X$ coming from the Cartier isomorphism in \autoref{defprop.CartierIsomorphism}, is an isomorphism $\omega_X \to F^!\omega_X$ we obtain
\[
    F^!M \tensor \omega_X^{-1} \cong F^*M \tensor F^! \omega_X \tensor \omega_X^{-1} \tensor F^* \omega \cong F^*(M \tensor \omega_X^{-1})
\]
which allows us to describe the functor from Cartier modules to Lyubeznik's $F$-finite modules. Starting with a Cartier module $M$ with structural map $\kappa \colon F_*M \to M$ we first consider its adjoint $\kappa' \colon M \to F^!M$ and tensor it with $\omega_X^{-1}$ to obtain
\[
    \gamma \colon M \tensor \omega_X^{-1} \to[\kappa' \tensor \id] F^!(M) \tensor \omega_X^{-1} \cong F^*(M \tensor \omega_X^{-1})
\]
where the final isomorphism is the one derived above. Let us pause for a moment to recall the definition of Lyubeznik's $F$-finite modules, which we phrase in a way convenient for our purpose:
\begin{definition}
Let $R$ be regular. Given a finitely generated $R$-module $N$ together with a map $\gamma \colon N \to F^*N$, then an \emph{$F$-finite module} is the limit $\sN$ of the directed system
\[
    N \to[\gamma] F^*N \to[F^*\gamma] F^{2*}N \to[F^{2*}\gamma] F^{3*}N \to \cdots
\]
together with the induced map $\theta \colon \sN \to[\cong] F^*\sN$ which is immediately verified to be an isomorphism.

Phrased differently, an \emph{$F$-finite module} is a (not necessarily finitely generated) $R$-module $\sN$ together with an isomorphism $\theta \colon \sN \to[\cong] F^*\sN$ which arises in the above described manner from a \emph{finitely generated} $R$-module $N$.
\end{definition}
It is shown in \cite{LyubeznikFModulesApplicationsToLocalCohomology} that $F$-finite modules are an Abelian category which is closed under extensions, that local cohomology modules $H^i_I(R)$ are $F$-finite modules, and that $F$-finite modules enjoy a number of important finiteness results. For example they have only finitely many associated primes, and all Bass numbers are finite.

From this definition it is immediate how to connect the Cartier modules with $F$-finite modules. The $F$-finite module attached to a Cartier module $M$ is just the limit of
\[
    M \tensor \omega_X^{-1} \to F^*(M \tensor \omega_X^{-1}) \to F^{2*}(M \tensor \omega_X^{-1}) \to \cdots
\]
One obtains the following Proposition \cite{BlickleBoeckleCartierModulesFiniteness}.
\begin{proposition}
For a regular ring $R$, the just described construction assigning to a coherent Cartier module $M$ on $R$ an $F$-finite module is an essentially surjective functor
\[
    \left\{ \text{coherent Cartier modules} \right\} \to \left\{ \text{$F$-finite modules} \right\}
\]
which sends nilpotent Cartier modules to zero. The induced functor
\[
    \left\{ \text{coherent Cartier \emph{crystals}} \right\} \to[\cong] \left\{ \text{$F$-finite modules} \right\}
\]
is an equivalence of categories.
\end{proposition}
\begin{proof}
All statements are shown in \cite{BlickleBoeckleCartierModulesFiniteness} but with the above preparations none of them is particularly difficult.
\end{proof}
Hence we obtain as an immediate consequence of \autoref{thm.CrystalsFiniteness} the following finiteness result for $F$-finite modules, which partially extends one of the main results of \cite{LyubeznikFModulesApplicationsToLocalCohomology}:
\begin{theorem}
Let $R$ be regular and $F$-finite, then
\begin{enumerate}
\item $F$-finite modules over $R$-have finite length.
\item The $\Hom$-sets in the category of $F$-finite modules are finite.
\item An $F$-finite module has only finitely many $F$-finite submodules.
\end{enumerate}
\end{theorem}
Part (a) of the theorem has been proven for $R$ regular and of finite type over a regular local ring in \cite{LyubeznikFModulesApplicationsToLocalCohomology}, and for arbitrary $F$-finite schemes $X$ in \cite{BliBoe.CartierCrys}.  The latter results also are shown for regular rings (part (b) even without the $F$-finiteness assumption) in \cite{HochsterSomeFinitenessPropertiesOfLyubeznik}.  Finally, let us state the aforementioned finiteness result for local cohomology modules.
\begin{theorem}
Let $M$ be an $F$-finite module, $I \subseteq R$ an ideal in a regular ring, then $H^j_I(M)$ is an $F$-finite $R$-module and hence has only finitely many associated primes. \end{theorem}
\begin{proof}
We only have to show that $H^j_I(M)$ is an $F$-finite module. The crucial step is to show  that for $f \in R$ we have that the localization $M_f$ is also $F$-finite (\cf~\autoref{ex.CartierOpenImmersion}). Once this is established, the \Cech-complex finishes the proof.
\end{proof}

\subsection{Contracting property of $p^{-e}$-linear maps}
In this section we point out a simple fact about $p^{-e}$-linear map which has a number of interesting consequences. In particular we give an elementary proof of the finite length result for Cartier modules.  The idea goes back at least to a paper of Anderson \cite{AndersonElementaryLFunctions} and says that a $p^{-e}$-linear endomorphism reduces the degree in a graded context. For this we consider $X = \Spec S$ with $S = k[x_1,\ldots,x_n]$ a polynomial ring over a perfect field. Then we consider the filtration of $S$ given by the finite-dimensional vector spaces
\[
    S_d \colonequals k \langle x_1^{i_1}\cdots x_n^{i_n}\, |\, 0 \leq i_j \leq d \text{ for } j = 1, \ldots, n \rangle\, .
\]
Hence $S_d$ is the $k$--subspace of $S$ freely generated by the monomials with degree $\leq d$ \emph{in each variable}. One immediately verifies that
\[
    S_{-\infty} \colonequals 0, \quad S_0 = k, \quad S_dS_{d'} \subseteq S_{d+d'},\text{ and} \quad S_d+S_d' \subseteq S_{\max{d+d'}}\, .
\]
For each choice of a set of generators $m_1,\ldots,m_k$ of an $S$ module $M$ we define the induced filtration on $M$ given by
\[
    M_{-\infty} \colonequals 0 \text{ and } M_d = S_d \langle m_1,\ldots,m_k \rangle.
\]
For $m \in M$ we write $\delta(m)=d$ if and only if $m \in M_d\setminus M_{d-1}$ and call $\delta=\delta_M$ a \emph{gauge} for $M$. One should think of the gauge $\delta$ as a substitute for a degree on $M$, and the contracting property of $p^{-e}$-linear maps on $M$ is measured in terms of the gauge $\delta$. Spelling out the definition we see that $\delta(m)\leq d$ if $m$ can be written as a $S$-linear combination of the $m_i$ such that all coefficients are in $S_d$. $S$ itself has a gauge, induced by the generator $1$.
We summarize the immediate properties of a gauge (the proof is left to the reader in \autoref{ex.PropertiesOfGauge}):
\begin{lemma}
\label{lem.PropertiesOfGauge}
Let $M$ be finitely generated over $S=k[x_1,\ldots,x_n]$, and $\delta$ a gauge corresponding to some generators $m_1,\ldots,m_k$ of $M$. Then
\begin{enumerate}
\item $\delta(m)=-\infty$ if and only if $m=0$.
\item Each $M_d$ is finite dimensional over $k$ (since $S_{d}$ is).
\item $\bigcup_d M_d = M$ (since the $m_i$ generate $M$).
\item $\delta(m+m') \leq \max\{\delta(m),\delta(m')\}$
\item $\delta_M(fm) \leq \delta_S(f)+\delta_M(m)$
\end{enumerate}
\end{lemma}

\begin{proposition}[\protect\cite{AndersonElementaryLFunctions}, Proposition 3]\label{prop.gaugebound}
Let $M$ be a finitely generated $S$-module, and $\delta=\delta_M$ a gauge corresponding to some generators $m_1,\ldots,m_k$ of $M$ and let $\phi:M \to M$ be a $p^{-e}$-linear map. Then there is a constant $K$ such that for all $m \in M$:
\[
    \delta(\phi(m)) \leq \frac{\delta(m)}{p^e}+\frac{K}{p^e}
\]
Furthermore, for all $n \geq 0$ we have
\[
    \delta(\phi^n(m)) \leq \frac{\delta(m)}{p^{ne}}+\frac{K}{p^e-1}\, .
\]
\end{proposition}
\begin{proof}
By definition, we may write $m = \sum_{l=1}^k f_l m_l$ with $\delta_S(f_l) \leq \delta(m)$. For each $l$ write uniquely $f_l = \sum_{{\bf x^i} \in S_{\delta(m)}} r_{l,i}^{p^e}{\bf x^i}$ with ${\bf x^i}=x_1^{i_1}\cdots x_n^{i_n}$. Then \autoref{ex.GaugeOnCoeff}  shows that $\delta_S (r_{l,i}) \leq \lfloor \delta(m)/{p^e} \rfloor$. Writing this out
\[
    \phi(m)=\sum_{l=1}^k \sum_{x^i \in S_{\delta(m)}} r_{l,i}\phi({\bf x}^i m_l)
\]
we consequently obtain
\[
    \delta(\phi(m)) \leq \max_{l,i} \{\delta_S(r_{l,i})+\delta(\phi(x^i m_l))\} \leq \lfloor \frac{\delta(m)}{p^e} \rfloor +\frac{K}{p^e}
\]
taking for $K = {p^e} \cdot \max_{l,i} \{\delta(x^i m_l)\}$ we obtain the claimed inequality.

The final inequality follows by applying the first inequality iteratively and then to use the geometric series (exercise!).
\end{proof}

This proposition has an important consequence about the generators of the images of a submodule under a $p^{-e}$-linear map.

\begin{lemma}\label{lem.GaugeBoundForImage}
Let $\phi \colon M \to M$ be a $p^{-e}$-linear map on the finitely generated $S$-module $M$ with gauge $\delta$ and bound $K$ as in \autoref{prop.gaugebound}. Suppose that the $S$-submodule $N \subseteq M$ is generated by elements with gauge $\leq d$. Then $\phi^n(N) \subseteq M$ is generated by elements of gauge at most $d/p^{ne}+K/(p^e-1)+1$.
\end{lemma}
\begin{proof}
If $N$ is generated by $n_1,\ldots,n_t$, then $\phi^n(N)$ is generated by $\phi(x^i n_j)$ where $0 \leq i_1,\ldots,i_n \leq p^{ne}-1$ and $j = 1, \ldots, t$. Now, if each $\delta(n_j) \leq d$ then
\[
    \delta(\phi(x^i n_j)) \leq \frac{\delta(x^i n_j)}{p^{ne}} + \frac{K}{p^{e}-1} \leq \frac{(p^{ne}-1)+d}{p^{ne}} + \frac{K}{{p^e}-1} \leq 1 + \frac{d}{p^{ne}} + \frac{K}{{p^e}-1}
\]
\end{proof}
\begin{corollary}
Let $\phi \colon M \to M$ be a Cartier module with gauge $\delta$ and bound $K$ (as in \autoref{prop.gaugebound}). Then every Cartier submodule $N \subseteq M$ with surjective structural map $\phi \colon N \onto N$ is generated by elements in the finite dimensional $k$-vector space $M_{\frac{K}{p^e-1}+1}$ (independently of $N$).
\end{corollary}
\begin{proof}
If $N$ has surjective structural maps, then for each $n$ we have $\phi^n(N)=N$. Since $N$ is finitely generated it is generated by elements of some gauge $\leq d$. By the above Lemma, we have hence for all $n$ that $N$ is generated by elements of gauge $\leq d/p^{ne}+K/(p^e-1)+1$. But for $n$ big enough the first term is irrelevant (less than 1), and the result follows.
\end{proof}

\begin{corollary}
In a coherent Cartier module $M$ there are no infinite proper chains of Cartier submodules $N_i$ each with surjective structural map.
\end{corollary}
\begin{proof}
Each $N_i$ has generators in the finite dimensional $k$-vector space $M_{\frac{K}{p^e-1}+1}$, hence there cannot be any infinite proper chains.
\end{proof}

As we have alluded to (in \autoref{ex.FiniteLengthImpliesTest}) before, the fact that there are no infinite chains of Cartier submodules with surjective structural maps implies the existence of the test module $\tau(M)$. By definition of being the smallest Cartier submodule of $M$ which generically agrees with $\sigma(M)$ it is clear that $\tau(M)$ has surjective structural map (since the image under the structural map would again be of that type). The intersection of two Cartier submodules agreeing generically with $\sigma(M)$ clearly also has this property. Now, the existence of $\tau(M)$ follows from the stabilization of any chain of submodules generically agreeing with $\sigma(M)$ and with surjective structural maps, which we just showed.

In the next section, \autoref{subsec.AlgebrasOfMaps}, we will show this approach to Cartier modules via gauges also gives an elementary proof of the discreteness of jumping numbers for test ideals.

We conclude this section with pointing out that our restriction to the polynomial ring $S=k[x_1,\ldots,x_n]$ is not very restrictive after all. The case of an arbitrary scheme $X$ we may reduce to the affine case by considering an affine cover. Then any finite type $k$-algebra $R=S/I$ is the quotient of a polynomial ring. Then we can use the Kashiwara-equivalence \autoref{prop.BasicKashiwaraCartier} to reduce to the case of the polynomial ring itself.

\begin{remark}[Historical discussion]
The major source of inspiration to explore the contracting property of $p^{-e}$-linear maps in \cite{BlickleTestIdealsViaAlgebras} came from a paper of Anderson \cite{AndersonElementaryLFunctions} where he uses this property to study $L$-functions $\mod p$ on varieties over $\mathbb{F}_p$. The key observation there is that if $\phi \colon M \to M$ is a $p^{-e}$-linear map of $R$-modules (say $R$ of finite type over $\mathbb{F}_p$) then there is a finite dimensional $\mathbb{F}_p$-subspace into which every element of $M$ is eventually contracted by iterated application of $\phi$. This allows him, inspired by Tate's work \cite{Tate.ResiduesDifferentialsCurves} to develop a trace calculus for these operators. This is then used to show the rationality of $L$-functions $\mod p$ attached to a finitely generated $R$-module $M$ with a left action of Frobenius $F$ on $M$. In fact, he shows that if $R$ is the polynomial ring and $M$ is projective, this $L$-function is equal to the characteristic polynomial (or its inverse) of the action of the dual of $F$ on $M^\vee$. This dual is a Cartier linear endomorphism of $M^\vee$ and the characteristic polynomial is defined via the important contracting property of Cartier linear maps.
\end{remark}

\subsection{Algebras of maps and the test ideal}\label{subsec.AlgebrasOfMaps}

Suppose that $X = \Spec R$ is an affine variety (for simplicity).  Previously we considered finitely generated $R$-modules $M$ and $p^{-e}$-linear maps $\phi : M \to M$.  Unless $M = \omega_R$ (or is obtained functorially from $T : \omega_Y \to \omega_Y$ from some other variety $Y$), there probably is no \emph{natural} choice of $\phi$.

The obvious solution is to choose all possible $\phi$, see \cite{SchwedeTestIdealsInNonQGor,BlickleTestIdealsViaAlgebras} and \cf \cite{LyubeznikSmithCommutationOfTestIdealWithLocalization} for a dual formulation.  For any \emph{finitely generated} module $M$, we set $\End_e(M)$ to be the set of $p^{-e}$-linear maps from $M$ to $M$.  In other words, $\End_e(M)$ is just $\Hom_R(F^e_* M, M)$.  Of course $\End_e(M)$ has an $R$-module structure via both the source and target $R$-module structures.  Notice that if $\phi \in \End_e(M)$ and $\psi \in \End_d(M)$, then we can form the composition $\psi \circ \phi \in \End_{e+d}(M)$.  Thus $\End_*(M) = \oplus_{e \geq 0} \End_e(M)$ forms a non-commutative graded ring.  Unfortunately, the ring $\End_0(M)$ is often too big and so we set $\sC^M_0$ to denote the image of $R$ inside $\End_0(M)$ via the natural map that sends $r \in R$ to the multiplication by $r$ map on $M$.

\begin{definition}[Cartier Algebras]
An \emph{(abstract) Cartier algebra} over $R$ \footnote{It is important to note that while we call it an \emph{algebra}, it is not generally an $R$-algebra because $R$ is not central.} is a $\bN$-graded ring $\sC = \bigoplus_{e \geq 0} \sC_e$ satisfying the rule $r \cdot \phi_e = \phi_e \cdot r^{p^e}$ for all $\phi_e \in \sC_e$ and $r \in R$ and furthermore such that $\sC_0 \cong R/I$ for some ideal $I$.
\end{definition}

\begin{example}
Suppose that $M$ is a finitely generated $R$-module.  The \emph{total Cartier algebra on $M$}, denoted $\sC^M$, is the following graded subring of $\End_*(M)$.
\[
\sC^M := \sC^M_0 \oplus \left(\bigoplus_{e > 0} \End_e(M) \right) = \bigoplus_{e \geq 0} \sC^M_e.
\]
It is obviously a Cartier algebra.

A \emph{Cartier-subalgebra (on M)} is any graded subring $\sC \subseteq \sC^M$ such that $[\sC]_0 = \sC^M_0$.

%If $\phi$ is a $p^{-e}$-linear endomorphism on $M$, then the Cartier-subalgebra generated by $\phi$ is isomorphic to $R[F^e]^{\mathrm{op}}$ (here $F^e$ is just a formal symbol used to denote the $e$th Frobenius).
\end{example}

With the above definitions, if $\sC$ is an (abstract) Cartier algebra, and $M$ is any left-$\sC$-module, then there is a natural map $\sC \to \sC^M$, the image of which is a Cartier-subalgebra on $M$.  Conversely, note that any Cartier-subalgebra $\sC \subseteq \sC^M$ acts on $M$ by the application of functions.  In particular, $M$ is also a $\sC$-module.

\begin{remark}
Most commonly, we will consider $\sC^R$, in which case $\sC^R_0 = \Hom_R(R, R) = \End_0(R)$ automatically.
\end{remark}

Now suppose that $\sC$ is a Cartier-algebra and that $M$ is a left $\sC$-module (or that $\sC$ is a Cartier-submodule on $M$), we use $\sC_+$ to denote $\oplus_{e > 0} \sC_e$.  It is easy to see that $\sC_+$ is a 2-sided ideal.  For any $\sC$-submodule $N \subseteq M$, we define
\[
\sC_+ N := \langle  \phi(x) \,|\, x \in N, \phi \in \sC_e \text{ for some $e > 0$} \rangle_R \subseteq N
\]
to be the submodule generated by all $\phi(x)$ for homogeneous $\phi \in \sC_+$ and $n \in N$.  We set
\[
(\sC_+)^n N := \underbrace{\sC_+( \sC_+( \cdots \sC_+}_{\text{$n$-times}} (N))) \subseteq N.
\]
A crucial step from dealing with an algebra of Cartier linear operators as opposed to a single one, is to establish the right notion of nilpotence. With following definition the theory develops in surprising analogy to the single operator case dealt with above.
\begin{definition}
We say that $N$ is \emph{$\sC$-nilpotent} if $(\sC_+)^n N = 0$ for some $n > 0$.
\end{definition}

It is obvious we have a chain of inequalities:
\begin{equation}
\label{eq.ChainOfAlgebraSubmodules}
N \supseteq \sC_+ N \supseteq (\sC_+)^2 N \supseteq \dots \supseteq (\sC_+)^i N \supseteq (\sC_+)^{i+1} N \supseteq \cdots
\end{equation}
The following remarkable theorem about this chain generalizes \autoref{prop.CartierImagesStabilize} above.

\begin{theorem}\cite[Proposition 2.14]{BlickleTestIdealsViaAlgebras}
\label{thm.ChainOfAlgebraSubmodulesStabilizes}
Suppose that $M$ is a finitely generated $R$-module that is also a left $\sC$-module for some Cartier algebra $\sC$.  Then $(\sC_+)^n M = (\sC_+)^{n+1} M$ for all $n \gg 0$.  In other words, the chain of submodules in \autoref{eq.ChainOfAlgebraSubmodules} eventually stabilizes.
\end{theorem}
\begin{proof}
The proof is similar to that of \autoref{prop.CartierImagesStabilize} and left to the reader in \autoref{ex.ChainOfAlgebraSubmodulesStabilizes}.
\end{proof}

As an immediate corollary we obtain:

\begin{corollary} \cite[Corollary 2.14]{BlickleTestIdealsViaAlgebras}
\label{cor.MUnderline}
Let $M$ be a finitely generated $R$-module that is also an $\sC$-module for some Cartier algebra $\sC$.  Then there is a unique $\sC$-submodule $\sigma(M) \subseteq M$ such that
\begin{enumerate}
\item  the quotient $M / \sigma(M)$ is nilpotent, and
\item  $\sC_+ \sigma(M) = \sigma(M)$ and so $\sigma(M)$ does not have nilpotent quotients.
\end{enumerate}
\end{corollary}
\begin{proof}
Set $\sigma(M) = (\sC_+)^n M$ for $n \gg 0$, then verify the statements in \autoref{ex.MUnderline}.
\end{proof}

Suppose that $M$ is a finitely generated $R$-module and a left $\sC$-module.  We can now define a notion of the test ideal on $M$.

\begin{definition}
Suppose that $M$ and $\sC$ are as above.  Then we define the \emph{test submodule} $\tau(M, \sC)$ to be the unique smallest submodule $N$ of $M$ which
\begin{enumerate}
\item is a $\sC$-module,
\item which satisfies $(\sigma(M))_{\eta} = N_{\eta}$ for every minimal prime of $R$.\footnote{This definition differs slightly from the original one given in \cite{BlickleTestIdealsViaAlgebras} where one requires equality for every minimal prime of $\sigma(M)$ instead of $R$. Though this yields different results in general, in light of the Kashiwara equivalence \autoref{prop.BasicKashiwaraCartier} the respective theories imply each other.}
\end{enumerate}
\emph{if} it exists.
\end{definition}
The existence of $\tau(M, \sC)$ is known in many important cases, but not in all generality.  It is known to exist if $R$ is of finite type over a field (or a localization of such), or if $\sC$ is generated by a single operator, see \cite[Theorem 4.13, Corollary 3.18]{BlickleTestIdealsViaAlgebras}.  It is also known to exist if $M = R$ by the same argument as \autoref{prop.TauExistsForRings} below.

For the rest of the section, we consider $\sC^R$, the total Cartier algebra on $R$, and subalgebras of it.  Indeed, a common way to construct a Cartier algebra is as follows.

\begin{definition}
\label{def.CartierAlgebraFromTriple}
Suppose that $R$ is a normal domain with $X = \Spec R$.  Suppose further that $\Delta \geq 0$ is an effective $\bQ$-divisor, $\ba \subseteq R$ is a nonzero ideal and $t \geq 0$ is a real number.  Then we define the following Cartier subalgebra of $\sC^R$.  For each $e \geq 0$ first identify $\Hom_R(F^e_* R, R)$ with $\sC^R_e$ and fix $\sC^{\Delta}_e$ to be the subset
\[
\Hom_R(F^e_* R(\lceil (p^e - 1)\Delta \rceil), R) \subseteq \Hom_R(F^e_* R, R) = \sC^R_e.
\]
Here $R(\lceil (p^e - 1)\Delta \rceil) = \Gamma(X, \O_X( \lceil (p^e - 1)\Delta \rceil))$.

It follows that
\[
\sC^{\Delta} := \bigoplus_{e \geq 0} \sC^{\Delta}_e
\]
is a Cartier subalgebra of $\sC^R$ (the details will be left as \autoref{ex.CartierAlgebraOfDelta}).

Furthermore, we can form $\sC^{\Delta, \ba^t}_e := \sC^{\Delta}_e \cdot \ba^{\lceil t(p^e - 1) \rceil}$ (where multiplication on the right is pre-composition, in other words $\sC^{\Delta, \ba^t}_e$ is identified with $\Hom_R(F^e_* R(\lceil (p^e - 1)\Delta \rceil), R) \cdot (F^e_* \ba^{\lceil t(p^e -1) \rceil})$ ).  Again the direct sum
\[
\sC^{\Delta, \ba^t} := \bigoplus_{e \geq 0} \sC^{\Delta, \ba^t}_e
\]
is a Cartier-subalgebra of $\sC^R$, see \autoref{ex.CartierAlgebraOfDelta}.
\end{definition}

With these definitions, we can now define the test ideal $\tau(R; \Delta, \ba^t) := \tau(R, \sC^{\Delta, \ba^t})$ \cite{SchwedeTestIdealsInNonQGor,BlickleTestIdealsViaAlgebras}.

\begin{remark}
Test ideals (with $\Delta = 0$ and $\ba = R$) were originally introduced by Hochster and Huneke in their theory of tight closure \cite{HochsterHunekeTC1}.  In fact, what we call the test ideal is often called the \emph{big test ideal} \cite{HochsterFoundations} and is denoted by $\tld \tau$ or $\tau_b$.  This object though is better behaved with respect to geometric operations (such as localization \autoref{ex.TauLocalizes}).  It is conjectured that $\tld \tau$ and $\tau$ coincide in general \cite{LyubeznikSmithStrongWeakFregularityEquivalentforGraded,LyubeznikSmithCommutationOfTestIdealWithLocalization}.

Even with $\Delta \neq 0$ and $\ba \neq R$, this definition we gave is not the original one.  For $\ba \neq R$, $\tau(R; \ba^t)$ was originally defined in \cite{HaraYoshidaGeneralizationOfTightClosure} (and $\tld \tau(R; \ba^t)$ was studied in \cite{HaraTakagiOnAGeneralizationOfTestIdeals}).  For $\Delta \neq 0$, $\tau(R; \Delta)$ was introduced in \cite{TakagiInterpretationOfMultiplierIdeals}.
\end{remark}

\begin{proposition}
\label{prop.TauExistsForRings}
Suppose $R$ is a normal domain.
The test ideal $\tau(R, \sC^{\Delta, \ba^t}) = \tau(R; \Delta, \ba^t)$ exists.
\end{proposition}
\begin{proof}
The main point is the following Lemma, which is a generalization of a result of Hochster and Huneke.
\begin{lemma} \cite[Section 6]{HochsterHunekeTC1} \cite[Lemma 3.21]{SchwedeTestIdealsInNonQGor}
There exists an element $0 \neq c \in R$ such that for every $0 \neq d \in R$, there exists $e > 0$ such that $c \in \sC^{\Delta, \ba^t}_{e}( dR)$.
\end{lemma}
Now choose $c$ as in the Lemma, and it follows that $c \in I$ for any non-zero $\sC^{\Delta, \ba^t}$-submodule $I \subseteq R$.  However,
\[
\displaystyle\sum_{e \geq 0} \sC^{\Delta, \ba^t}_e(Rc)
\]
is evidently the smallest $\sC^{\Delta, \ba^t}$ submodule containing $c$.
\end{proof}

One of the aspects of the test ideal which has attracted the most interest over the past few years is how the test ideal $\tau(R; \Delta, \ba^t)$ changes as $t$ varies.  First we mention the following lemma which serves as a baseline for how the test ideal behaves.

\begin{lemma} \cite[Remark 2.12]{MustataTakagiWatanabeFThresholdsAndBernsteinSato}, \cite[Proposition 2.14]{BlickleMustataSmithDiscretenessAndRationalityOfFThresholds}, \cite[Lemma 3.23]{BlickleSchwedeTakagiZhang}
\label{lem.TauDoesntChangeFromAbove}
With notation as above, for every real number $t \geq 0$, there exists an $\epsilon > 0$ such that
\[
\tau(R; \Delta, \ba^t) = \tau(R; \Delta, \ba^{s})
\]
for every $s \in [t, t+\epsilon]$.
\end{lemma}
\begin{proof}
The containment $\supseteq$ is obvious.  A substantial hint is given in \autoref{ex.TauDoesntChangeFromAbove}.
\end{proof}

Because of this, we make the following definition:

\begin{definition}[$F$-jumping numbers]
Suppose that $(R, \Delta, \ba^t)$ are as above.  Then a number $t > 0$ is called an \emph{$F$-jumping number} if
\[
\tau(R; \Delta, \ba^t) \neq \tau(R; \Delta, \ba^{t - \varepsilon})
\]
for all $1 \gg \varepsilon > 0$.
\end{definition}

Based on the above Lemma, and a connection between test ideals and multiplier ideals \cite{HaraYoshidaGeneralizationOfTightClosure,TakagiInterpretationOfMultiplierIdeals} it is natural to expect that the set of jumping numbers for the test ideal is discrete. In the case that $X$ is smooth this was shown to be the case in \cite{BlickleMustataSmithDiscretenessAndRationalityOfFThresholds} and \cite{BlickleMustataSmithFThresholdsOfHypersurfaces}. The singular case was obtained in \cite{BlickleSchwedeTakagiZhang}, see also \cite{HaraMonskyFPureThresholdsAndFJumpingExponents,KatzmanLyubeznikZhangOnDiscretenessAndRationality,SchwedeTakagiRationalPairs,TakagiTakahashiDModulesOverRingsWithFFRT,SchwedeTuckerZhangTestIdealsSingleAlteration,AlvarezBoixZarzuelaCartierAlgebrasStanleyReisner}. We will outline here an elementary proof based on the contracting property of $p^{-e}$ linear maps that was investigated in the preceding section. In order to be able to handle not only a single $p^{-e}$-linear map but a whole Cartier algebra, we need to generalize the results on gauge bounds obtained above slightly. To keep things simple we will consider a Cartier algebra of the type
\[
    \sC = \bigoplus R \cdot \phi^n \ba^{\lceil t(p^{ne}-1) \rceil}
\]
where $\ba$ is an ideal in $R$, $t \geq 0$ is a real number and $\phi$ is a single $p^{-e}$-linear operator on $R$. This is essentially the case $\sC = \sC^{\Delta,\ba^t}$ for $(p^e-1)(K_R +\Delta)$ is a Cartier divisor (\ie~the pair $(R,\Delta)$ is $\mathbb{Q}$-Cartier with index not divisible by $p$). We first state a generalization of \autoref{lem.GaugeBoundForImage} to this context.
%\todo{\textbf{mnl:} should we explain this better of exercise it?\\ {\bf Karl: } Maybe exercise it with a hint?}
\begin{lemma}
Let $\sC$ be the Cartier subalgebra of $\sC^R$ generated by $\phi$, a $p^{-e}$-linear map on $R=k[x_1,\ldots,x_n]/I$. Let $M$ be a coherent $\sC$ module and suppose that for all $m \in M$ and $n > 0$ one has
\[
    \delta(\phi^n(m)) \leq  \frac{\delta(m)}{p^{ne}}+\frac{K}{p^e-1}
\]
for some bound $K \geq 0$ as in \autoref{prop.gaugebound}. Then, if $\ba \subseteq R$ is an ideal generated by element of gauge $\leq d$, and $N \subseteq M$ is a $R$-submodule generated by elements of gauge $\leq D$, then $(\sC^{\ba^t}_+)^n(N) \subseteq N$ is generated by elements of gauge $\leq \frac{D}{p^{ne}}+\frac{K}{p^e-1}+td+1$.
\end{lemma}
\begin{proof}
Note that $R$ has a set of generators over $R^{p^{ne}-1}$ each of gauge $\leq p^{ne}-1$ (the images of the relevant monomials of $k[x_1,\ldots,x_n]$ in $R$ will do fine). Next, it is easy to check that $\ba^{\lceil t(p^{ne}-1) \rceil}$ is generated by element with gauge $\leq tdp^{ne}+1$. Hence, as in the proof of \autoref{lem.GaugeBoundForImage} the ideal $(\sC_+^{\ba^t})^n$ is generated as a left $R$-modules by elements $\psi$ of the form $\psi = \phi^{l}\cdot b \cdot a$ where $l \geq n$ and $b$ (resp. $a$) is one of the just described generators of $R$ over $R^{p^{ln}}$ (resp. of $\ba^{\lceil t(p^{ne}-1) \rceil}$). Ranging over all such $\psi$ and a set of $R$-generators $m$ of $N$ we see that $(\sC^{\ba^t}_+)^n(N)$ is generated by elements of the form $\psi(m)$. Now we just compute
\begin{align*}
    \delta(\psi(m))=\delta(\phi^l\cdot b \cdot a \cdot m) &\leq \frac{\delta(bam)}{p^{le}} + \frac{K}{p^e-1} \leq \frac{\delta(m)}{p^{le}} +\frac{(p^{le}-1)+(tdp^{le}+1)}{p^{le}} + \frac{K}{p^e-1} \\ &\leq \frac{\delta(m)}{p^{ne}-1} +\frac{K}{p^e-1} + td + 1
\end{align*}
This shows the claim.
\end{proof}
\begin{corollary}
With notation as in the Lemma, let $N$ be an $R$-submodule of $M$ such that $\sC^{\ba^t}_+(N)=N$. Then $N$ is generated by elements of gauge $\leq \frac{K}{p^e-1}+td+1$.

For $T \geq 0$ there no infinite chains of $R$-submodules $N$ of $M$ for which $\sC^{\ba^t}_+(N)=N$ for some $t < T$.
\end{corollary}
\begin{proof}
Clearly, $\sC^{\ba^t}_+(N)=N$ implies that $(\sC^{\ba^t}_+)^n(N)=N$ for all $n$ and hence the first claim follows from the preceding Lemma. The second claim follows from the first one since each such $N$ is generated by elements in the finite dimensional vector space $M_{\leq \frac{K}{p^e-1}+Td+1}$, hence there cannot be infinite chains.
\end{proof}
Now, the discreteness of the jumping numbers for the test ideal is an immediate consequence.
\begin{theorem}
For $(R, \Delta, \ba^t)$ as above, the $F$-jumping numbers form a discrete subset of $\mathbb{Q}$.
\end{theorem}
\begin{proof}
In the case that $(p^e-1)(K_X + \Delta)$ is Cartier, the Cartier algebra $\sC^{\Delta}$ is of the form considered above. Since each test ideals $\tau(R,\Delta,\ba^t)$ has the properties $\tau(R,\Delta,\ba^t) \supseteq \tau(R,\Delta,\ba^{t'})$ for $t' \geq t$ and $\sC^{\Delta,\ba^t}_+ \tau(R,\Delta,\ba^t) = \tau(R,\Delta,\ba^t)$, The preceding corollary shows that there are only finitely for $t$ below a fixed bound $T$. Hence the jumping numbers must be discrete. The general case is similar or can be reduced to this case by using the methods of \autoref{subsec.FiniteMaps}, see \cite{SchwedeTuckerTestIdealFiniteMaps,SchwedeTuckerZhangTestIdealsSingleAlteration}.
\end{proof}

\subsection{Exercises}
\begin{exercise}\label{ex.GaugeOnCoeff}
Let $f \in S=k[x_1,\ldots,x_n]$ with $k$ perfect and with gauge $\delta$ corresponding to the generator $1$. Show that, if $\delta(f) \leq d$ and writing uniquely
\[
    f = \sum_{x^i \in S_{p^e-1}} s_i^{p^e} x^i
\]
one has $\delta(s_i) \leq \lfloor d/p^e \rfloor$. (Here we used multi-exponent notation $x^i$ as shorthand for $x_1^{i_1}\cdots x_n^{i_n}$)
\end{exercise}

\begin{exercise}\label{ex.FrobeniusCommutesMatlis}
Use the duality for finite morphisms to prove \autoref{lem.FrobeniusCommutesMatlis}.
\end{exercise}

\begin{exercise}
\label{ex.PropertiesOfGauge}
Prove \autoref{lem.PropertiesOfGauge}.
\end{exercise}

\begin{exercise}
\label{ex.ShrieckCommutesTensorRegular}
Let $R \into S$ be a module-finite and flat ring extension. Show that the natural map
\[
    \Hom_R(S,R) \tensor_R M \to[\phi \tensor n \mapsto (r \mapsto \phi(r)n)] \Hom_R(S,M)
\]
is an isomorphism. Derive from this the statement of \autoref{lem.ShreickCommutesTensor}.
\end{exercise}

\begin{exercise}
Consider the example of a Cartier structure $\kappa$ on the polynomial ring $k[x]$ given by sending $1 \mapsto x^t$ and $x,x^2,\ldots,x^{p-1} \mapsto 0$. Show that $\delta(\kappa(f)) \leq \delta(f)/p + t$ where $\delta$ is the gauge on $k[x]$ induced by the generator $1 \in k[x]$.
\end{exercise}

\begin{starexercise}
\label{ex.ChainOfAlgebraSubmodulesStabilizes}
Prove \autoref{thm.ChainOfAlgebraSubmodulesStabilizes} by using the same strategy as in \autoref{prop.CartierImagesStabilize}.
\end{starexercise}

\begin{exercise}\label{ex.FrobOnLocalCohom}
Let $(R,\bm)$ be complete local of dimension $d$ and denote by $F \colon H^d_\bm(R) \to H^d_\bm(R)$ the natural Frobenius action. Show that any left action $\phi$ on $H^d_\bm(R)$  of the Frobenius is of the form $\phi=r \cdot F$ for some $r \in R$.
\end{exercise}

\begin{exercise}
\label{ex.MUnderline}
Prove \autoref{cor.MUnderline}.
\end{exercise}

\begin{exercise}
\label{ex.CartierAlgebraOfDelta}
With notation as in \autoref{def.CartierAlgebraFromTriple}, show that $\sC^{\Delta}$ and $\sC^{\Delta, \ba^t}$ are Cartier subalgebras of $\sC^{R}$.  For a proof, see \cite[Remark 3.10]{SchwedeTestIdealsInNonQGor}.
\end{exercise}

\begin{exercise}
Suppose that $R$ is a normal local domain and that $\Delta \geq 0$ is a $\bQ$-divisor on $X = \Spec R$ such that $K_X + \Delta$ is $\bQ$-Cartier with index not divisible by $p > 0$.  Prove that $\sC^{\Delta}$ is a finitely generated ring over $\sC_0^{\Delta} = R$.
\vskip 3pt
\emph{Hint: } Show that $\Hom_R(F^e_* R(\lceil (p^e -1)\Delta \rceil), R) \cong F^e_* R$ for some $e > 0$ and then use \autoref{ex.CompositionOfGeneratingMaps}.  For additional discussion see \cite[Section 4]{SchwedeTestIdealsInNonQGor}.
\end{exercise}

\begin{exercise}
\label{ex.TauLocalizes}
Suppose that $R$ is a normal domain, $W \subseteq R$ is a multiplicative system, $\Delta \geq 0$ is a $\bQ$-divisor on $X = \Spec R$, $\ba \subseteq R$ is a nonzero ideal and $t \geq 0$ is a real number. Set $U = \Spec (W^{-1} R) \subseteq \Spec R = X$.  Prove that
\[
W^{-1} \tau(R; \Delta, \ba^t) = \tau(W^{-1}R; \Delta|_U, (W^{-1}\ba)^t).
\]
\end{exercise}

\begin{starexercise}
\label{ex.TauDoesntChangeFromAbove}
Prove \autoref{lem.TauDoesntChangeFromAbove}.
\vskip 3pt
\emph{Hint: } Use the description of $\tau(R; \Delta, \ba^t)$ from the proof of \autoref{prop.TauExistsForRings}.  Also use the fact that $R$ is Noetherian to see that the sum from \autoref{prop.TauExistsForRings} is a finite sum ($e = 0$ to $m$).  Now notice that if $c$ works in that sum, then so does $bc$ where $0 \neq b \in \ba$.  Set $\epsilon = {1 \over p^m}$.
\end{starexercise}

\begin{starexercise}
Suppose that $R$ is a normal ring and that $X = \Spec R$.  Consider the anticanonical ring
\[
K := \bigoplus_{n \geq 0} \O_X(-nK_X).
\]
Set $K_{F} := \bigoplus_{e \geq 0} \O_X( (1-p^e)K_X)$ to be the summand of $K$ made up of terms of degree $p^e - 1$ for some $e \geq 0$.  This is not a subring of $K$.  However, define a non-commutative multiplication on $K_F$ as follows.  If $\alpha \in \O_X( (1-p^e)K_X)$ and $\beta \in \O_X( (1-p^d)K_X)$ then define $\alpha \star \beta = \alpha^{p^d} \beta \in \O_X( ((1-p^e)p^d + p^e)K_X) = \O_X( (1-p^{e+d}) K_X)$.

With this ring operation, prove that $K_F$ is isomorphic to $\sC^R$.
\end{starexercise}

\appendix \section{Reflexification of sheaves and Weil divisors}\label{sec.reflex}

In this section, we briefly recall basic properties of reflexive sheaves and Weil divisors on normal varieties.  This material is all ``well known'' but there isn't a good source for it in the literature (we note that it is certainly assumed in \cite{KollarMori}).  We note that substantial generalizations of all this material (and complete proofs) can be found in \cite{HartshorneGeneralizedDivisorsOnGorensteinSchemes}.
As before, all schemes are of finite type over a field (or localizations or completions of such schemes).  We assume the reader is familiar with the basic notion of depth and $\textnormal{S}_n$ (Serre's $n$th condition) and the connections with local cohomology / cohomology with support.  See for example \cite[Chapter III, Exercises in Section 3]{Hartshorne}, \cite{BrunsHerzog} or \cite{HartshorneLocalCohomology}.

\subsection{Reflexive sheaves}

Given a coherent sheaf $\sF$ on any scheme $X$, there is the following (dualizing) operation:  $\sF^{\vee} = \sHom_{\O_X}(\sF, \O_X)$.  Furthermore there is a natural map from $\sF$ to the double-dual, $\sF \rightarrow (\sF^{\vee})^{\vee}$.

\begin{definition}
If this map is an isomorphism, we say that \emph{$\sF$ is reflexive} (or more specifically that it is $\O_X$-reflexive).
\end{definition}

Note that if a sheaf is reflexive, it is also coherent (by definition).  If $X = \Spec R$ and $M$ is a coherent $R$-module, we say that $M$ is reflexive if the corresponding sheaf is reflexive (equivalently, if $M \to \Hom_R(\Hom_R(M, R), R)$ is an isomorphism).

Notice first that any locally free sheaf is reflexive.  But there are other reflexive sheaves as well.  If one is careful, one can check that $\langle x, z\rangle \subseteq k[x,y,z]/(xy - z^2)$ corresponds to a reflexive ideal sheaf after taking Spec, \autoref{ex.ReflexiveSheafExample}.  There are a few basic facts about reflexive sheaves that should be mentioned.  We now limit ourselves to varieties (\ie integral schemes) which makes dealing with torsion much easier.  One can do analogues of the following in more general situations  (say for reduced schemes), but the statements become much more involved.

\begin{lemma}
\label{LemmaFDualIsTorsionFree}
Suppose that $X$ is a variety and suppose that $\sF$ is a coherent sheaf on $X$.  Then $\sF^{\vee}$ is torsion-free.  (That is, if $U \subset X$ is open and $0 \neq r \in \O_X(U)$ and $0 \neq z \in \sF^{\vee}(U)$, then $rz \neq 0$).  In particular, a reflexive sheaf is torsion-free.
\end{lemma}

Note that a torsion-free sheaf is necessarily $\textnormal{S}_1$ (any nonzero element makes up a rather short regular sequence).

\begin{lemma}
\label{LemmaTorsionFreeMeansNaturalMapInjects}
Suppose that $X$ is a variety and that $\sF$ is a torsion-free coherent sheaf.  Then the natural map $\alpha : \sF \rightarrow \sF^{\vee \vee}$ is injective.
\end{lemma}

\begin{lemma} \cite[Proposition 1.1]{HartshorneStableReflexiveSheaves}
A coherent sheaf $\sF$ on a quasi-projective variety $X$ is reflexive if and only if it can be included in an exact sequence
\[
0 \rightarrow \sF \rightarrow \sE \rightarrow \sG \rightarrow 0
\]
where $\sE$ is locally free and $\sG$ is torsion-free.
\end{lemma}

We note that the $\O_X$-dual of any coherent sheaf is always reflexive.

\begin{theorem}
If $\sF$ is a coherent sheaf on a variety $X$, then $\sF^{\vee}$ is reflexive.  More generally, if $\sF$ is coherent and $\sG$ is reflexive, then $\sHom_{\O_X}(\sF, \sG)$ is reflexive.
\end{theorem}

We now come to a very useful criterion for checking whether a sheaf is reflexive.

\begin{theorem} \cite[Theorem 1.9]{HartshorneGeneralizedDivisorsOnGorensteinSchemes}
Suppose that $X$ is a normal (not necessarily quasi-projective) variety and that $\sF$ is a coherent sheaf on $X$ such that $\Supp(\sF) = X$.   Then $\sF$ is $\textnormal{S}_2$  if and only if $\sF$ is reflexive.
\end{theorem}

The key reason why the previous criterion is so useful is the Hartog's phenomenon associated with $\textnormal{S}_2$ sheaves.

\begin{corollary}
Let $X$ be a integral, normal (not necessarily quasi-projective) variety and suppose that $\sF$ is a reflexive sheaf on $X$ (defined as above).  Let $Y \subset X$ be a closed subset of codimension $\geq 2$ and set $U = X \backslash Y$.  Then if $i : U \rightarrow X$ is the natural inclusion, then the natural map $\sF \rightarrow i_* \sF|_U$ is an isomorphism.
\end{corollary}

\begin{corollary}
\label{cor.ExtensionOfReflexiveIsReflexive}
 Suppose that $\sF$ is a reflexive sheaf on $U \subseteq X$ (where $X$ is as above) such that $X- U$ is codimension two.  Let us denote by $i : U \rightarrow X$ the inclusion.  Then $i_* \sF$ is a reflexive sheaf on $X$.
\end{corollary}

\subsection{Divisors}

Let $X$ be a normal variety of finite type over a field.  By a \emph{Weil divisor} on $X$, we mean a formal sum of integral codimension $1$ subschemes (prime divisors).  Recall that a divisor $D$ is called \emph{effective} if the coefficients of $D$ are nonnegative.  Just like in the regular case, each prime divisor $D$ corresponds to some discrete valuation $v_D$ of the fraction field of $X$ (although the reverse direction is not true).  %This is because the stalk at the generic point of a prime divisor is still a regular ring (since normal rings are regular in codimension $1$).

\begin{definition}
Choose $f \in \sK(X)$, $f \neq 0$.  We define the \emph{principal divisor} $\divisor(f)$ as in the regular case:  $\divisor(f) = \Sigma_i v_{D_i}(f) D_i$.   Likewise, we say that two Weil divisors $D_1$ and $D_2$ are \emph{linearly equivalent}, if $D_1 - D_2$ is principal.
\end{definition}

\begin{definition}
Given a divisor $D$, we define $\O_X(D)$ be the sheaf associated to the following rule:
\[
\Gamma(V, \O_X(D)) = \{ f \in \sK(X)\; |\; \divisor(f)|_V + D |_V \geq 0 \}
\]
A divisor $D$ is called \emph{Cartier} if $\O_X(D)$ is an invertible sheaf.  It is called \emph{$\bQ$-Cartier} if $nD$ is Cartier for some $n > 0$.
\end{definition}

Note that $D$ is effective if and only if $\O_X(D) \supseteq \O_X$.

\begin{proposition}
Suppose that $D$ is a prime divisor, then $\O_X(-D) = \sI_D$, the ideal sheaf defining $D$. Furthermore, if $D$ is any divisor, then $\O_X(D)$ is reflexive.
\end{proposition}
\begin{proof}
We first show the equality.  The object defined above is clearly a sheaf.  We will prove the equality of the sheaves in the setting where $U$ is affine.  Then $\Gamma(U, \O_X(D))$ is just the functions in $\O_X$ which vanish to order at least $1$ along $D$, in other words the ideal of $D$.

We now want to show that this sheaf is reflexive (or equivalently, that it is $\textnormal{S}_2$).  First notice that clearly if $U$ is the regular locus of $X$, then $\Gamma(V \cap U, \O_X(D)) \cong \Gamma(V, \O_X(D))$ for any open set $V$.  This is because $V \cap U = U \backslash \{\text{non-regular locus}\}$, the non-regular locus is codimension 2, and the sections of $\O_X(D)$ obviously do not change when removing a codimension 2 subset.  This implies that the natural map $\O_X(D) \rightarrow i_* \O_X(D)|_U$ is an isomorphism, but then we notice that $\O_X(D)|_U$ is reflexive (since it is invertible) and thus, by corollary \autoref{cor.ExtensionOfReflexiveIsReflexive}, $\O_X(D)$ is also reflexive.
\end{proof}

We now list some basic properties of rank-1 reflexive sheaves which completely link their behavior to divisors.

\begin{proposition}
\label{prop.DivisorReflexiveSheafCorrespondence}
Suppose that $X$ is a normal variety.  Then:
\begin{itemize}
\item[(a)]  If $X$ is regular, then every reflexive rank-1 sheaf $\sF$ on $X$ is invertible.  \cite[Proposition 1.9]{HartshorneStableReflexiveSheaves}
\item[(b)]  Every rank one reflexive sheaf $\sF$ on a normal scheme $X$ embeds as a subsheaf of $\sK(X)$.
\item[(c)]  Any reflexive rank 1 subsheaf of $\sK(X)$ is $\O_X(D)$ for some (uniquely determined) divisor $D$.
\end{itemize}
\end{proposition}
\begin{proof}
Left to the reader in \autoref{ex.DivisorReflexiveSheafCorrespondence}.
\end{proof}
The addition operations for divisors translates into the tensor of the associated sheaves, up to reflexification.

\begin{proposition}
\label{prop.SumOfDivisorsReflexification}
Suppose that $X$ is a normal variety and $D$ and $E$ are divisors on $X$.  Then
\begin{itemize}
\item[(a)]  If $E$ is Cartier, then $\O_X(D) \tensor \O_X(E) \cong \O_X(D+E)$
\item[(b)]  In general, $\O_X(D+E) \cong (\O_X(D) \tensor \O_X(E))^{\vee \vee}$
\item[(c)]  $\O_X(-D) = \sHom_{\O_X}(\O_X(D), \O_X) = \O_X(-D)^{\vee}$
\end{itemize}
\end{proposition}
\begin{proof}
Left to the reader, see \autoref{ex.SumOfDivisorsReflexification}
\end{proof}

Finally, we mention a result relating sections and linearly equivalent divisors, which will be a key part of this paper.

\begin{theorem}
\label{thm.BijectionBetweenSectionsAndDivisors}
Suppose that $X$ is a normal variety and $D$ is a Weil divisor on $X$.  Then there is a bijection between the following two sets
\[
\left\{ \begin{array}{c} \text{Effective divisors $E$} \\ \text{linearly equivalent to $D$} \end{array} \right\} \longleftrightarrow \left\{ \begin{array}{c} \text{Nonzero sections $\gamma \in H^0(X, \O_X(D))$}\\ \text{modulo equivalence} \end{array}\right\}
\]
where we define $\gamma$ and $\gamma'$ in $H^0(X, \O_X(D))$ to be equivalent if there exists a unit $u \in H^0(X, \O_X)$ such that $u\gamma = \gamma'$.
\end{theorem}
\begin{proof}
Set $\sM = \O_X(D)$.
The choice $\gamma$ induces an embedding $i_{\gamma} : \sM \hookrightarrow \sK(X)$ which sends $\gamma$ to $1$.  Thus $\gamma$ induces a divisor via \autoref{prop.DivisorReflexiveSheafCorrespondence}.  It follows from the same argument that $\gamma$ and $\gamma'$ induce the same divisor if and only if $i_{\gamma}$ and $i_{\gamma'}$ have the same image in $\sK(X)$.  But this happens if and only if $\gamma$ and $\gamma'$ are unit multiplies of one another.
\end{proof}

\subsection{Exercises}

\begin{exercise}
\label{ex.ReflexiveSheafExample}
Show that $\langle x, z\rangle \in k[x,y,z]/\langle xy - z^2 \rangle$ corresponds to a reflexive ideal sheaf after taking Spec.
\end{exercise}

\begin{exercise}
 Which of the following $k[x,y]=R$-modules are reflexive?  If a module is not reflexive, compute its double dual $M^{\vee \vee}$.
\begin{itemize}
\item[(a)]  The ideal $\langle x \rangle$.
\item[(b)]  The ideal $\langle x, y\rangle$.
\item[(c)]  The module $R/\langle x,y\rangle$.
\item[(d)]  The module $R/\langle x\rangle $.
\item[(e)]  The ideal $\langle x^2, xy\rangle = \langle x,y \rangle^2 \cap \langle y\rangle$.
\end{itemize}
\end{exercise}

\begin{exercise}
Suppose that $\pi : Y \to X$ is a finite dominant map of normal varieties and $\sF$ is a coherent sheaf on $Y$.  Then $\sF$ is reflexive on $Y$ if and only if $\pi_* \sF$ is reflexive on $X$. \vskip 3pt
\emph{Hint: } Use the fact that you can check whether a sheaf is reflexive by checking whether it is $\textnormal{S}_2$.  Then use the criterion for checking depth via local cohomology.
\end{exercise}

\begin{exercise}
\label{ex.DivisorReflexiveSheafCorrespondence}
Prove \autoref{prop.DivisorReflexiveSheafCorrespondence}.
\end{exercise}

\begin{exercise}
\label{ex.SumOfDivisorsReflexification}
Prove \autoref{prop.SumOfDivisorsReflexification}.
\end{exercise}

\bibliographystyle{skalpha}
\bibliography{CommonBib}

\end{document}